\documentclass[a4paper,10pt]{article}
\usepackage[left=3cm,right=3cm,top=3cm,bottom=3.5cm]{geometry}
\usepackage{epsfig}
\usepackage{wrapfig}
\usepackage{color}
\usepackage{amsmath,amssymb,amsthm,color}
\theoremstyle{plain}
\newtheorem{theorem}{Theorem}%[chapter]% reset theorem numbering for each chapter

\theoremstyle{definition}
\newtheorem{lemma}[theorem]{Lemma}% definition numbers are dependent on theorem numbers
\newtheorem{remark}[theorem]{Remark}% definition numbers are dependent on theorem numbers
\newtheorem{proposition}[theorem]{Proposition}% definition numbers are dependent on theorem numbers

\def\eps{\varepsilon}
\newcommand{\kk}[1]{\relax}
\newcommand{\nus}{\sigma}
\newcommand{\GO}{G}
\newcommand{\AO}{\alpha}
\newcommand{\EC}{\epsilon_\textrm{J}}
\newcommand{\SO}{s}
\newcommand{\HC}{\textrm{h}}
\newcommand{\cP}{\mathcal P}
\newcommand{\RO}{r}
\newcommand{\RP}{\rho}
\newcommand{\UNP}{_{0}}
\newcommand{\CT}{c}
\newcommand{\EXP}{\textrm{e}}
\newcommand{\RPETBP}{RPETBP}
\def\textb{\relax}
\def\textr{\relax}

\begin{document}
\title{Global instability in the \textb{restricted planar} elliptic three body
problem\footnote{AD and TMS were partially supported by the Spanish MINECO-FEDER Grant MTM2015-65715-P,
the Catalan Grant 2017SGR1049 and the Russian Scientific Foundation Grant 14-41-00044.
VK was partially supported by the DMS-NSF grant 1402164 and the Simons Fellowship.}}
\author{Amadeu Delshams$^1$, Vadim Kaloshin$^2$, Abraham de la Rosa$^3$\\
        and Tere M. Seara$^4$\\
\begin{tabular}{ll}
$^1$&Department of Mathematics and Laboratory of Geometry and Dynamical\\
&Systems, Universitat Polit\`ecnica de Catalunya, \texttt{Amadeu.Delshams@upc.edu}\\
$^2$&University of Maryland at College Park, \texttt{vadim.kaloshin@gmail.com}\\
$^3$&Universitat Polit\`ecnica de Catalunya\footnote{Now at GeoNumerics, S.L.}, \texttt{abraham.delarosa@gmail.com}\\
$^4$&Department of Mathematics, Universitat Polit\`ecnica de Catalunya and\\
&BGSMath, \texttt{Tere.M-Seara@upc.edu}
\end{tabular}}
\maketitle
\begin{abstract}
The   restricted planar elliptic three body problem (\RPETBP) describes the motion of a massless
particle (a comet or an asteroid) under the gravitational field of two massive bodies (the primaries, say
the Sun and Jupiter) revolving around their center of mass on elliptic orbits with some positive eccentricity.
The aim of this paper is to show the existence of orbits whose angular momentum performs arbitrary
excursions in a large region. In particular, there exist diffusive orbits, that is, with a large variation of
angular momentum.

The leading idea of the proof consists in analyzing parabolic motions of the comet.
By a well-known result of McGehee, the union of future (resp. past) parabolic orbits is
an analytic manifold $\cP^+$ (resp. $\cP^-$). In a properly chosen coordinate
system these manifolds are stable (resp. unstable) manifolds
of a manifold at infinity $\cP_\infty$, which we call the manifold at parabolic infinity.

%In other words, studying the so-called \textb{manifold at %parabolic infinity}

On $\cP_\infty$ it is possible to define two scattering maps, which contain the
map structure of the homoclinic trajectories to it, i.e. orbits parabolic both in the future
and the past. Since the inner dynamics inside \textb{$\cP_\infty$} is trivial, two different
scattering maps are used. The combination of these two scattering maps permits
the design of the desired diffusive pseudo-orbits.
Using shadowing techniques and these pseudo orbits we show
the existence of true trajectories of the {\RPETBP} whose angular
momentum varies in any predetermined fashion.
%}

\par\vspace{12pt}
\noindent\emph{2000 Mathematics Subject Classification}:
Primary 37J40, 70F15.

%\par\vspace{12pt}
\noindent\emph{Keywords}:
  Elliptic Restricted Three Body problem,
% \textb{
parabolic motions,
 Manifold of parabolic motions at infinity,
 %}
  Arnold diffusion,
  splitting of separatrices,
  Melnikov integral.
\end{abstract}

\section{Main result and methodology}

The   restricted planar elliptic three body problem (\RPETBP) describes
the motion $q$ of a massless particle (a \emph{comet}) under
the gravitational field of two massive bodies (the \emph{primaries},
say the \emph{Sun} and \emph{Jupiter}) with mass ratio $\mu$ revolving
around their center of mass on elliptic orbits with eccentricity $\EC $.
In this paper we search for trajectories of motion which show a large
variation of the angular momentum $G=q\times \dot{q}$.
In other words, we search for global instability (``diffusion'' is
the term usually used) in the angular momentum of this problem.
Notice that for $\mu=0$ the angular momentum is a first integral.

If the eccentricity \textb{of Jupiter} vanishes,
the primaries revolve along circular orbits, and such diffusion
is not possible, since the restricted (planar)  circular
three body problem (RPCTBP)
%\footnote{If the problem is not resticted we can't say that this problem is circular.
%Therefore, circular should come after restricted %}
is governed by an autonomous
Hamiltonian with two degrees of freedom. This is not the case
for the \textb{\RPETBP}, which is a 2+1/2 degrees-of-freedom
Hamiltonian system with time-periodic Hamiltonian.
Our main result is the following

%\begin{theorem}\label{MainResult}
%There exist two constants $C>0$, $c>0$ and $\mu^*=\mu^*(C,c)>0$ such that for
%any $0<\EC <c/C$ and $0<\mu<\mu^*$, and for any two values of the angular momentum in the region
%$C\leq G_1^*<G_2^*\leq c/\EC $, there exists a trajectory of the \textb{\RPETBP} such that
%$G(0)<G^*_1$, $G(T)>G^*_2$ for some $T>0$.
%\end{theorem}

\begin{theorem}\label{MainResult}
There exist two constants $C>0$, $c>0$ such that for
any $0<\EC <c/C$ there is $\mu^*=\mu^*(C,c,\EC)>0$\footnote{
the upper bound on $\mu^*$ can be improved in the sense that for
$\EC\le c/G_2^*$ we can choose $\mu^*=\mu^*(C,c,c/G_2^*)$}
such that for any $0<\mu<\mu^*$
and any $C\leq G_1^*<G_2^*\leq c/\EC $
there exists a trajectory of the \textb{\RPETBP} such that
$G(0)<G^*_1$, $G(T)>G^*_2$
for some $T>0$.
\end{theorem}

This result will be a consequence of Theorem~\ref{Thm:Main},
%, where the large $C$ and the small constant $c$ are explicitly computed ($C=32$, $c=1/8$), and
where it is also shown the existence
of trajectories of motion such that their angular momentum performs
arbitrary excursions along the region
$C\leq G_1^*<G_2^*\leq c/\EC $. Comments about the values $C$ and $c$ can be found in Remark \ref{rem:Cc}.
\smallskip

\subsection{Previous works}
Let us recall related results about oscillatory motions and diffusion for the RPCTBP or the
{\RPETBP}.
They hold close to a region when there is some kind of hyperbolicity in the Three Body Problem,
like the Euler libration points~\cite{LlibreMS85,CapinskiZ11,DelshamsGR13,DelshamsGR16}, collisions~\cite{Bolotin06}, the \textb{parabolic} infinity
~\cite{GoK, LlibreS80,Xia93, Xia92, Moeckel07,Mos,MartinezP94, MartinezS14}
or near mean motion resonances~\cite{FejozGKR14}.
Aubry-Mather theory was used to study oscillatory
motions and instabilities not close to parabolic motions
\cite{GaK}.\smallskip

Among these papers, two were very influential for our
computations: the first one is \cite{LlibreS80}, where the method of steepest descent
was used along special complex paths to compute several integrals, and the second is
\cite{MartinezP94}, whhere asymptotic formulae for a scattering map on
the infinity manifold for large values of $\EC  \GO $ is computed. We also believe \cite{GuardiaMS12}
to be very important in the future, since the proof of transversal manifolds of the infinity
manifold is established for the RPCTBP for any $\mu\in(0,1/2]$.\smallskip

It is worth to mention the paper \cite{Bolotin06} where the existence of trajectories with
diffusion of  $G$ was proven assuming small $0<\eps$ and $0<\mu\ll\eps$.
Diffusive trajectories in \cite{Bolotin06} are of a very different nature:
$G$ travels in a bounded interval while trajectories come close to collisions.
In the present paper  $G$ is very large, and trajectories come near infinity.
However the idea of the proof in \cite{Bolotin06} is similar: after regularization of collisions
there appears a normally hyperbolic symplectic invariant manifold $M$ with trivial inner dynamics
and it is possible to define several scattering maps which give rise to diffusive trajectories.
\smallskip

\subsection{Comments on the proof: a parabolic infinity and scattering maps}
Concerning the proof of our main result, let us first notice that, for a non-zero mass parameter
small enough ($0<\mu \ll 1/2$) and zero eccentricity ($\EC=0$), the RPCTBP is non-integrable.  Although
for  large $G$ it is very close to integrable, since  its chaotic zones have a size which is exponentially
small for large $G$, more precisely, of size $O( \exp(-G^3/3))$\,(see ~\cite{LlibreS80,GuardiaMS12}).
This phenomenon adds the first difficulty in proving the global instability of
the angular momentum $G$ in the \textb{\RPETBP} for large values of $G$.\smallskip

The framework for proving our result consists in considering the motion close to
the parabolic orbits of the Kepler problem that takes place when the mass
parameter $\mu$ is zero. To this end we study the \emph{manifold at parabolic infinity},
which turns out to be an invariant object \emph{topologically equivalent to
a normally hyperbolic invariant manifold} (TNHIM), in the sense that it is an invariant manifold of fixed points which, even if it is not normally hyperbolic,
it  has stable and unstable manifolds which consist of the union of the stable and unstable manifolds of its fixed points as proved in \cite{GuardiaMSS17}.
\smallskip

More concretely, recall that a motion of the comet $q(t)$ is called
{\it future (resp. past) parabolic} if $\lim_{t\to +\infty} |q(t)|=\infty$
and $\lim_{t\to +\infty} \dot q(t)=0$ (resp. $t\to +\infty$ is
replaced by $t\to -\infty$).
For the RPCTBP McGehee \cite{McGehee73} proved
(see \cite{GuardiaMSS17} for the {\RPETBP}) that the set of future (resp. past) parabolic motions, denoted
$\cP^+_\mu$ (resp. $\cP^-_\mu$), is an immersed analytic
manifold.
The intersection $\cP^+_\mu\cap \cP^-_\mu$
consists of orbits both future and  past  parabolic.
For $\mu=0$ we have that $\cP^+_0=\cP^-_0$ and
they correspond to parabolic motions of the Kepler problem (between the Sun and the comet).
These manifolds are stable and unstable manifolds of
the manifold at the parabolic infinity denoted $\cP_\infty$.
The infinity manifold  $\cP_\infty$ is independent of $\mu$ and turns out to
be topologically equivalent to a normally hyperbolic
invariant manifold (TNHIM).\smallskip
%}}

%\textb{

%}

On this TNHIM $\cP^\infty$, it is possible to define two \emph{scattering maps}~\cite{DelshamsLS00,DelshamsLS08}, which contain
the map structure of the homoclinic  trajectories to $\cP^\infty$. A non-canonical symplectic
structure still persists close to $\cP^\infty$ and extends naturally to
a $b^3$-symplectic structure in the sense of \cite{Scott13,KiesenhoferMS15}). Therefore,
on $\cP^\infty$, it is possible to define a symplectic scattering map, which contains
the map structure of the homoclinic trajectories to the TNHIM. Unfortunately, the inner
dynamics within $\cP^\infty$ is trivial, so it cannot be combined with the scattering map
to produce pseudo-orbits adequate for diffusion, and adds a second difficulty.
Because of this, in this paper we introduce the use of \emph{two} different scattering
maps whose combination produces the desired diffusive pseudo-orbits. It is worth remarking
that this strategy of combining several scattering maps to get diffusing orbits have been
already applied to several problems~\cite{Bolotin06,DelshamsGR16,DelshamsS17,DelshamsS17a}.
%which eventually give rise to true trajectories of the system
%with the help of  the shadowing results given in
%\cite{GuardiaMS15}.
Using the results in \cite{GuardiaMSS17} we prove
the existence of orbits of the system shadowing diffusive
pseudo-orbits.

The main issue of computing the two scattering maps consists in evaluating
the \emph{Melnikov potential}~\eqref{cL} associated to the TNHIM $\cP^\infty$.
The main difficulty comes from the fact that its size is exponentially small for
a large angular momentum $G$, so it is necessary to perform very accurate
estimates for its Fourier coefficients. Such computations are performed in
\textb{Section \ref{EstMeliPot}}, see  Theorem~\ref{thepropositionmain},
and they involve a careful treatment of several Fourier expansions, as well as
the computation of several integrals using the method of steepest descent
along adequate complex paths, playing both with the eccentric and
the true anomaly. To guarantee the convergence of the Fourier series,
we have to assume that $\GO $ is large enough ($\GO \geq C$, $C=32$),
and $\EC $ small enough ($\GO  \EC \leq c$, $c=1/8$). For a larger value of
$C$ and a  smaller value of $c$, one can ensure
%Under these two assumptions,
that the dominant part of the Melnikov potential consists on four
harmonics, from which it is possible to compute the existence
of two functionally independent scattering maps (see Remark \ref{rem:Cc})
which are globally defined in the manifold of parabolic infinity $\cP_\infty$.

The combination of these two scattering maps permits the design of the desired diffusive pseudo-orbits,
under the assumption of a mass \textb{$\mu$  very} small compared to the eccentricity ($0<\mu<\mu^*$,
see~\eqref{Boundmu}).
Shadowing these pseudo-orbits by
true trajectories of the system is done using the results of
\cite{GuardiaMSS17}.

It is worth noticing that since {\it all the diffusive trajectories}
found in this paper shadow ellipses close to parabolas of
the Kepler problem, that is, with a very large semi-major axis,
their energy is close to zero. The orientation of their
semi-major axis \textb{(precession)}  changes only slightly
at each revolution.

%\kk{aclarir}

\subsection{Other parabolic regimes}
The case of arbitrary eccentricity $0<\EC <1$ and arbitrary  mass parameter
$0<\mu<1$ remains open in this paper. \textb{As it turns out}, the case
$\EC  \GO  \approx 1$ involves the analysis of an infinite number of dominant
Fourier coefficients of the Melnikov potential, whereas for the case $\EC  G > 1$,
the qualitative properties of the Melnikov function should be known without
using its Fourier expansion. Larger values of the mass parameter $\mu$
than those considered in this paper involve improving the estimates of
the error terms of the splitting of separatrices in complex domains, as is
usual when the splitting of separatrices is exponentially small. The computation
of the explicit trajectories from the pseudo-orbits found in this paper needs
a suitable shadowing result given in \cite{GuardiaMSS17}, which  involves
the translation to TNHIM of the usual shadowing techniques for NHIM.

\subsection{Plan of the paper}
The plan of this paper is as follows. In Section~\ref{Sec:Setting} we introduce the equations of the \RPETBP,
as well as the McGehee coordinates to be used to study the motion close to infinity.
In Section~\ref{sectionmu0} we recall the geometry of the Kepler problem, i.e. when $\mu=0$,
%the mass parameter vanishes,
close to the \emph{parabolic infinity manifold} and its associated separatrix. Next, in
Section~\ref{Sec:ERTBP}, we study the transversal intersection of the invariant manifolds
for the \textb{\RPETBP}, as well as the \emph{scattering map} associated, which depend on
the \emph{Melnikov potential} of the problem, whose detailed computation is deferred to
Section~\ref{EstMeliPot}. The global instability is proven in Section~\ref{sece0G0s}, using
the computation of the Melnikov potential, and is based on the computation of \emph{two different}
scattering maps, whose combination gives rise to a heteroclinic chain of periodic
orbits with increasing angular momentum and, finally to
\textb{trajectories with diffusing} angular momentum.

\section{Setting of the problem}\label{Sec:Setting}

Fix a coordinate reference system with the origin at the center of mass and call $q_\textrm{S}$ and $q_\textrm{J}$ the position of the primaries,
then under the classical assumptions regarding time units, distance and masses normalization, the motion $q$ of a massless
particle under Newton's law of universal gravitation is given by
\begin{equation}
\label{eqn:ERTBP}
 \frac{d^2 q}{dt^2}=(1-\mu)\frac{q_\textrm{S}-q}{|q_\textrm{S}-q|^3}
 +\mu\frac{q_\textrm{J}-q}{|q_\textrm{J}-q|^3}
\end{equation}
where $1-\mu$ is the mass of the particle at $q_\textrm{S}$ and $\mu$ the mass of the particle at $q_\textrm{J}$.
Introducing the conjugate momentum $p=dq/dt$ and the self-potential function % (see \cite[p.~28]{MR2468466})
\begin{equation}\label{eqn:SelfPotential}
 U_{\mu}(q,t;\EC )=\dfrac{1-\mu}{|q- q_\textrm{S}|}+ \dfrac{\mu}{|q-q_\textrm{J}|},
\end{equation}
\textb{then the }equation~\eqref{eqn:ERTBP} can be
rewritten as a 2+1/2 degree-of-freedom Hamiltonian system with time-periodic Hamiltonian
\begin{equation}\label{firstH}
H_\mu(q,p,t;\EC )=\dfrac{p^2}{2}-U_{\mu}(q,t;\EC ).
\end{equation}

In the (planar) \textb{\RPETBP}, the two primaries are assumed to be revolving around their center of mass on elliptic orbits
with eccentricity $\EC $, unaffected by the motion of the comet $q$.
In polar coordinates $q=\RP (\cos \alpha, \sin \alpha)$, the equations of motion of the primaries are
\begin{equation}\label{pri:q}
 q_\textrm{S}=\phantom{-}\mu \RO  (\cos f,\sin f) %\label{pri:q1}
\qquad
 q_\textrm{J}=-(1-\mu)\RO (\cos f,\sin f). %\label{pri:q2}.
\end{equation}
By the first Kepler's law the distance $\RO $ between the primaries~ \cite[p.~195]{Wintner41}
can be written as a function $\RO =\RO (f,\EC )$
\begin{equation}
\RO =\frac{1-\EC ^2}{1+\EC  \cos f}\label{r0f}
\end{equation}
where $f=f(t,\EC )$ is the so called \textit{true anomaly}, which satisfies~\cite[p.~203]{Wintner41}
\begin{equation}\label{f}
\frac{df}{dt}=\frac{(1+\EC \cos f)^{2}}{(1-\EC ^{2})^{3/2}}.
\end{equation}

Taking into account the expression~\eqref{pri:q} for the motion of the primaries, we can
write explicitly the denominators of the  self-potential function~\eqref{eqn:SelfPotential}
\begin{equation}\label{eqn:primaries}
\begin{split}
  |q-q_\textrm{S}|^2\  &=\RP ^2- 2\mu \RO \RP \cos (\alpha - f) +\mu^2 \RO ^2,
\\
|q- q_\textrm{J}|^2\ &= \RP ^2+2(1-\mu)  \RO \RP \cos (\alpha - f) +(1-\mu)^2 \RO ^2.
\end{split}
\end{equation}
We now perform a standard polar-canonical change of variables
$(q,p) \longmapsto (\RP ,\alpha,P_\RP , P_{\alpha})$
\[
 q=(\RP \cos \alpha,\RP \sin \alpha),
\quad
 p=\left(P_\RP \cos \alpha -\frac{P_{\alpha}}{\RP}\sin \alpha ,P_\RP \sin \alpha +\frac{P_{\alpha}}{\RP}\cos \alpha\right)
\]
to Hamiltonian \eqref{firstH}.
The equations of motion in the new coordinates are the associated to the Hamiltonian
\begin{equation}\label{hampolar}
 H_\mu^{*}(\RP ,\alpha,P_\RP ,P_{\alpha},t;\EC )=\frac{P_\RP ^{2}}{2}+\frac{P_{\alpha}^{2}}{2\RP ^{2}}-U_\mu^{*}(\RP ,\alpha,t;\EC )
\end{equation}
with a self-potential $U_\mu^*$
\[
U_\mu^{*}(\RP ,\alpha,t;\EC )=U_\mu(\RP \cos \alpha,\RP \sin \alpha,t;\EC ).
\]
From now on we will write
$$
G=P_\alpha,\qquad y=P_\RP ,
$$
so that Hamiltonian  \eqref{hampolar} becomes
\begin{equation}\label{eqn:Ham*}
H_\mu^{*}(\RP ,\alpha,y,G,t;\EC )=\frac{y^2}{2}+\frac{G^{2}}{2\RP ^{2}}-U_\mu^{*}(\RP ,\alpha,t;\EC ).
\end{equation}

\begin{remark}\label{rmk:RTBP}
In the (planar) circular case $\EC =0$ (RTBP), it is clear from equations \eqref{r0f} and \eqref{f}
that $\RO =1$ and $f=t$, and that the expressions for the distances~\eqref{eqn:primaries}
between the primaries depend on the time $t$ and the angle $\alpha$ just through
their difference $\alpha-t$.
%\kk{$\alpha-t$\\sinodic\\angle? \textb{Yes, it is a sinodic angle}}
As a consequence, $U_\mu^*(\RP ,\alpha,t;0)$, as well as $H_\mu^{*}(\RP ,\alpha,y,G,t;0)$,
depend also on  $t$ and $\alpha$ just through the same difference $\alpha-t$, called the sinodic angle.
This implies that the Jacobi constant $H^*+G$ is a first integral of the system.
\end{remark}

\subsection{McGehee coordinates}
To study the behavior of orbits near infinity, we make the McGehee~\cite{McGehee73}
non-canonical change of variables
\begin{equation}\label{x2}
\RP =\frac{2}{x^2}
\end{equation}
for $x>0$. This brings the infinity $\RP =\infty$ to the origin $x=0$ (and extends naturally to
a $b^3$-symplectic structure in the sense of \cite{Scott13,KiesenhoferMS15}; other related
examples can be found in \cite{DelshamsKM17,BraddellDMOP17}).

In these McGehee coordinates, the equations associated to
the Hamiltonian~\eqref{hampolar} become
\begin{equation}\label{Mc1:e}
\begin{aligned}
\frac{d x}{d t}&=-\frac{1}{4}x^{3}y \mspace{140mu}&\frac{d y}{d t}&=\frac{1}{8}G^{2}x^{6}-\frac{x^{3}}{4}\frac{\partial \mathcal{U}_\mu}{\partial x}\\
\frac{d \alpha}{d t}&=\phantom{-}\frac{1}{4} x^{4}G &\frac{d G}{d t}&=\frac{\partial \mathcal{U}_\mu}{\partial \alpha},
\end{aligned}
\end{equation}
where the self-potential $\mathcal{U}_\mu$ is given by
\begin{equation}\label{eq:potencial}
\mathcal{U}_\mu(x,\alpha,t;\EC )=U_\mu^{*}(2/x^2,\alpha,t;\EC )=\frac{x^2}{2}\left(\frac{1-\mu}{\sigma_S}+\frac{\mu}{\sigma_J}\right)
\end{equation}
with
\begin{align*}
 |q-q_\textrm{S}|^{2}=\sigma_\textrm{S}^{2}&=1-\mu \RO  x^{2}\cos(\alpha - f)+\frac{1}{4} \mu^{2} \RO ^{2}x^{4},\\
 |q-q_\textrm{J}|^{2}=\sigma_\textrm{J}^{2}&=1+(1-\mu) \RO  x^{2}\cos(\alpha - f)+\frac{1}{4}(1-\mu)^{2} \RO ^{2}x^{4}.
\end{align*}
It is important to notice that the true anomaly $f$ is present in these equations, so that the equation for $f$
given in \eqref{f} should be added to have the complete description of the dynamics.

\subsubsection{The symplectic structure}%\label{quasihamiltoniansection}
Under McGehee change of variables~\eqref{x2}, the canonical form $d\RP \wedge dy + d\alpha\wedge dG$ is transformed to
\begin{equation*}%\label{eqn:SymplecticForm}
\omega= -\frac{4}{x^3} dx\wedge dy + d\alpha\wedge dG
\end{equation*}
which, on $x>0$, is a (non-canonical) symplectic form.
Therefore, expressing the Hamiltonian~\eqref{eqn:Ham*} in the McGehee coordinates
\begin{equation}\label{cuasih}
 \mathcal{H}_\mu(x,\alpha,y,G,t;\EC )=\frac{y^{2}}{2}+\frac{x^{4}G^{2}}{8}-\mathcal{U}_\mu(x,\alpha,t;\EC ),
\end{equation}
the equations \eqref{Mc1:e} can be written as
\begin{equation}\label{McH:e}
\begin{aligned}
\frac{d x}{d t}&=-\frac{x^{3}}{4}\, \left(\frac{\partial \mathcal{H}_\mu}{\partial y}\right)\mspace{140mu}
&\frac{d y}{d t}&=-\frac{x^{3}}{4} \left(-\frac{\partial \mathcal{H}_\mu}{\partial x}\right) \\%\label{McH:e1}\\
\frac{d \alpha}{d t}&=\phantom{-}\frac{\partial \mathcal{H}_\mu}{\partial G}
&\frac{d G}{d t}&= -\frac{\partial \mathcal{H}_\mu}{\partial \alpha}.%\label{McH:e2}
\end{aligned}
\end{equation}

Equivalently, we can write the equations \eqref{McH:e} as $dz/dt=\{z,\mathcal{H}_\mu\}$ in terms of the Poisson bracket
\begin{equation*}
\{ f,g \}
= -\frac{x^3}{4}\Bigg(\frac{\partial f}{\partial x}\frac{\partial g}{\partial y}-\frac{\partial f}{\partial y}\frac{\partial g}{\partial x}\Bigg)
+\frac{\partial f}{\partial \alpha}\frac{\partial g}{\partial G}-\frac{\partial f}{\partial G}\frac{\partial g}{\partial \alpha}.
%\label{poisipara}
\end{equation*}

\section{Geometry of the Kepler problem ($\mu=0$)}\label{sectionmu0}
\subsection{The manifold at parabolic infinity}

For $\mu=0$ and $G> 0$, the Hamiltonian~\eqref{cuasih} becomes Duffing Hamiltonian (see Figure~\ref{levelH0}):
\begin{equation*}%\label{H0}
\mathcal{H}_{0}(x,y,G)=\frac{y^{2}}{2}+\frac{x^{4}G^{2}}{8}-\mathcal{U}\UNP (x)
=\frac{y^{2}}{2}+\frac{x^{4}G^{2}}{8}-\frac{x^{2}}{2}
\end{equation*}
and is a first integral, since the system is autonomous.
Moreover, $\mathcal{H}_{0}$ is also independent of $\EC $
and $\alpha$. Its associated equations are
\begin{equation}\label{Mc20:e}
\begin{aligned}
\frac{d x}{d t}&=-\frac{1}{4}x^{3}y\mspace{144mu}
&\frac{d y}{d t}&=\frac{1}{8}G^{2}x^{6}-\frac{1}{4}x^{4}\\% \label{Mc20:e1}\\
\frac{d \alpha}{d t}&=\phantom{-}\frac{1}{4} x^{4}G    &
\frac{d G}{d t}&=0%\label{Mc20:e2}
\end{aligned}
\end{equation}
\begin{wrapfigure}[24]{L}{45mm}
\includegraphics[width=43mm]{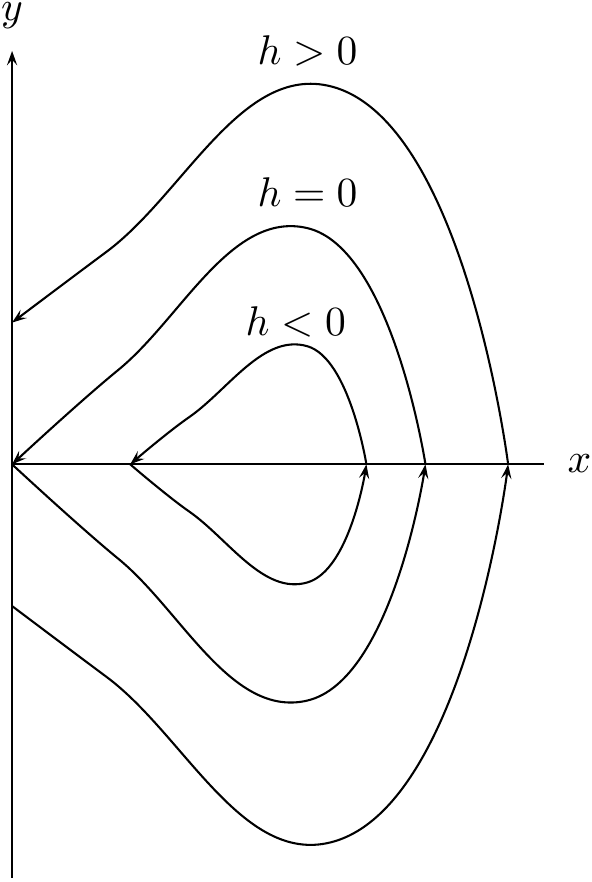}
\caption{Level curves of $\mathcal{H}\UNP $ in the $(x\geq0,y)$ plane, for fixed $G>0$}
\label{levelH0}
\end{wrapfigure}
%\kk{Figura 5D? \textb{What does it mean?}}
where it is clear that $G$ is a conserved quantity, which
will be restricted to the case $G>0$ from now on, that is, $G\in\mathbb{R}_{+}$.
The phase space, including
the invariant locus $x=0$ is given by
$(x,\alpha,y,G)\in \mathbb R_{\ge 0}\times \mathbb{T}\times \mathbb{R}\times\mathbb{R}_{+}$.
From  equations \eqref{Mc20:e} it is clear that
\begin{equation*}%\label{equi}
\mathcal{E}_\infty=\{z= (x=0,\alpha,y,G)\in \mathbb R_{\ge 0}\times \mathbb{T}\times \mathbb{R}\times\mathbb{R}_{+} \}
\end{equation*}
is the set of equilibrium points of system~\eqref{Mc20:e}.
Moreover, for any fixed  $\AO \in \mathbb{T},\GO \in\mathbb{R}$,
$$
\Lambda_{\AO ,\GO }=\{(0,\AO ,0,\GO )\}
$$
is a  parabolic equilibrium point, which is topologically equivalent to a saddle point,
since it possesses stable and unstable 1-dimensional invariant manifolds.
The union of such points is the 2-dimensional manifold of
equilibrium points
$$
\Lambda_{\infty}=\bigcup _{\AO ,\GO } \Lambda _{\AO ,\GO },
$$
which was previously denoted as $\cP_\infty$.

As we will deal with a time-periodic Hamiltonian, it is natural to work in the extended phase space
$$
\tilde z=(z,s)=(x,\alpha,y,G,s)\in \mathbb R_{\ge 0}\times \mathbb{T}\times \mathbb{R}\times\mathbb{R}_{+}\times\mathbb{T}
$$
just by writing $s$ instead of $t$ in the Hamiltonian and adding the equation
\[
\frac{d s}{d t}=1
\]
to systems \eqref{McH:e} and \eqref{Mc20:e}.
We write now the extended version of the invariant sets we have defined so far. For any
$\AO \in \mathbb{T},\GO \in\mathbb{R}$, the set
\begin{equation*}%\label{periodicorbit}
\tilde \Lambda_{\AO ,\GO }=\{\tilde z=(0,\AO ,0,\GO ,\SO ),\SO \in \mathbb{T}\}
\end{equation*}
is a $2\pi$-periodic orbit with motion determined by $ds/dt = 1$.
The union of such periodic orbits is
the $3$-dimensional invariant manifold (the \textb{parabolic} \emph{infinity manifold})
\begin{equation}\label{Lambdainfty}
\tilde\Lambda_{\infty}=\bigcup _{\AO ,\GO } \tilde\Lambda _{\AO ,\GO }=
 \{(0,\AO ,0,\GO ,\SO ),\, (\AO ,\GO ,\SO )\in \mathbb{T}\times \mathbb{R}_+\times \mathbb{T}\}
\simeq\mathbb{T} \times \mathbb{R}_+\times\mathbb{T},
\end{equation}
which is \emph{topologically equivalent to a normally hyperbolic invariant manifold} (TNHIM).

Parameterizing the points in $\tilde\Lambda_{\infty}$ by
\[
\tilde{\mathbf{x}}\UNP = \tilde{\mathbf{x}}\UNP (\AO ,\GO ,\SO )=
(\mathbf{x}\UNP (\AO ,\GO ),\SO )=(0,\AO ,0,\GO ,\SO )
\in \tilde\Lambda_{\infty}\simeq\mathbb{T} \times \mathbb{R}_+\times\mathbb{T}
\]
the inner dynamics on $\tilde\Lambda_\infty$
is trivial, since it is given by the dynamics on each
periodic orbit $\tilde \Lambda_{\AO ,\GO }$:
\begin{equation}\label{orbper}
\tilde \phi_{t,0}(\tilde{\mathbf{x}}\UNP )=(0,\AO ,0,\GO ,\SO +t)=
(\mathbf{x}\UNP (\AO ,\GO ),\SO +t)=\tilde{\mathbf{x}}\UNP (\AO ,\GO ,\SO +t),
\end{equation}
where we denote by $\tilde \phi _{t,\mu}$ the flow of system \eqref{McH:e} in the extended phase space.

\subsection{The scattering map}
In the region of the phase space with positive angular momentum $G$,
let us now look at the homoclinic orbits to the previously introduced invariant objects.

The equilibrium points $\Lambda_{\AO ,\GO }$ have stable and unstable 1-dimensional invariant manifolds
which coincide:
\begin{eqnarray*}
\gamma _{\AO ,\GO }&=&W^{\textrm{u}}(\Lambda_{\AO ,\GO })=W^{\textrm{s}}(\Lambda_{\AO ,\GO })\\
&=& \biggl\{z=(x,\hat\alpha,y,\GO ), \  \mathcal{H}\UNP (x,y,\GO )=0, \hat\alpha=\AO -\GO \int_{\mathcal{H}\UNP =0} \frac{x}{y}dx\biggl\},
\end{eqnarray*}
whereas the 2-dimensional manifold of equilibrium points $\Lambda_\infty$ has stable and unstable 3-dimensional invariant manifolds
which coincide and are given by
\[
\gamma=W^{\textrm{u}}(\Lambda_{\infty})=W^{\textrm{s}}(\Lambda_{\infty})= \{z=(x,\alpha,y,G), \  \mathcal{H}\UNP (x,y,G)=0\}.
\]
The surface
\begin{align}
\tilde\gamma _{\AO ,\GO }&=W^{\textrm{u}}(\tilde\Lambda_{\AO ,\GO })=W^{\textrm{s}}(\tilde\Lambda_{\AO ,\GO })\nonumber\\
&=\biggl\{\tilde z=(x,\hat\alpha,y,\GO ,\SO ),
 \SO \in \mathbb{T},\  \mathcal{H}\UNP (x,y,\GO )=0, \hat\alpha=\AO -\GO \int_{\mathcal{H}\UNP =0} \frac{x}{y}dx\biggl\}
\label{eqn:2Dhomoclinic}
\end{align}
is a 2-dimensional homoclinic manifold to the periodic orbit $\tilde \Lambda_{\AO ,\GO }$ in the extended phase space.
The 4-dimensional stable and unstable manifolds of the infinity manifold
$\tilde\Lambda_{\infty}$ coincide along the 4-dimensional homoclinic invariant manifold (the \emph{separatrix}), which is just
the union of the homoclinic surfaces $\tilde\gamma _{\AO ,\GO }$:
\begin{eqnarray*}
\tilde \gamma &=&W^{\textrm{u}}(\tilde\Lambda_{\infty})=W^{\textrm{s}}(\tilde\Lambda_{\infty})=\bigcup _{\AO ,\GO }\tilde\gamma _{\AO ,\GO }\notag\\
&=& \{\tilde z= (x,\alpha,y,G,s), (\alpha,G,s)\in \mathbb{T}\times \mathbb{R}_+\times \mathbb{T},\  \mathcal{H}\UNP (x,\alpha,y,G)=0\}
%\label{gammatilde}
\end{eqnarray*}

Due to the presence of the factor $-x^{3}/4$ in front of equations~\eqref{Mc20:e},
it is more convenient to parameterize the separatrix $\tilde \gamma _{\AO ,\GO }$, given in \eqref{eqn:2Dhomoclinic},
by the solutions of the Hamiltonian flow contained in $\mathcal{H}\UNP =0$ in
some time $\tau$ satisfying (see \cite{MartinezP94})
\begin{equation}\label{ttau1}
\frac{dt}{d\tau}=\frac{2G}{x^2}.
\end{equation}
In this way, the homoclinic solution to the periodic orbit $\tilde \Lambda_{\AO ,\GO }$
of system~\eqref{Mc20:e} can be written as
\begin{subequations}\label{homoclinic:h}
\begin{align}
 x_{\HC}(t;\GO )&=\frac{2}{\GO (1+\tau^{2})^{1/2}}\label{homoclinic:h1}\\
 \alpha_{\HC}(t;\AO ,\GO )&=\AO +\pi+ 2\arctan\tau\label{homoclinic:h2}\\
 y_{\HC}(t;\GO )&=\frac{2\tau}{\GO (1+\tau^{2})}\notag\\%\label{homoclinic:h3}\\
 G_{\HC}(t;\GO )&=\GO \notag\\%\label{homoclinic:h4}\\
 s_{\HC}(t;\SO )&=\SO +t,\notag%\label{homoclinic:h5}
\end{align}
\end{subequations}
where $\AO $ and $\GO $ are free parameters and the relation between $t$ and $\tau$ is
\begin{equation}\label{cambiotau}
t=\frac{\GO ^3}{2}\Bigl(\tau+\frac{\tau^3}{3}\Bigl),
\end{equation}
which is equivalent to \eqref{ttau1} on $\mathcal{H}\UNP $.
From the expressions above, we see that the convergence along the separatrix to the infinity manifold
is power-like in $\tau$ and $t$:
\begin{equation}\label{eqn:power-likeAsymp}
x_{\HC}, y_{\HC}, \frac{\AO -\alpha_{\HC}\textb{+}\pi}{\GO }
 \sim \frac{2}{\GO  \tau} \sim \frac{2}{\sqrt[3]{\textb{\pm}6t}}, \quad \tau, t \to\pm\infty.
\end{equation}

We now introduce the notation
\begin{equation}
\mathbf{\tilde z}_{0}=\mathbf{\tilde z}\UNP (\nus,\AO ,\GO ,\SO )=(\mathbf{z}\UNP (\nus,\AO ,\GO ),\SO )
=( x_{\HC}(\nus;\GO ),\alpha_{\HC}(\nus;\AO ,\GO ), y_{\HC}(\nus;\GO ), \GO , \SO )\in \tilde\gamma \label{z0tilde}
\end{equation}
so that we can parameterize any surface $\tilde \gamma _{\AO ,\GO }$ as
$$
\tilde \gamma _{\AO ,\GO }=\{\mathbf{\tilde z}_{0}=\mathbf{\tilde z}\UNP (\nus,\AO ,\GO ,\SO )=(\mathbf{z}\UNP (\nus,\AO ,\GO ),\SO ),
\  \nus \in \mathbb{R}, \SO \in \mathbb{T}\} .
$$
and we can parameterize the 4-dimensional separatrix as
\begin{equation*}%\label{paramgammatilde}
\tilde \gamma =W(\tilde\Lambda_{\infty})=\{
\mathbf{\tilde z}_{0}=\mathbf{\tilde z}\UNP (\nus,\AO ,\GO ,\SO )=(\mathbf{z}\UNP (\nus,\AO ,\GO ),\SO ),
\nus \in\mathbb{R}, \GO \in\mathbb{R}_+, (\AO ,\SO )\in\mathbb{T}^2\}.
\end{equation*}
The motion on $\tilde \gamma$ is given by
\begin{equation}\label{eqn:z0t}
\tilde \phi_{t,0}(\mathbf{\tilde z}_{0})=
\mathbf{\tilde z}\UNP (\nus+t,\AO ,\GO ,\SO +t )=
(\mathbf{z}\UNP (\nus+t,\AO ,\GO ),\SO +t)
\end{equation}
and by equations \eqref{homoclinic:h}, \eqref{cambiotau} % and \eqref{homoclinics},
the following asymptotic formula follows:
%\kk{write asymptotic}
\begin{equation}
\label{eqn:Asymptotics}
 \tilde \phi_{t,0}(\mathbf{\tilde z}_{0})-\tilde \phi_{t,0}(\mathbf{\tilde x}\UNP  )=
 (\mathbf{z}\UNP (\nus+t,\AO ,\GO ),\SO +t)-(\mathbf{x}\UNP (\AO ,\GO ),\SO +t)
 \xrightarrow[t \to \pm \infty]{} 0 .
\end{equation}

The \emph{scattering map} $\widetilde S$ describes the homoclinic orbits to the infinity manifold
$\tilde \Lambda_{\infty}$ (defined in \eqref{Lambdainfty}) to itself.
Given $\mathbf{\tilde x}_-,\mathbf{\tilde x}_+\in\tilde \Lambda_{\infty}$,
we define
$$
\widetilde S_\mu(\mathbf{\tilde x}_-):=\mathbf{\tilde x}_+
$$
if there exists
$\mathbf{\tilde z}^*\in W_{\mu}^{\textrm{u}}(\tilde \Lambda_{\infty})\cap W_{\mu}^{s}(\tilde \Lambda_{\infty})$
such that
\[
\tilde  \phi_{t,\mu}(\mathbf{\tilde z}^*)-\tilde \phi_{t,\mu}(\mathbf{\tilde x}_\pm)\to 0 \quad \text{for $t\to \pm \infty$} .
\]
In the case $\mu=0$ the asymptotic relation~\eqref{eqn:Asymptotics}
implies $\widetilde S_{0}(\mathbf{\tilde x}_{0})=\mathbf{\tilde x}_{0}$ so that that the scattering map
$\widetilde S\UNP :\tilde \Lambda_{\infty}\longrightarrow \tilde \Lambda_{\infty}$ is the identity.

\section{Invariant manifolds for the \textb{\RPETBP} ($\mu > 0$)}\label{Sec:ERTBP}

\subsection{The parabolic infinity manifold}
In order to analyse the structure of system \eqref{McH:e}, we will write $\mathcal{H}_\mu$ given in \eqref{cuasih} as
\begin{equation}\label{Hmumel}
\mathcal{H}_\mu(x,\alpha,y,G,s;\EC )=\mathcal{H}\UNP (x,y,G)-\mu\Delta \mathcal{U}_\mu(x,\alpha,s;\EC )
\end{equation}
where we have written $\mathcal{U}_\mu$ in \eqref{eq:potencial} as
\begin{equation*}%\label{calUcalUstar}
\mathcal{U}_\mu(x,\alpha,s;\EC )=\mathcal{U}\UNP (x)+\mu\Delta \mathcal{U}_\mu(x,\alpha,s;\EC )
=\frac{x^2}{2}+\mu\Delta\mathcal{U}_\mu(x,\alpha,s;\EC ),
\end{equation*}
and we proceed to study the dynamics as a perturbation of the limit case $\mu=0$. From \eqref{eq:potencial},
\begin{align}
\Delta\mathcal{U}\UNP (x,\alpha,s;\EC )&=\lim_{\mu\to 0} \Delta\mathcal{U}_\mu(x,\alpha,s;\EC )\notag\\
&=\frac{x^{2}}{\bigl[4+x^{4}\RO ^{2}+4x^{2}\RO \cos (\alpha- f)\bigl]^{1/2}}
-\Bigl(\frac{x^{2}}{2} \Bigl)^{2} \RO \cos (\alpha -f)-\frac{x^{2}}{2}\Biggl.\label{Ucal0star}
\end{align}
where
$\RO =\RO (f,\EC )$ and $f=f(s,\EC )$ are given, respectively, in (\ref{r0f}-\ref{f}).

For $\mu>0$, it is clear from equations~\eqref{McH:e} that
the set $\mathcal{E}_\infty$ remains invariant and, therefore, so does the infinity manifold
$\tilde \Lambda _{\infty}$,
being again a TNHIM, and all the periodic orbits $\tilde\Lambda_{\AO ,\GO }$ also persist.
The inner dynamics on $\tilde\Lambda_{\infty}$ is the same as in the case $\mu=0$,
so that the parametrization $\tilde{\mathbf{x}}\UNP $ as well as its trivial dynamics~\eqref{orbper} remain the same.

\subsection{The scattering map}
From \cite{McGehee73, GuardiaMSS17} we know that $W_\mu^{\textrm{s}}(\tilde\Lambda_{\infty})$ and $W_\mu^{\textrm{u}}(\tilde\Lambda_{\infty})$
exist for $\mu$ small enough and are 4-dimensional in the extended phase space.
The existence of a scattering map will depend on the
transversal intersection between these two manifolds.

Let us take an arbitrary $\mathbf{\tilde z}_{0}=(\mathbf{z}_{0},\SO )=(\mathbf{z}\UNP (\nus,\AO ,\GO ),\SO )\in \tilde\gamma$ as in \eqref{z0tilde}.
Now, we have to construct points in $W_\mu^{\textrm{s}}(\tilde \Lambda_{\infty})$ and $W_\mu^{\textrm{u}}(\tilde \Lambda_{\infty})$ to measure the distance
between them. It is clear from the definition of $\tilde \gamma$ that
$$
\mathbf{\tilde v}=(\nabla \mathcal{H}_{0}(\mathbf{z}\UNP ),0)
$$
is orthogonal to $\tilde \gamma = W^{\textrm{u}}(\tilde \Lambda_\infty)=W^{\textrm{s}}(\tilde \Lambda_\infty)$
at $\mathbf{\tilde z}_{0}$ and then if the normal bundle to $\tilde \gamma$ is denoted by
$$
N(\mathbf{\tilde z}\UNP )=\left\{\mathbf{\tilde z}\UNP +\lambda\, \mathbf{\tilde v}, \lambda\in \mathbb{R}\right\}
$$
we have that, if $\mu$ is small enough,  there exist unique points
$\mathbf{\tilde z}^{\textrm{s},\textrm{u}}_{\mu}=( z^{\textrm{s},\textrm{u}}_{\mu},\SO )$ such that
\begin{equation}\label{zmusu}
\{\mathbf{\tilde z}^{\textrm{s},\textrm{u}}_{\mu}\}
=W^{\textrm{s},\textrm{u}}_{\mu}(\tilde\Lambda_{\infty})\,\cap N(\mathbf{\tilde z}\UNP ) .
\end{equation}
The distance we want to compute between
$W_\mu^{\textrm{s}}(\tilde \Lambda_{\infty})$ and $W_\mu^{\textrm{u}}(\tilde \Lambda_{\infty})$
is the signed magnitude  given by
\begin{equation*}%\label{dist}
 d(\mathbf{\tilde z}\UNP ,\mu)=
 \mathcal{H}\UNP (\mathbf{\tilde z}^{\textrm{u}}_\mu)-\mathcal{H}\UNP (\mathbf{\tilde z}^{s}_\mu).
\end{equation*}
We now introduce the \emph{Melnikov potential} (see \cite{DelshamsG00,DelshamsLS06})
\begin{equation}\label{calL}
 \mathcal{L}(\AO ,\GO ,\SO ;\EC )=
\int_{-\infty}^{\infty}\Delta\mathcal{U}\UNP (x_{\HC}(t;\GO ),\alpha_{\HC}(t;\AO ,\GO ),\SO +t;\EC )\, dt,
\end{equation}
where $\Delta\mathcal{U}\UNP $ is defined in \eqref{Ucal0star}.
Thanks to the asymptotic behavior~\eqref{eqn:power-likeAsymp}
of the solutions along the separatrix and of the self potential close to the \textb{parabolic} infinity manifold
$$
\Delta\mathcal{U}\UNP (x,\alpha,s;\EC ) =O(x^4) \quad \text{as \ $x\to 0$}
$$
this integral is absolutely convergent, and will be computed in detail in Section~\ref{EstMeliPot}.
\begin{proposition}\label{theprop}
 Given $(\AO ,\GO ,\SO )\in \mathbb{T}\times\mathbb{R}^+\times\mathbb{T}$, assume that the function
\begin{equation}\label{appLcal}
\nus \in \mathbb{R}\longmapsto \mathcal{L}(\AO ,\GO ,\SO -\nus;\EC )\in\mathbb{R}
\end{equation}
has a non-degenerate critical point $\nus^*=\nus^*(\AO ,\GO ,\SO ;\EC )$. Then, there exists $\mu^*=\mu^*(\GO ,\EC)$, such that for $0<\mu<\mu^*$,
close to the point
$\mathbf{\tilde z}\UNP ^{*}= (\mathbf{z}\UNP (\nus^{*},\AO ,\GO ),\SO )\in \tilde\gamma$
(see the parameterization in~\eqref{z0tilde}), there exists a
locally unique point
\begin{equation*}
\mathbf{\tilde z}^{*}=\mathbf{\tilde z}^*(\nus^*,\AO ,\GO ,\SO ;\EC ,\mu)
\in W_{\mu}^{s}(\tilde \Lambda_\infty)\pitchfork W_{\mu}^{\textrm{u}}(\tilde \Lambda_\infty)
\pitchfork N(\mathbf{\tilde z}\UNP ^{*})
\end{equation*}
of the form
$$
\mathbf{\tilde z}^{*}=\mathbf{\tilde z}\UNP ^{*}+O(\mu).
$$
Also, there exist unique points
$\mathbf{\tilde x}_\pm=(0,\alpha_\pm,0,G_\pm,\SO )=(0,\AO ,0,\GO ,\SO )+O(\mu)\in\tilde\Lambda_{\infty}$ such that
\begin{equation*}
\tilde  \phi_{t,\mu}(\mathbf{\tilde z}^{*})-\tilde \phi_{t,\mu}(\mathbf{\tilde x}_\pm) \longrightarrow 0\quad \text{for $t\to\pm\infty$}.
\end{equation*}
Moreover, we have
\begin{equation}\label{G+G-}
G_+-G_-=\mu \frac{\partial \mathcal{L}}{\partial \AO }(\AO ,\GO ,\SO -\nus^*(\AO ,\GO ,\SO ;\EC ))+O(\mu^2).
\end{equation}

\end{proposition}
\begin{proof}
 From the equation \eqref{z0tilde} we know that any point $ \mathbf{\tilde z}\UNP  \in \tilde \gamma$ has the form
$$
 \mathbf{\tilde z}\UNP =\mathbf{\tilde z}\UNP (\nus,\AO ,\GO ,\SO ).
$$
As in \eqref{zmusu}, we consider
\begin{equation*}
\mathbf{\tilde z}^{\textrm{s},\textrm{u}}_{\mu}=
(\mathbf{z}^{\textrm{s},\textrm{u}}_{\mu},\SO )\in
W^{\textrm{s},\textrm{u}}_{\mu}(\tilde\Lambda_{\infty})\,\cap N(\mathbf{\tilde z}\UNP ),
\end{equation*}
and we  look for
$\mathbf{\tilde z}\UNP  $ such that $\mathbf{\tilde z}^{\textrm{s}}_{\mu}=\mathbf{\tilde z}^{\textrm{u}}_{\mu}$.
There must exist points $ \mathbf{\tilde x}_\pm=(\mathbf{x}_\pm,\SO )\in \tilde \Lambda_\infty$ such that
\begin{equation}\label{phisu}
\tilde  \phi_{t,\mu}(\mathbf{\tilde z}^{\textrm{s},\textrm{u}}_\mu)-
 \tilde \phi_{t,\mu}(\mathbf{\tilde x}_\pm) \xrightarrow[t \to \pm \infty]{} 0,
\end{equation}
moreover
$\tilde \phi_{t,\mu}(\mathbf{\tilde z}^{\textrm{s},\textrm{u}}_\mu)-\tilde \phi_{t,0}(\mathbf{\tilde z}\UNP )=O(\mu)$
for $\pm t \geq 0$ (see \cite{McGehee73, GuardiaMSS17}).
Since $\mathcal{H}\UNP $ does not depend on time, by~\eqref{Hmumel}
and the chain rule we have that
$$
 \frac{d}{d t}\mathcal{H}\UNP (\tilde \phi_{t,\mu}(\mathbf{\tilde z}^{\textrm{s},\textrm{u}}_\mu))=
\{ \mathcal{H}\UNP ,\mathcal{H}_\mu \}(\tilde \phi_{t,\mu}(\mathbf{\tilde z}^{\textrm{s},\textrm{u}}_\mu);\EC)
=-\mu\{ \mathcal{H}\UNP ,\Delta\mathcal{U}_\mu\}(\tilde \phi_{t,\mu}(\mathbf{\tilde z}^{\textrm{s},\textrm{u}}_\mu);\EC ).
$$
Since $\mathcal{H}\UNP =0$ in $\tilde \Lambda_\infty$, using \eqref{phisu} and the trivial dynamics on
$\tilde \Lambda_\infty$ we obtain
$$
\mathcal{H}\UNP ( \mathbf{\tilde z}^{\textrm{s},\textrm{u}}_\mu)=
-\mu \int _{0}^{\pm \infty}\{ \mathcal{H}\UNP ,\Delta\mathcal{U}_\mu \}
(\tilde \phi_{t,\mu}(\mathbf{\tilde z}^{\textrm{s},\textrm{u}}_\mu);\EC )\,dt.
$$
Taylor expanding in $\mu$ and using the notation~\eqref{eqn:z0t}
\begin{align}
 \mathcal{H}\UNP (\mathbf{\tilde z}^{\textrm{u}}_\mu)-\mathcal{H}\UNP (\mathbf{\tilde z}^{\textrm{s}}_\mu)
&=\mu \int_{-\infty}^{\infty}\{ \mathcal{H}\UNP ,\Delta\mathcal{U}\UNP  \}(\tilde \phi_{t,0}(\mathbf{\tilde z}\UNP );\EC )\, dt
+O(\mu^2)\notag \\
&=\mu \int_{-\infty}^{\infty}\{ \mathcal{H}\UNP ,\Delta\mathcal{U}\UNP  \}(\mathbf{z}\UNP (\nus+t,\AO ,\GO ),\SO +t;\EC )\,
dt+O(\mu^2).
\label{eqn:H0us}
\end{align}
On the other hand, from \eqref{calL}
$$
\mathcal{L}(\AO ,\GO ,\SO ;\EC )
=\int_{-\infty}^{\infty}\Delta\mathcal{U}\UNP (x_{\HC}(\nu -\SO ;\GO ),\alpha_{\HC}(\nu -\SO ;\AO ,\GO ),\nu ;\EC )\, d\nu
$$
and then
\[
\frac{\partial \mathcal{L}}{\partial \SO }(\AO ,\GO ,\SO ;\EC )
=-\int_{-\infty}^{\infty}\{ \Delta\mathcal{U}\UNP ,\mathcal{H}\UNP  \}
(\mathbf{z}\UNP (\nu -\SO ,\AO ,\GO ),\nu ;\EC )\,d\nu
\]
so that
\begin{align}
\frac{\partial \mathcal{L}}{\partial \SO }(\AO ,\GO ,\SO -\nus;\EC )&=
\int_{-\infty}^{\infty}\{ \mathcal{H}\UNP ,\Delta\mathcal{U}\UNP  \}
(\mathbf{z}\UNP (\nu -\SO +\nus,\AO ,\GO ),\nu ;\EC )\,d\nu \notag\\
&=\int_{-\infty}^{\infty}\{ \mathcal{H}\UNP ,\Delta\mathcal{U}\UNP  \}
(\mathbf{z}\UNP (t+\nus,\AO ,\GO ),\SO +t;\EC )\,dt\label{eqn:DsL}
\end{align}
and therefore, from~\eqref{eqn:H0us} and~\eqref{eqn:DsL}
\begin{equation}\label{eq:dist}
d(\mathbf{\tilde z}\UNP ,\mu)=\mathcal{H}\UNP (\mathbf{\tilde z}^{\textrm{u}}_\mu)-
\mathcal{H}\UNP (\mathbf{\tilde z}^{\textrm{s}}_\mu)=
\mu\frac{\partial \mathcal{L}}{\partial \SO }(\AO ,\GO ,\SO -\nus;\EC )+O(\mu^2).
\end{equation}

For \textb{a non-zero} small enough $\mu$, it is clear by the Implicit Function Theorem
that a non degenerate critical value $\nus^*$ of the function~\eqref{appLcal}
gives rise to a homoclinic point $\mathbf{\tilde z}^{*}$ to $\tilde \Lambda _\infty$
where the manifolds $W_\mu^{\textrm{s}}(\tilde\Lambda_{\infty})$
and $W_\mu^{\textrm{u}}(\tilde\Lambda_{\infty})$ intersect transversally
and has the desired form $\mathbf{\tilde z}^{*}=\mathbf{\tilde z}\UNP ^{*}+O(\mu)$.

Consider now the solution of system \eqref{McH:e} in the extended phase space represented by $\tilde \phi_{t,\mu}(\mathbf{\tilde z}^{*})$.
By the Fundamental Theorem of Calculus and~\eqref{Hmumel} we have
\begin{align*}
G_+-G_-&=
-\int_{-\infty}^{\infty} \frac{\partial \mathcal{H}_\mu}{\partial \alpha}(\tilde \phi_{t,\mu}(\mathbf{\tilde z}^{*}))\, dt
=\mu
\int_{-\infty}^{\infty} \frac{\partial \Delta\mathcal{U}_\mu}{\partial \alpha}(\tilde \phi_{t,\mu}(\mathbf{\tilde z}^{*});\EC )\, dt\\
&=\mu\int_{-\infty}^{\infty} \frac{\partial \Delta\mathcal{U}\UNP }{\partial \alpha}(\tilde \phi_{t,0}(\mathbf{\tilde z}\UNP ^{*});\EC )\, dt+O(\mu^2)\\
&=\mu\int_{-\infty}^{\infty} \frac{\partial \Delta\mathcal{U}\UNP }{\partial \alpha}(\mathbf{z}\UNP (\nus^*+t,\AO ,\GO ),\SO +t;\EC )\, dt+O(\mu^2)\\
&=\mu\frac{\partial \mathcal{L}}{\partial \alpha}(\AO ,\GO ,\SO -\nus^*;\EC )+O(\mu^2).
\end{align*}
\end{proof}
\begin{remark}\label{rem:dist}
From \eqref{eq:dist} it is clear that, to apply the implicit function Theorem, we need that $\mu \ll \mu^*$, where
\[
\mu ^*= O\left(\frac{\partial^2}{\partial \SO^2 }\left(\mathcal{L}\left(\AO ,\GO ,\SO -\nus^*(\AO ,\GO ,\SO ;\EC );\EC \right)\right)\right).
\]
We will give more precise information about $\mu ^*$ after the computation of the Melnikov function given in Theorem \ref{thepropositionmain}, where we will
see that it is exponentially small for large $G$.
 \end{remark}

Once we have found a critical point $\nus^*=\nus^*(\AO ,\GO ,\SO ;\EC )$ of \eqref{appLcal}
on a domain of $(\AO ,\GO ,\SO )$, we can define the \emph{reduced Poincar\'e function} (see \cite{DelshamsLS06})
\begin{equation}\label{Lcalstar}
\mathcal{L}^*(\AO ,\GO ;\EC ):=\mathcal{L}(\AO ,\GO ,\SO -\nus^*;\EC )=\mathcal{L}(\AO ,\GO ,\SO ^*;\EC )
\end{equation}
with $\SO ^*=\SO -\nus^*$.
Note that the reduced Poincar\'e function does not depend on the $\SO $ chosen, since by Proposition~\ref{theprop}
\[
\frac{\partial}{\partial \SO }\left(\mathcal{L}\left(\AO ,\GO ,\SO -\nus^*(\AO ,\GO ,\SO ;\EC );\EC \right)\right)\equiv 0.
\]
Note also that if the function \eqref{appLcal} in Proposition \ref{theprop} has different non degenerate critical points there will exist different scattering maps.

The next Proposition gives an approximation of the scattering map in the general case $\mu> 0$.

\begin{proposition}\label{scatteringlemma}
The associated scattering map $(\alpha_+,G_+,s_+)=\widetilde S_{\mu}(\alpha,G,s)$
for any non degenerate critical point $\nus^*$
of the function defined in \eqref{appLcal} is given by
$$
(\alpha,G,s)\longmapsto \Bigl(\alpha -\mu \frac{\partial \mathcal{L}^*}{\partial G}(\alpha,G;\EC )+O(\mu^2),G
 +\mu \frac{\partial\mathcal L^*}{\partial \alpha}(\alpha,G;\EC )+O(\mu^2), s\Bigl)
$$
where $\mathcal{L}^*$ is the Poincar\'e reduced function introduced in \eqref{Lcalstar}.
\end{proposition}

\begin{proof}
By hypothesis we have a non degenerate critical point $\nus^*$ of \eqref{appLcal}.
By definition~\eqref{Lcalstar}, Proposition~\ref{theprop}
gives
$$
G_+ - G = \mu \frac{\partial \mathcal{L}^*}{\partial \alpha}(\AO ,\GO )+O(\mu^2).
$$
as well as $G_-=\GO +O(\mu)$ to get the correspondence between $G_+$ and $G_-$ that were
looking for.

The companion equation to~\eqref{G+G-}
$$
\alpha_+ - \alpha = - \mu \frac{\partial \mathcal{L}^*}{\partial G}(\AO ,\GO )+O(\mu^2)
$$
follows from the fact that the scattering map is of the form $\widetilde S_{\mu}(\alpha,G,s)= (S_{\mu}(\alpha,G,s), s)$ and,
for each fixed $s\in \mathbb{T}$,  $S_{\mu}$ is symplectic.

Indeed, this is a standard result for a scattering map associated
to a NHIM, and is proven in \cite[Theorem 8]{DelshamsLS08}. For what concerns our scattering map
defined on a TNHIM, the only difference is that the stable contraction (expansion) along
$W^{\textrm{s},\textrm{u}}_\mu(\tilde{\Lambda}_\infty)$ is power-like~\eqref{eqn:power-likeAsymp}
instead of exponential with respect to time.
Therefore, we only have to check that Proposition~10 in\cite{DelshamsLS08}~still holds,
namely that
$\text{Area}\left(\phi_{t,\mu}(\mathcal{R})\right) \rightarrow0$ when $t\rightarrow 0$
for every 2-cell $\mathcal{R}$ in
$W^{\textrm{s}}_\mu(\tilde{\Lambda}_\infty)$ parameterized by $R:[0,1] \times [0,1]\to W^{\textrm{s}}_\mu(\tilde{\Lambda}_\infty)$
in such a way that $R(t_1,t_2) \in W^{\textrm{s}}_\mu(\tilde{\Lambda}_\infty)$, $R(0,t_2) \in\tilde{\Lambda}_\infty$.
But this is a direct consequence of the fact that the stable coordinates contract
at least by $C/\sqrt[3]{t}$ (see~\eqref{eqn:power-likeAsymp}) and the
coordinates along $\Lambda_\infty$ do not expand at all.
\end{proof}

\begin{remark}%\label{Ganulada}
In the (planar) circular case $\EC =0$ (RTBP), $\Delta\mathcal{U}_\mu (x,\AO ,s;\EC )$
depends on the time $s$ and the angle $\alpha$ just through their difference $\alpha-s$, see Remark~\ref{rmk:RTBP}.
From
\[
\frac{\partial\Delta\mathcal{U}_\mu }{\partial \alpha}(x,\alpha,s;0)=-\frac{\partial\Delta\mathcal{U}_\mu}{\partial s}(x,\alpha,s;0)
\]
one readily obtains
\[
\frac{\partial \mathcal{L}}{\partial \SO }(\AO ,\GO ,\SO ;0)=
-\frac{\partial \mathcal{L}}{\partial \AO }(\AO ,\GO ,\SO ;0)
\]
and, therefore,
\[
\frac{\partial \mathcal{L}}{\partial \AO }(\AO ,\GO ,\SO -\nus^*;\EC )=
-\frac{\partial \mathcal{L}}{\partial \SO }(\AO ,\GO ,\SO -\nus^*;0)=0
\]
and consequently the reduced Poincar\'e function $\mathcal{L}^*$ does not depend on $\AO $,
and $G_+=G_-+O(\mu^2)$.

But indeed $G_+\equiv G_-$ in the circular case, since there exists the first integral
provided by the Jacobi constant $C_J=\mathcal{H}_\mu+G$ and as $\mathcal{H}_\mu=0$
on $\tilde \Lambda_\infty$, $G_+=G_-$. Therefore, in the circular case there is no possibility
to find diffusive orbits studying the intersection of $W_{\mu}^{\textrm{s}}(\tilde \Lambda_\infty)$
and $W_{\mu}^{\textrm{u}}(\tilde \Lambda_\infty)$ since any scattering map preserves
the angular momentum.
\end{remark}

\section{Global diffusion in the \textb{\RPETBP}}\label{sece0G0s}
We have already the tools to derive the scattering maps to the infinity manifold $\tilde \Lambda_\infty$, namely,
Proposition~\ref{theprop} to find transversal homoclinic orbits to $\tilde \Lambda_\infty$ and Proposition~\ref{scatteringlemma}
to give their expressions. Both of them rely on computations on the Melnikov potential $\mathcal{L}$.
Inserting in the Melnikov potential introduced in~\eqref{calL} the expression for
$\Delta\mathcal{U}\UNP $ in \eqref{Ucal0star}, we get

\begin{equation}\label{cL}
\mathcal{L}(\AO ,\GO ,\SO ;\EC )=\int_{-\infty}^{\infty}\Biggl[ \frac{x_{\HC}^{2}}{\bigl[4+x_{\HC}^{4}\RO ^{2}+4x_{\HC}^{2}\RO \cos (\alpha_{\HC}- f)\bigl]^{1/2}}
+\Bigl(\frac{x_{\HC}^{2}}{2} \Bigl)^{2} \RO \cos (\alpha_{\HC} -f)
-\frac{x_{\HC}^{2}}{2}\Biggl] \,dt
\end{equation}
where $x_{\HC}$ and $\alpha_{\HC}$, coordinates of the homoclinic orbit defined
in \eqref{homoclinic:h}, are evaluated at $t$, whereas $\RO $ and $f$, defined in
\eqref{r0f} and \eqref{f}, are evaluated at $\SO +t$.

To evaluate the above Melnikov potential, we will compute its Fourier coefficients with
respect to the angular variables $\AO ,\SO $. Since $x_{\HC}$ and $\RO $ are even
functions of $t$ and $f$ and $\alpha_{\HC}$ are odd, $\mathcal L$ is an even function
of the angular variables $\AO ,\SO $:
$\mathcal{L}(-\AO ,\GO ,-\SO ;\EC )=\mathcal{L}(\AO ,\GO ,\SO ;\EC )$,
and therefore $\mathcal L$ has a  Fourier Cosine series with real coefficients $L_{q,k}$:
\begin{equation}\label{eqn:LFourier1}
 \mathcal{L}=\sum_{q\in\mathbb{Z}}\sum_{k\in \mathbb{Z}}L_{q,k}\EXP ^{i(q\SO +k\AO )}
 =L_{0,0}+2\sum_{k\geq1}L_{0,k}\cos k\AO  +2\sum_{q\geq 1}\sum_{k\in \mathbb{Z}}L_{q,k}\cos(q\SO +k\AO ).
\end{equation}
The concrete computation of the Fourier coefficients of the Melnikov potential \eqref{cL}
will be carried out in section~\ref{EstMeliPot}. In that section, the following general bounds will
be obtained in Proposition~\ref{boundLqkbis} and Lemma~\ref{boundL0k}:
\begin{proposition}\label{boundLqk}
Let $\GO \geq 32$, $q\geq 1$, $k\geq 2$ and $\ell\geq 0$.
Then $|L_{q,\ell}|\leq B_{q,\ell}$ and $|L_{0,\ell}|\leq B_{0,\ell}$, where
\begin{equation}\label{eqn:BBounds}
\begin{split}
 B_{q,0}&= 2^{9+q} \,\EXP ^{2q} \EC ^q  \GO ^{-3/2}\EXP ^{-q\GO ^3/3}\\
 B_{q,1}&= 2^{11}  \,\EXP ^{2q}  \frac{\EC ^{q+1}}{\sqrt{1-\EC ^2}} \, \GO ^{-7/2} \EXP ^{-q \GO ^3/3}\\
 B_{q,- 1}&= 2^{9+q}  \,\EXP ^{2q}  \EC ^{|1-q|} \GO ^{-1/2} \EXP ^{-q\GO ^3/3}\\
 B_{q,k}&=  2^{2k+5} \EXP ^{2q} \frac{\EC ^{q+k}}{(\sqrt{1-\EC ^2})^k}\, \GO ^{-2k-1/2}  \EXP ^{-q \GO ^3/3}\\
 B_{q,-k}&= 2^{5+q+2k} \,\EXP ^{2q} \, \EC ^{|k-q|} \GO ^{k-1/2} \EXP ^{-q\GO ^3/3}\\
 B_{0, \ell}&= 2^{8+2\ell}\  \EC ^\ell\, \GO ^{-2\ell-3}
\end{split}
\end{equation}
\end{proposition}

Directly from this Proposition, we first see that the harmonics $L_{q,\ell}$ are exponentially
small for large $G$ and $q\geq 1$, so it will be convenient to split the Fourier expansion~\eqref{eqn:LFourier1} as
\begin{equation}\label{eqn:LFourier2}
 \mathcal{L}=\mathcal{L}_{0}+\mathcal{L}_{1}+\mathcal{L}_{2}+\cdots = \mathcal{L}_{0}+\mathcal{L}_{1}+\mathcal{L}_{\geq2}
\end{equation}
where
\begin{equation}\label{eqn:LFourier3}
\begin{split}
\mathcal{L}_{0}(\AO ,\GO ;\EC )&=L_{0,0}+2\sum_{k\geq1}L_{0,k}\cos k\AO ,\\
\mathcal{L}_{q}(\AO ,\GO ,\SO ;\EC )&=2\sum_{k\in \mathbb{Z}}L_{q,k}\cos(q\SO +k\AO ), \qquad q\geq 1.
\end{split}
\end{equation}
The function $\mathcal{L}_{0}$ does not depend on the angle $\SO $
as it contains the harmonics of $\mathcal{L}$ of order $0$ in $\SO $,
which are of finite order in terms of $\GO $,
$\mathcal{L}_{1}$ the harmonics of first order, which are of order $\EXP ^{-q\GO ^3/3}$,
and all the harmonics of $\mathcal{L}_{q}$ for $q\geq 2$
are much exponentially smaller for large $G$ than those of $\mathcal{L}_{1}$, so we will estimate $\mathcal{L}_{0}$ and $\mathcal{L}_{1}$
and bound $\mathcal{L}_{\geq2}$.

To this end, it will be necessary to sum the series in~\eqref{eqn:LFourier3}.
From the bounds~$B_{q,k}$ in~\eqref{eqn:BBounds} for the harmonics $L_{q,k}$
we get the quotients
\begin{equation}\label{eqn:Bquotients}
\frac{B_{q,k+1}}{B_{q,k}}=\frac{4}{\GO ^2} \frac{\EC }{\sqrt{1-\EC ^2}}\mbox{ for } k\geq 2,\quad
\frac{B_{q,-(k+1)}}{B_{q,-k}}=4\EC  \GO  \mbox{ for } k\geq q,\quad
\frac{B_{0,\ell+1}}{B_{0,\ell}}=\frac{4 \EC }{\GO ^2} \mbox{ for } \ell\geq 0.
\end{equation}
%which indicate that, for fixed $q$.
To guarantee the convergence of the  Fourier series of $\mathcal{L}_q$, we impose
the following conditions
$$
\GO >\sqrt{\frac {4 \EC}{\sqrt{1-\EC ^2 }}}\ \ \text{ and }\ \ \EC  \GO  <1/4.
$$
This is one of the  reasons why we are going to restrict ourselves to the region
$\GO \geq C$ large enough and $\EC  \GO  \leq c$ small  enough along
this paper to get the diffusive orbits.

Among the harmonics \textr{$L_{0,\ell}$ of $0$ order in $\SO $, by~\eqref{eqn:Bquotients},
the harmonic $L_{0,0}$ appears to be the dominant one, but
we will also estimate $L_{0,1}$ to get information about the variable $\AO $,
and bound the rest of harmonics $L_{0,\ell}$ for $\ell\geq 2$}.
Among the harmonics of first order $L_{1,k}$, again by~\eqref{eqn:Bquotients}, the five harmonics $L_{1,k}$
for $|k|\leq 2$ are the only candidates to be the dominant ones,
but the quotients from~\eqref{eqn:BBounds}
\textr{\begin{equation*}%\label{eqn:B1quotients}
%\frac{B_{1,2}}{B_{1,-1}}=\frac{(1+\EC )^2}{8\EXP \GO ^4},\qquad
\frac{B_{1,2}}{B_{1,-1}}=\frac{\EC ^3}{2(1-\EC^2)\GO ^4},\qquad
%\frac{B_{1,1}}{B_{1,-1}}=\frac{(1+\EC )^4}{8\EXP \GO ^3},\qquad
\frac{B_{1,1}}{B_{1,-1}}=\frac{2 \EC ^2}{\sqrt{1-\EC^2}\GO ^3},\qquad
\frac{B_{1,0}}{B_{1,-1}}=\frac{\EC }{\GO }=\frac{\EC \GO }{\GO ^2},
\end{equation*}
}
indicate that $L_{1,-1}$ and $L_{1,-2}$ appear to be the two dominant harmonics of order 1.
\textr{Nevertheless, as we will need to use two different scatterig maps, the coefficient
$L_{1,-3}$ will be necessary to check that both scatterig maps are independent.}
Summarizing, to compute the series~\eqref{eqn:LFourier1} we compute only
the \textr{five harmonics $L_{0,0}$, $L_{0,1}$, $L_{1,-1}$, $L_{1,-2}$ and $L_{1,-3}$}, and bound all the rest,
providing the following result, whose proof will also be carried out in section~\ref{EstMeliPot}.

\begin{theorem}\label{thepropositionmain}
For $\GO \geq 32$, $ \EC  \GO  \leq 1/8$,
the Melnikov potential~\eqref{cL} is given by
\begin{equation}\label{MeliL0L1}
 \mathcal{L}(\AO ,\GO ,\SO ;\EC )=
 \mathcal{L}_0(\AO ,\GO ;\EC )+\mathcal{L}_1(\AO ,\GO ,\SO ;\EC )+\mathcal{L}_{\geq2}(\AO ,\GO ,\SO ;\EC )
\end{equation}
with
\begin{align}
 \mathcal{L}_0(\AO ,\GO ;\EC )&=L_{0,0}+ L_{0,1}\cos\AO  +\mathcal{E}_{0}(\AO ,\GO ;\EC )\label{calL0}\\
 \mathcal{L}_1(\AO ,\GO ,\SO ;\EC )&=2 L_{1,-1}\cos(\SO -\AO )+2 L_{1,-2}\cos(\SO -2\AO )\nonumber \\
&+2 L_{1,-3}\cos(\SO -3\AO ) +\mathcal{E}_{1}(\AO ,\GO ,\SO ;\EC ), \nonumber%\label{calL1}
\end{align}
where the four harmonics above are given by
\begin{align}
L_{0,0}&=L_{0,0}(\GO ;\EC )=\frac{\pi}{2\GO ^{3}}(1+E_{0,0})\label{L00}\\
L_{0,1}&=L_{0,1}(\GO ;\EC )=-\frac{15\pi \EC }{8 \GO ^{5}}(1+E_{0,1})\label{L01}\\
\textr{2\,L_{1,-1}}&=2\,L_{1,-1}(\GO ;\EC )=\sqrt{\frac{\pi}{8\GO }}\EXP ^{-\GO ^3/3}(1+E_{1,-1})\label{L1-1}\\
\textr{2\,L_{1,-2}}&=2\,L_{1,-2}(\GO ;\EC )=-3\sqrt{2\pi}\EC \GO ^{3/2}\EXP ^{-\GO ^3/3}(1+E_{1,-2})\label{L1-2}\\
\textr{2\,L_{1,-3}}&=2\,L_{1,-3}(\GO ;\EC )=\frac{19}{8}\sqrt{2\pi}\EC ^2\GO ^{5/2}\EXP ^{-\GO ^3/3}(1+E_{1,-3})\label{L1-3}
\end{align}
and the error functions satisfy
\begin{align}
|E_{0,0}|&\leq 2^{12}\GO ^{-4}+2^2 \, 49 \, \EC ^2\notag\\
|E_{0,1}|&\leq 2^{13} \GO ^{-4}+\EC ^2\notag\\
|E_{1-1}|&\leq  2^{21} \GO ^{-1}+2 \, 49 \, \EC ^2\notag\\
|E_{1,-2}|&\leq  2^{17}\GO ^{-1}+\frac{49}{3}\EC \notag\\
\textr{|E_{1,-3}|}&\leq  2^{17}\GO ^{-1}+15\EC \notag\\
|\mathcal{E}_0|&\leq 2^{14} \,\EC ^2 \GO ^{-7}\notag\\
%\EC ^2 \GO ^{-3} + \GO ^{-7}\notag\\
|\mathcal{E}_1|&\leq 2^{18} \EC\EXP ^{-\GO ^3/3}\left[ \EC ^2\GO ^{7/2}+ \GO ^{-3/2} \right]\label{BoundE1}\\
\left|\mathcal{L}_{\geq2}\right|&\leq 2^{28} \GO ^{3/2}\EXP ^{-2\GO ^3/3}\label{BoundL2}
\end{align}
\end{theorem}

\begin{remark}\label{rem:Cc}
To estimate properly the  first  harmonics $L_{0,0}$, $L_{0,1}$, $L_{1,-1}$, $L_{1,-2}$, \textr{$L_{1,-3}$} we will need to take $G>C$, with $C$ big
enough and $ \EC G <c$ with $c$ small enough to ensure that the corresponding relative errors $E_{i,j}$ are smaller than one, say   $|E_{i,j}|\le 1/2$.
This is the main reason why
we have to enlarge the constant $C=32$ given in Theorem \ref{thepropositionmain}.
\end{remark}

The function $\mathcal{L}_1$ introduced in \eqref{eqn:LFourier3} with $q=1$, contains only harmonics of first order in $\SO $,
so we can write it as a cosine function in $\SO $.
Introducing the parameters (depending on $\GO$ and $\EC$)
\begin{align}
p&:=-\frac{L_{1,-2}}{L_{1,-1}}=12\EC \GO ^2\frac{1+E_{1,-2}}{1+E_{1,-1}}=:12\EC \GO ^2(1+E_p)\label{eqn:p}\\
q&:=-\frac{L_{1,-3}}{L_{1,-2}}=\frac{19}{24}\EC \GO \frac{1+E_{1,-3}}{1+E_{1,-2}}=:\frac{19}{24}\EC \GO (1+E_q)\label{eqn:q}
\end{align}
with
\[
E_p,E_q =O(\EC,\GO^{-1})%\label{EpEq}
\]
in the expression~\eqref{eqn:LFourier3} of $\mathcal{L}_1$, we can write
\begin{align}
\mathcal{L}_1&=2L_{1,-1}\left(\sum_{k\in \mathbb{Z}}\frac{L_{1,k}}{L_{1,-1}}\cos(\SO +k\AO )\right)\notag\\
&=2L_{1,-1}\left(\cos(\SO -\AO )-p\cos(\SO -2\AO )+qp\cos(\SO -3\AO )+
\sum_{k\neq -1,-2,-3}\frac{L_{1,k}}{L_{1,-1}}\cos(\SO +k\AO )\right)\notag\\
&=2L_{1,-1}\Re\left(\EXP ^{i(\SO -\AO )}\left(1-p\EXP ^{-i\AO }+qp\EXP ^{-2i\AO }+
\sum_{k\neq -1,-2,-3}\frac{L_{1,k}}{L_{1,-1}}\EXP ^{i(k+1)\AO }\right)\right)\notag\\
&=2L_{1,-1}\Re\left(\EXP ^{i(\SO -\AO )}B \EXP ^{-i\theta}\right)
=2L_{1,-1}B \cos(\SO -\AO -\theta),
\label{eqn:L1cosine}
\end{align}
where $B=B(\AO ,\GO ;\EC )\geq0$ and $-\theta=-\theta(\AO ,\GO ;\EC )\in[-\pi,\pi)$
are the modulus and the argument of the complex expression
\begin{equation}\label{eqn:Btheta}
1-p\EXP ^{-i\AO }+qp\EXP ^{-2i\AO }+
\sum_{k\neq -1,-2,-3}\frac{L_{1,k}}{L_{1,-1}}\EXP ^{i(k+1)\AO }=:B \EXP ^{-i\theta}.
\end{equation}
Writing also in polar form the quotient of the sum
in~\eqref{eqn:Btheta} by the parameter $p$ introduced in~\eqref{eqn:p}
\[
E \EXP ^{-i\phi}:=
\sum_{k\neq -1,-2,-3}\frac{L_{1,k}}{pL_{1,-1}}\EXP ^{i(k+1)\AO }=
-\sum_{k\neq -1,-2,-3}\frac{L_{1,k}}{L_{1,-2}}\EXP ^{i(k+1)\AO }=
q\sum_{k\neq -1,-2,-3}\frac{L_{1,k}}{L_{1,-3}}\EXP ^{i(k+1)\AO },
\]
with $E=E(\AO ,\GO ;\EC )\geq0$ and $-\phi=-\phi(\AO ,\GO ;\EC )\in[-\pi,\pi)$,
equation~\eqref{eqn:Btheta} for $B$ and $\theta$ reads now as
\begin{equation}\label{eqn:Bthetacomplex}
B \EXP ^{-i\theta}= 1-p\EXP ^{-i\AO }+qp\EXP ^{-2i\AO }+p E \EXP ^{-i\phi}
\end{equation}
or, equivalently, as the couple of real equations
\begin{align}
B \cos \theta&=1-p \cos\AO + qp \cos2\AO + p E \cos \phi= 1-qp -(1-2q\cos\AO)p\cos\AO + p E \cos \phi\label{eqn:Btheta1}\\
-B \sin \theta&= p \sin\AO  - qp \sin2\AO - p E \sin \phi =(1-2q \cos\AO)p \sin\AO  - p E \sin \phi\label{eqn:Btheta2}.
\end{align}
\textr{One can also obtain explicit formulas for $B$:
\begin{equation}\label{eq:B}
\begin{array}{rcl}
B^2&=&1+p^2(1+q^2+E^2)\\
&+&
2p\left( -(1+qp)\cos \AO+ q\cos(2\AO)+E\cos \phi+pE\cos(\phi+\AO) +p q E\cos(\phi-2\AO) \right)
\end{array}
\end{equation}
}

The function $E=E(\AO ,\GO ;\EC )$ is small, since, by~\eqref{BoundE1},~\eqref{L1-2} and~\eqref{eqn:p},
if $\GO>C$ is large enough and $\GO \EC<c$ is small enough (see Remark \ref{rem:Cc}),
\begin{equation}\label{BoundE}
|E|\leq \frac{|\mathcal{E}_1|}{|L_{1,-2}|}
%=\frac{|\mathcal{E}_1|}{|p L_{1,-1}|}
\leq
\frac{2^{19} \EC (\EC ^2 \GO ^{7/2}+ \GO ^{-3/2})}{\frac {3}{2}\sqrt{2\pi}\EC \GO ^{3/2}}
= \frac{2^{20}}{3 \sqrt{2 \pi}}  (\EC ^2 \GO ^2 + \GO ^{-3})
%\frac{1}{3\sqrt{2\pi}}\left(\left(1+\frac{3\pi(1+\EC )^4}{2\GO }\right)\frac{1}{\GO ^3}+\EC  \GO \right)
=O\left(\GO ^{-3},\EC ^2 \GO ^2\right),
\end{equation}
with an analogous bound for its derivative with respect to $\AO $.

 From expression~\eqref{eqn:L1cosine}, $\mathcal{L}_1$  is a genuine cosine function in s
(non identically zero) as long as $B>0$. If we first consider the case $E=0$ in the
equations~(\ref{eqn:Btheta1}-\ref{eqn:Btheta2}) defining $B$, it follows that $B=0$ only for
$\AO=0$ and $1-p+qp=0$, or $p=1/(1-q)$, that is, for $\GO \simeq (12 \EC )^{-1/2}$
(see~(\ref{eqn:p}-\ref{eqn:q})).
A totally analogous property holds when $E$ is taken into account:
%Writing equation~\eqref{eqn:Bthetacomplex} as
%\begin{equation}\label{BBhat}
%B \EXP ^{-i\theta}=\widehat{B} \EXP ^{-i\widehat{\theta}}+p E \EXP ^{-i\phi}
%\end{equation}
%one gets the explicit formulae for $\widehat{B}$ and $\widehat{\theta}$
%\begin{align}
%\widehat{B}&=\sqrt{1-2p\cos \AO +p^2}=\sqrt{(1-p)^2+4p\sin^2\frac{\AO}{2} }\geq 0,
%\label{eqn:Bhat}\\
%\widehat{\theta}&=
%-2 \arctan\left(\frac{p\sin\AO }{\widehat{B}+1-p\cos\AO }\right)\in(-\pi,\pi]\notag
%\end{align}
%from which we see that $\widehat{B}$ behaves like a distance to the point $p=1$ and
%$\AO =0$. The angle $\widehat{\theta}$ is not well defined when $\widehat{B}=0$, but
%this happens only for $\AO =0$ and $p=1$, that is, for $\GO \simeq (12 \EC )^{-1/2}$ (see \eqref{eqn:p}).
%A totally analogous property holds for $B$:
\begin{lemma}\label{Bnotzero}
There exists $C>32$ and $c<1/8$ such that, for $G\ge C$ and $\EC G<c$, then $B(\AO ,\GO ;\EC )>0$ except for $\AO =0$ and
$\sum_{k\in\mathbb{Z}}L_{1,k}=0$.
\end{lemma}
\begin{remark}
$\displaystyle\sum_{k\in\mathbb{Z}}L_{1,k}=0 \iff 1-p+qp+\sum_{k\neq -1,-2,-3}\frac{L_{1,k}}{L_{1,-1}}=0\iff
p=\frac{\displaystyle 1+\sum_{k\neq -1,-2,-3}\frac{L_{1,k}}{L_{1,-1}}}{1-q}$.
\end{remark}
\begin{proof}
For $B=0$,
equation~\eqref{eqn:Btheta2} reads as
\begin{equation}\label{eqn:alpha}
\sin\AO =f(\AO ).
\end{equation}
where $f(\AO )=f(\AO ,\GO ;\EC ):=E \sin\phi/(1-2q\cos\AO)$.
By~\eqref{eqn:q} and~\eqref{BoundE}, if $\GO>C$ is large enough and $\GO \EC<c$ is small enough, we have that
$f^2 + (\partial f/\partial\AO )^2 <1$, and therefore there are exactly two simple solutions
of equation~\eqref{eqn:alpha} in the interval $[-\pi/2,3\pi/2]$; one is $\alpha_{0,+}^*\in(-\pi/2,\pi/2)$
obtained as a fixed point of the contraction $\AO =\arcsin\left(f(\AO ,\GO ;\EC )\right)$,
and a second $\alpha_{0,-}^* \in(\pi/2,3\pi/2)$
fixed point of the contraction $\AO =\pi-\arcsin\left(f(\AO ,\GO ;\EC )\right)$.
Taking a closer look at equation~\eqref{eqn:Bthetacomplex},
we see that if $\AO $ changes to $-\AO $,  then $-\phi, -\theta, B, E$ are solutions of~\eqref{eqn:Bthetacomplex}
or, in other words, $\phi,\theta$ are odd functions of $\AO $ and $B, E$ even. Therefore,
$\AO =0,\pi$  are the unique solutions of equation~\eqref{eqn:Btheta2} for $B=0$.
Substituting $\AO =0,\pi$ in~\eqref{eqn:Btheta1} for $B=0$, only $\AO =0$ provides
a positive $p$, which is then given by $p=1+qp+pE=(1+\sum_{k\neq -1,-2,-3}L_{1,k}/L_{1,-1})/(1-q)$.
\end{proof}

We are now in position to find critical points of the function
$\SO \mapsto \mathcal{L}(\AO ,\GO ,\SO ;\EC )$.
To this end we will check that $\SO \mapsto \mathcal{L}(\AO ,\GO ,\SO ;\EC )$ is indeed
a \emph{cosine-like} function, that is, with a non-degenerate maximum (minimum) and
no other critical points.
By Theorem~\ref{thepropositionmain}, the dominant part
of the Melnikov potential $\mathcal{L}$ is given by $\mathcal{L}_0+\mathcal{L}_1$.
By equation \eqref{MeliL0L1} and the bounds for the error term, for $\GO>C$ big enough and $\GO \EC<c$ small enough,
the critical points in the variable $\SO $ are well approximated by the critical points of
the function $\mathcal{L}_0+\mathcal{L}_1$ \textr{(in fact of $\mathcal{L}_1$ because  $\mathcal{L}_0$ does not depend on $\SO$)}
and therefore will be close to
$\SO -\AO -\theta=0,\pi \pmod{2\pi}$ thanks to expression~\eqref{eqn:L1cosine}.
For this purpose, we introduce
\begin{equation}
\label{L1L1star}
 \mathcal{L}_1^*=\mathcal{L}_1^*(\AO ,\GO ;\EC )=2L_{1,-1} B%\label{L1*}
\end{equation}
where $B=B(\AO ,\GO ;\EC )$ is given in~\eqref{eqn:Btheta} \textr{(see also \eqref{eq:B})} and $L_{1,-1}$ is the harmonic
computed in~\eqref{L1-1}. By \eqref{eqn:L1cosine} and \eqref{L1L1star}, the function $\mathcal{L}_1$ can thus be written as a cosine function in $\SO $
\[
%\label{L1star}
\mathcal{L}_1(\AO ,\GO ,\SO ;\EC )=\mathcal{L}_1^*(\AO ,\GO ;\EC )\, \cos (\SO -\AO -\theta),
\]
and differentiating the Melnikov potential~\eqref{MeliL0L1} with respecto to $\SO $ we get
\begin{equation*}%\label{eqn:Lcritical}
\frac{\partial{\mathcal{L}}}{\partial \SO }=-\mathcal{L}_1^*\sin(\SO -\AO -\theta)
+\frac{\partial{\mathcal{L}}_{\geq2}}{\partial \SO }=0
\Longleftrightarrow
\sin(\SO -\AO -\theta)=\frac{1}{\mathcal{L}_1^*}\frac{\partial{\mathcal{L}}_{\geq2}}{\partial \SO }
\end{equation*}
which is a equation of the form~\eqref{eqn:alpha} for $\SO -\AO -\theta$ instead of $\AO $
and $f=\left(\partial{\mathcal{L}}_{\geq2}/\partial \SO \right)/\mathcal{L}_1^*=
\left(\partial{\mathcal{L}}_{\geq2}/\partial \SO \right)/(2L_{1,-1}B)$.
Therefore, as long as $f^2 + (\partial f/\partial\AO )^2 <1$, there exist exactly two non-degenerate
critical points $\SO_{\pm}^*$ of the function $\SO \mapsto \mathcal{L}(\AO ,\GO ,\SO ;\EC )$.

Now, by estimate~\eqref{L1-1} for $L_{1,-1}$, bound~\eqref{BoundL2} for $\mathcal{L}_{\geq2}$
and Lemma~\ref{Bnotzero}, it turns out that
$f^2 + (\partial f/\partial\AO )^2 <1$
happens outside of a neighborhood of size $O\left(\GO ^{3/2}\EXP ^{-\GO ^3/3}\right)$
of the point
\begin{equation}\label{eqn:G0*}
(\AO =0,\GO =\GO ^*)\ \textrm{where}\ \GO ^*\approx (12 \EC )^{-1/2}\ \textrm{is such that}\
p=\frac{\displaystyle 1+\sum_{k\neq -1,-2,-3}\frac{L_{1,k}}{L_{1,-1}}}{1-q}.
\end{equation}

%\marginpar{l'entorn hauria de ser $O\left(\GO ^2\EXP ^{-\GO ^3/3}\right)$}.

Let us recall now that the Melnikov function $\mathcal{L}$ (see \eqref{eqn:LFourier2},
\eqref{eqn:LFourier3}), as well as its terms $\mathcal{L}_q$ are all expressed
as Fourier Cosine series in the angles $\AO $ and $\SO $, or equivalently, they are even functions of
$(\AO ,\SO )$.
Consequently, $\partial \mathcal{L}_q/\partial \SO $ is an odd function of $(\AO ,\SO )$,
and it is easy to check that each critical point $\SO_{\pm}^*$  is an odd function of $\alpha$.
Moreover, using the Fourier Sine expansion of $\partial \mathcal{L}_q/\partial \SO $,
one sees that if $\SO $ is a critical point of $\SO \mapsto \mathcal{L}(\AO ,\GO ,\SO ;\EC )$, $\SO +\pi$ too,
so $\SO_{-}^*=\SO_{+}^*+\pi$. We state all this in the following Proposition.

\begin{proposition}%\label{coslike}
Let $\mathcal{L}$ be the Melnikov potential given in \eqref{MeliL0L1}.
There exists $C>32$ and $c<1/8$ such that, for $G\ge C$ and $\EC G<c$,
%, $\GO \geq 32$ and $\EC \GO \leq 1/8$.
except for a neighborhood of size $O\left(\GO ^{3/2}\EXP ^{-\GO ^3/3}\right)$ of the point
$(\AO =0,\GO =\GO ^*)$ given in~\eqref{eqn:G0*},
$\SO \mapsto \mathcal{L}(\AO ,\GO ,\SO ;\EC )$ is a \emph{cosine-like} function, and its critical
points are given by
\[
\SO_{+}^*=\SO_{+}^*(\AO ,\GO ;\EC )=\AO +\theta+\varphi^*, \qquad
\SO_{-}^*=\SO_{-}^*+\pi=\AO +\theta+\pi+\varphi^*
\]
where $\theta=\theta(\AO ,\GO ;\EC )$ is given in~\eqref{eqn:Btheta} and
$\varphi^*=O\left(\GO ^{3/2}\EXP ^{-\GO ^3/3}\right)$.
\end{proposition}

\marginpar{el mateix de $\varphi^*$}
From the Proposition above we know that there exist  $\SO_{+}^*$ and $\SO_{-}^*=\SO_{-}^*+\pi$, non-degenerate critical points of
$\SO \mapsto \mathcal{L}(\AO ,\GO ,\SO ;\EC )$.
Therefore, applying Proposition \ref{theprop} and Remark \ref{rem:dist}, we know that $W^u(\tilde \Lambda_\infty)$ intersects transversally
$W^s(\tilde \Lambda_\infty)$ if $0<\mu \ll \mu^*$ with
\[
\mu ^*= O\left(\frac{\partial^2}{\partial \SO^2 }\left(\mathcal{L}\left(\AO ,\GO ,\SO^*;\EC \right)\right)\right).
\]
Using Theorem \ref{thepropositionmain} and \eqref{L1L1star}, we see that it is enough to impose that:
\[
|\mu|\ll \left|\mathcal{L}^*_1\right|=2|L_{1,-1}B|=O\left(\GO ^{-1/2}\EXP ^{-\GO ^3/3}\right),
\]
that is, $\mu$ exponentially small for large $\GO $ in the region $C\leq \GO \leq c/\EC $
which a fortiori is satisfied for
\begin{equation}\label{Boundmu}
0<\mu \ll \mu^*= \EXP ^{-(c/\EC )^3 /3}.
\end{equation}
%\marginpar{no veig clar el $\mu^*$}

We will see that this is the relation between the eccentricity
and the mass parameter that we need to guarantee that our main result, Theorem \ref{MainResult}, holds.
This kind of relation is typical
in problems with exponentially small splitting, when the bound of the remainder, here $O(\mu^2)$, is obtained
through a direct application of the Melnikov method for the real system. To get better estimates for this remainder,
one needs to bound this remainder for complex values of the parameter $t$ or $\tau$ of the
parameterization~\eqref{homoclinic:h} of the unperturbed separatrix. Such approach has recently been
used for in the RPCTBP in~\cite{GuardiaMS12} and it is likely to work in the \textb{\RPETBP}, allowing us to
consider any $\mu\in(0,1/2]$, that is, imposing no restrictions on the mass parameter, although this is not the purpose
of this paper which focuses on the geometric mechanism that gives rise to diffusive orbits.

We are now in position to define two different scattering maps $\tilde S^\pm$.
By Proposition \ref{scatteringlemma}, we begin by defining  two different reduced Poincar\'e functions~\eqref{Lcalstar}
\begin{align*}
 \mathcal{L}^*_{\pm}(\AO ,\GO ;\EC )&=\mathcal{L}(\AO ,\GO ,\SO_{\pm}^*;\EC )\\
&= \mathcal{L}_0(\AO ,\GO ;\EC )\pm \mathcal{L}^*_1(\AO ,\GO ;\EC )+\mathcal{E}_{\pm}(\AO ,\GO ;\EC ).
\end{align*}

By the symmetry properties of $\mathcal{L}_{q}(\AO, \GO, \SO;\EC)$ (see \eqref{eqn:LFourier3}), it turns out that
each $(\mathcal{L}^*_{q})_{\pm}(\AO ,\GO ;\EC )=\mathcal{L}(\AO ,\GO ,s_\pm^*;\EC )$ is an even
function of $\alpha$.
\textr{Moreover, since $\SO_{-}^*=\SO_{+}^*+\pi$, one has that
$(\mathcal{L}^*_{q})_-=(-1)^q (\mathcal{L}^*_{q})_+$, so we can write the reduced Poincar\'e map as
\begin{equation}\label{eqn:L*sum}
\mathcal{L}^*_{\pm}=\mathcal{L}_0\pm \mathcal{L}^*_1 +\mathcal{L}^*_2\pm\mathcal{L}^*_3+\mathcal{L}^*_4\pm\cdots
\end{equation}
with $\mathcal{L}^*_{q}=(\mathcal{L}^*_{q})_+$.}

From the expression for the scattering map given in Proposition \ref{scatteringlemma} we can define two different scattering maps
$\widetilde S_{\pm}(\AO ,\GO ,\SO )=( S_{\pm}(\AO ,\GO ,\SO ), \SO)$, where
\begin{equation}\label{Spm}
S_{\pm}(\AO ,\GO ,\SO )=
\left(\AO  -\mu \frac{\partial \mathcal{L}^*_\pm}{\partial G}(\AO ,\GO ;\EC )+O(\mu^2),
\GO +\mu \frac{\partial\mathcal L^*_\pm}{\partial \alpha}(\AO ,\GO ;\EC )+O(\mu^2) \right).
\end{equation}
These two scattering maps are different since they depend on the two reduced
Poincar\'e-Melnikov potentials $\mathcal{L}^*_\pm$. From their expression~\eqref{Spm},
the scattering maps $ S_{\pm}$ follow closely the level curves of the Hamiltonians $\mathcal{L}^*_\pm$.
More precisely, up to $O(\mu^2)$ terms, $ S_{\pm}$ is given by the time $-\mu$ map of the Hamiltonian flow
of Hamiltonian $\mathcal{L}_\pm^*$. The $O(\mu^2)$ remainder will be negligible as long as
\[
|\mu|\ll \left|\frac{\partial \mathcal{L}^*_\pm}{\partial G}\right|,
\left|\frac{\partial\mathcal L^*_\pm}{\partial \alpha}\right|,
\]
which is already true for $\mu \ll \mu^*$ in \eqref{Boundmu}.
%Nevertheless, since we want to switch scattering maps, we will need to impose

We want to show now that the foliations of $\mathcal{L}^*_\pm=\textrm{constant}$ are different, since this will imply
that the scattering maps \textr{$S_{\pm}$} are different.
Even more, we will design a mechanism in which we will determine the places in the plane $(\AO ,\GO )$ where
we will change from one scattering map to the other, obtaining trajectories with increasing angular momentum $G$.
To check that the level curves of $\mathcal{L}^*_+$ and $\mathcal{L}^*_-$ are different, and indeed transversal,
we only need to check that their Poisson bracket is not zero.
Since $\mathcal{L}^*_+$ and $\mathcal{L}^*_-$ are even functions of $\AO $, their Poisson bracket
$\{\mathcal{L}^*_+,\mathcal{L}^*_+\}$ will be an odd function of $\AO $, so we already know that
it will have a factor $\sin \AO $.
Using equation~\eqref{eqn:L*sum} we can write
\begin{equation}\label{eq:Poisonb}
  \{\mathcal{L}^*_+,\mathcal{L}^*_-\}=
\{\mathcal{L}_0+\mathcal{L}^*_1+\mathcal{L}^*_2+\cdots,\mathcal{L}_0-\mathcal{L}^*_1+\mathcal{L}^*_2-\cdots\}
=-2 \{\mathcal{L}_0,\mathcal{L}^*_1\}+\mathcal{E}_3
\end{equation}

where $\mathcal{E}_3$ contains only Poisson brackets of odd order
\[
\mathcal{E}_3=-2\left(\{\mathcal{L}_0,\mathcal{L}^*_3\}+\{\mathcal{L}^*_1,\mathcal{L}^*_2\}\right)
-2\sum_{ q\ \textrm{odd}\geq5} \sum_{q=0}^{[q/2]}\{\mathcal{L}^*_{q'},\mathcal{L}^*_{q-q'}\}.
\]
Therefore, by formula \eqref{eqn:LFourier3} defining $\mathcal{L}_q$ and  the bounds~\eqref{eqn:BBounds} for the harmonics $L_{q,k}$,
the error term $\mathcal{E}_3=O\left(\EXP ^{-\GO ^3}\right)$ is much
exponentially smaller for large $\GO $ than $\{\mathcal{L}_0,\mathcal{L}^*_1\}$,
which is $O\left(\EXP ^{-\GO ^3/3}\right)$ and we now compute.

By differentiating $\mathcal{L}_0$, using~\eqref{calL0} and bounds \eqref{L00} and \eqref{L01}, one easily obtains:
\textr{
\begin{eqnarray*}
%\frac{\partial \mathcal{L}_0}{\partial \alpha}&=&\EC \sin \alpha\left(\frac{15 \pi }{8 \GO ^5} + O\left(\GO^{-9}, \EC  \GO^{-7}, \EC ^2 \GO^{-5}\right)\right)\\
\frac{\partial \mathcal{L}_0}{\partial \alpha}&=&\frac{15 \pi \EC}{8 \GO ^5} \sin \alpha \left(1 + O(\GO^{-4}, \EC  \GO^{-2}, \EC ^2 \GO^{-1})\right)\\
 %\frac{\partial \mathcal{L}_0}{\partial \GO}&=&  -\frac{3\pi}{2\GO ^4}+ \frac{75 \pi \EC}{8 \GO ^6} \cos \alpha+ O\left(\GO^{-8},  \EC ^2 \GO^{-4}\right).\\
 \frac{\partial \mathcal{L}_0}{\partial \GO}&=&  -\frac{3\pi}{2\GO ^4}\left(1+O(\GO^{-4}, \EC^2)\right)+
 \frac{75 \pi \EC}{8 \GO ^6} \cos \alpha \left(1+ O(\GO^{-4},  \EC ^2 )\right).
\end{eqnarray*}}
With respect to $\mathcal{L}^*_1=2 L_{1,-1}B$, we will use \eqref{eqn:Btheta} and the definitions of $p,q$ in (\ref{eqn:p}-\ref{eqn:q}) which give
\[
\mathcal{L}^*_1 \EXP ^{-i\theta}=
%2L_{1,-1}+2L_{1,-2}\EXP ^{-i\AO }+2L_{1,-3}\EXP ^{-2i\AO }+\sum_{k\neq -1,-2,-3}L_{1,k}\EXP ^{i(k+1)\AO }=
2L_{1,-1}+2L_{1,-2}\EXP ^{-i\AO }+2L_{1,-3}\EXP ^{-2i\AO } +\mathcal{E}^*_1,
\]
where the error term $\mathcal{E}^*_1$ contains a factor $\sin \AO$ and satisfies the same bound as $\mathcal{E}_1$ in~\eqref{BoundE1}:
\[
\mathcal{E}^*_1=\sum_{k\neq -1,-2,-3}L_{1,k}\EXP ^{i(k+1)\AO }=O\left(\EC\GO^{-\frac{3}{2}}, \EC^3\GO^{\frac{7}{2}}\right)\EXP ^{-\GO ^3/3}
\]
%$\mathcal{E}^*=O(\EC^2\GO^{\frac{5}{2}},\EC\GO^{-\frac{3}{2}})\EXP ^{-\GO ^3/3}$.

Taking into account the expressions for $L_{1,-1},L_{1,-2},L_{1,-3}$ given in~(\ref{L1-1}-\ref{L1-3}),
%\eqref{BBhat} and~\eqref{BoundE}, in their dominant and non-dominant parts
%\[
%\mathcal{L}_0=\widehat{\mathcal{L}}_0+\widehat{\mathcal{E}}_0,\qquad
% L_{1,-1}=\widehat{L}_{1,-1}(1+E_{1,-1}),\qquad
% B=\widehat{B}+E_B,
%\]
%where
%\[
 %\widehat{\mathcal{L}}_0= \frac{\pi}{2 \GO ^3}-\frac{15 \pi\EC}{8 \GO ^5}cos \alpha, \qquad \widehat{L}_{1,-1}=\sqrt{\frac{\pi}{8 \GO}}\EXP ^{-\GO ^3/3},
%\]
%and $\widehat{B}$ is given in \eqref{eqn:Bhat}.
%Calling $\widehat{\mathcal{L}}^*_1=2\widehat{L}_{1,-1}\widehat{B}$,
after a straightforward computation, we arrive at
%\begin{eqnarray*}
%2\mathcal{L}^*_1 \frac{\partial \mathcal{L}^*_1}{\partial \AO}&=&\EC \pi \sin \alpha\EXP ^{-2\GO ^3/3}\left(12 \GO  + O(1, \EC  \GO^{2}, \EC ^2 \GO^{4})\right)\\
%2\mathcal{L}^*_1  \frac{\partial \mathcal{L}_0}{\partial \GO}&=& \pi \EXP ^{-2\GO ^3/3}\left( -\GO - 144 \EC^2 \GO ^5+ 24 \EC \GO ^3 \cos \alpha+
%O\left(1, \EC \GO^{2},  \EC ^2 \GO^{4},  \EC ^3 \GO^{6}\right)\right).
%\end{eqnarray*}
\begin{eqnarray*}
\frac{\partial \mathcal{L}^*_1}{\partial \AO}&=&\frac{1}{B}\frac{\partial B}{\partial \AO}\mathcal{L}^*_1=
\frac{p \mathcal{L}^*_1 \sin\AO}{B^2}\left(1+qp-2q\cos\AO
%+O\left(\EC\GO^{-1},(\EC\GO)^3\GO\right)\right)\\
+O\left(\GO^{-3}, \EC \GO^{-1}, \EC^2\GO^{2},(\EC\GO)^3\GO\right)\right)\\
\frac{\partial \mathcal{L}^*_1}{\partial \GO}&=&-\GO^2(1+O(\GO^{-1})\mathcal{L}^*_1.
\end{eqnarray*}
Using these computations
%and recalling that $p\simeq 12\EC \GO^2$,
we arrive at
\begin{equation}\label{eq:Poisonb2}
\{\mathcal{L}_0,\mathcal{L}^*_1\}=-\frac{15\pi\EC\mathcal{L}^*_1 d\sin\AO}{8\GO^3B^2}
\end{equation}
with
%AMADEU:\begin{multline}\label{eqn:d}
%d:=B^2O\left(1+O\left(\EC^2,\GO^{-1}\right)\right)\\-\frac{4p}{5\EC\GO}\left(1+qp-4q\cos\AO
%+O\left(\EC^3\GO^4\right)\right)
%+\left(1-\frac{25\EC}{4\GO^2}\cos\AO+O\left(\GO^{-4},\EC^2\right)\right).
%\end{multline}
\textr{
\begin{multline}\label{eqn:d}
d:=B^2\left(1+O\left(\GO^{-1}\right)\right)
-\frac{4p}{5\EC\GO}\left(1+qp-2q\cos\AO
%-\frac{48 G}{5}\left(1+qp-2q\cos\AO
+O\left(\GO^{-3}, \EC \GO^{-1},\EC^2\GO^2,\EC^3\GO^4\right)\right)\times\\
\left(1-\frac{25\EC}{4\GO^2}\cos\AO \left(1+O\left(\GO^{-4},\EC^2\right)\right)+O\left(\GO^{-4},\EC^2\right)\right).
\end{multline}
}

%\begin{equation}\label{eq:Poisonb2}
% \begin{array}{rcl}
%\{\mathcal{L}_0,\mathcal{L}^*_1\}
%\begin{equation}\label{eq:Poisonb2}
% \begin{array}{rcl}
%\{\mathcal{L}_0,\mathcal{L}^*_1\}  &=&\frac{\EC \pi ^2}{2 \mathcal{L}^*_1}\sin \alpha \EXP ^{-2\GO ^3/3}
%\left[18 \GO^{-3}-18 \EC ^2 +O( \GO^{-4}, \EC \GO^{-2}, \EC^2, \EC ^3 \GO)\right.\\
%&+&\left.
%45 \EC \GO^{-2} \cos \AO(1+O( \GO^{-4}, \EC \GO^{-2}, \EC ^2))\right] \\
%%&=&\frac{\EC \pi ^2}{2 \GO ^3 \mathcal{L}^*_1}\sin \alpha \EXP ^{-2\GO ^3/3}
%%\left(d+O( \GO^{-1}, \EC \GO, (\EC \GO)^3)\right)
%&=&\frac{\EC \pi ^2}{2 \GO ^3 \mathcal{L}^*_1}\sin \alpha \EXP ^{-2\GO ^3/3}
%\left(d_1+d_2\cos \AO\right)
%\end{array}
%\end{equation}
%with
%$$
%d_1= 18-18\EC^2\GO ^3+O( \GO^{-1}, \EC \GO, \EC^2\GO^3, \EC ^3 \GO^4), \quad d_2=45\EC\GO +O( \EC\GO^{-3}, \EC^2 \GO^{-1}, \EC ^3\GO)
%$$
%where
%$$
%d=18 -18 \EC ^2 \GO^3+45 \EC \GO \cos \AO
%$$

%\begin{equation}\label{eq:Poisonb2}
%-2 \{\mathcal{L}_0,\mathcal{L}^*_1\}=-2 \{\widehat{\mathcal{L}}_0,\widehat{\mathcal{L}}^*_1\}+\mathcal{E}_J
%\end{equation}
%where
%\[
 %-2\{\widehat{\mathcal{L}}_0,\widehat{\mathcal{L}}^*_1\}=
%\frac{-\widehat{\mathcal{L}}^*_1}{\widehat{B}^2}\frac{3\pi p\sin \AO }{\GO ^4} d
%\]
%with
%\[
%d=\Biggl[1-\frac{25}{4}\frac{\EC \GO }{\GO ^3}\cos\AO -\frac{5}{48}\frac{B^2}{\GO }
%\left[1+\frac{1}{2\GO ^3}-\frac{-\cos \AO  +p}{B^2} \cdot \frac{24\EC \GO }{\GO ^2}\right]\Biggl].
%\]
%and a small error term
%\[
%\mathcal{E}_J=O\Bigl(\GO ^{-5}+\EC \GO ^{-3}+\EC ^2\GO ^3+p\EC ^2\GO ^4\bigl(1+p(\EC \GO +\GO ^{-6})\bigl)\Bigl)\GO ^{-1/2}\EXP ^{-\GO ^3/3}
%\]

\subsection{Strategy for diffusion}%\label{stratdiff}
The previous computations \eqref{eq:Poisonb}, \eqref{eq:Poisonb2} as well as   Lemma  \ref{Bnotzero} tell us that the level curves of
$\mathcal{L}^*_+$ and $\mathcal{L}^*_-$ are transversal in the region $\GO \geq C> 32$ and $\EC \GO \leq c< 1/8$,
except for the three curves $\alpha=0$, $\alpha=\pi$ and $d=0$, which are transversal to any of these level curves of
$\mathcal{L}^*_+$ and $\mathcal{L}^*_-$, see figure~\ref{LevelSetsPicture}.
Indeed, this is clear for the lines $\alpha=0$ and $\alpha=\pi$, and the same happens for
the curve $d=0$ using the expression of $d$ given in~\eqref{eqn:d} which implies
\[
%\GO=\left(\frac{2}{11\EC^2}\right)^{1/3}\left(1+\frac{95}{12} \left(\frac{2}{11}\right)^{4/3} \EC^{1/3}\cos\AO+O\left(\EC^{2/3}\right)\right).
\GO=\left(\frac{2}{11\EC^2}\right)^{1/3}\left( 1+K \EC^{1/3}\cos\AO+O\left(\EC^{2/3}\right)\right).
\]
with $K\ne 0$.
\kk{improve\\figures}

\begin{figure}[t]
\centering
\includegraphics[width=3.5in]{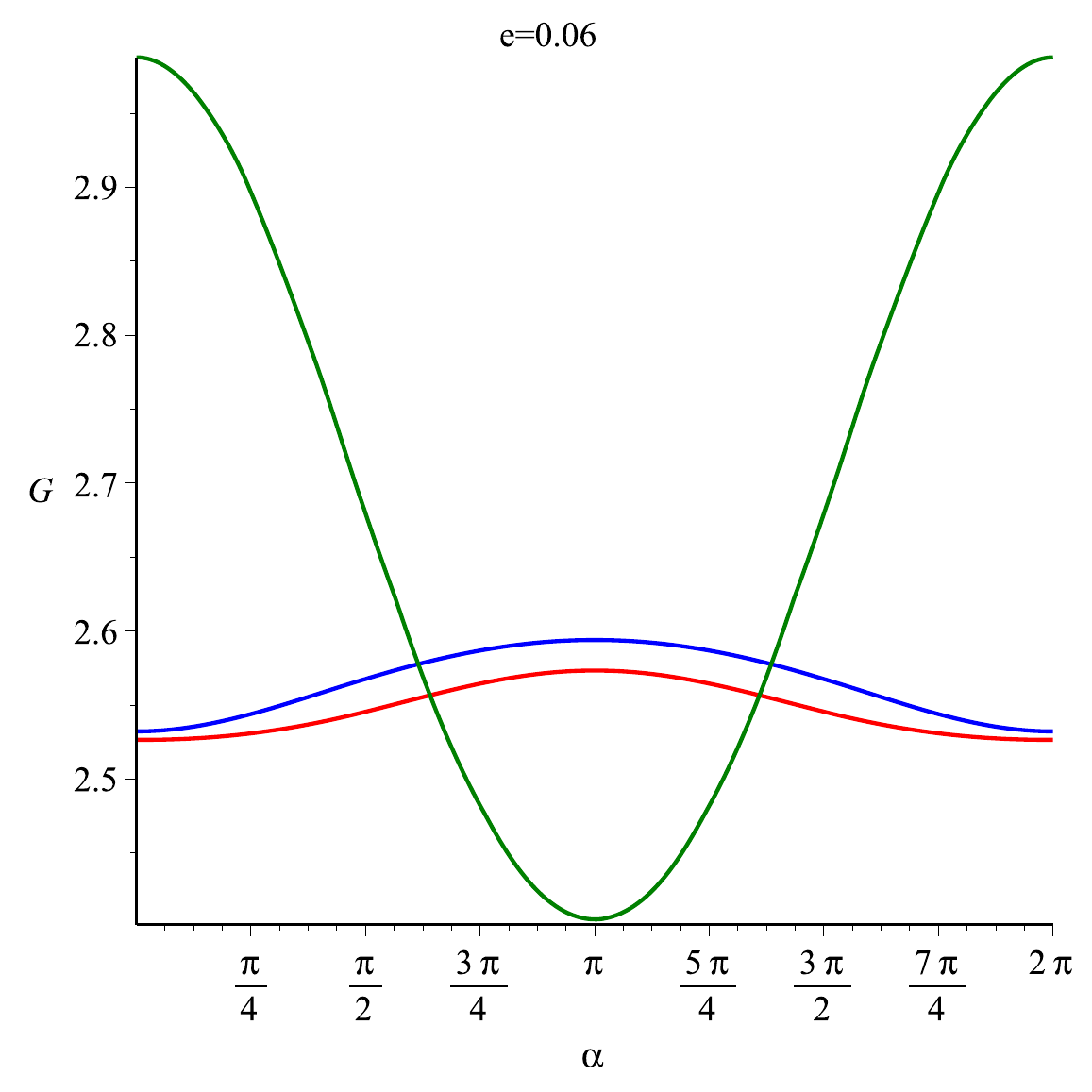}
\caption{Illustration of the level Sets of $\mathcal{L}^*_+$ ($\mathcal{L}^*_-$) in Blue (Red) and $d=0$ in Green}
\label{LevelSetsPicture}
\end{figure}

Thus, apart from these three curves, at any point in the
plane $(\AO ,\GO )$ the slopes $d\GO /d\AO $ of the level curves of $\mathcal{L}^*_+$ and $\mathcal{L}^*_-$
are different, and we are able to choose which level
curve increases more the value of $G$, when both slopes are positive, or alternatively,
to choose the level curve which decreases less the value of $G$, when both slopes are negative
(see Figure~\ref{TransitionPicture}). In the same way, we can find trajectories along which
the angular momentum performs arbitrary excursions. More precisely,
given an arbitrary finite sequence of values $G_i$, $i=1,\dots,n$
we can find trajectories which satisfy $G(T_i)=G_i$, $i=1,\dots,n$.

Strictly speaking, this mechanism given by the application of scattering maps produce indeed pseudo-orbits, that is,
heteroclinic connections between different periodic orbits in the infinity manifold which are commonly known as transition
chains after Arnold's pioneering work \cite{Arnold64}.
The existence of true orbits of the system which follow closely these transition chains relies on
shadowing methods, which are standard for partially hyperbolic periodic orbits (the so-called whiskered tori in the literature) lying on a normally hyperbolic
invariant manifold (NHIM)  \cite{Moeckel02,Moeckel07,GideaL06,GideaLS14}.
Such shadowing methods are equally applicable in our case as it is proven in \cite{GuardiaMSS17}, where we have an infinity
manifold $\tilde \Lambda_\infty$ which is only topologically equivalent to a
NHIM.

%The existence of true orbits of the system which follow closely these transition chains relies on
%shadowing methods, which are standard for partially hyperbolic periodic orbits (the so-called whiskered tori in the literature) lying on a normally hyperbolic
%invariant manifold (NHIM). Such shadowing methods are equally applicable in our case, where we have an infinity
%manifold $\tilde \Lambda_\infty$ which is only topologically equivalent to a NHIM (see \cite{Moeckel02,Moeckel07,GideaL06,GideaLS14}) and \cite{GuardiaMSS17}.

\begin{figure}
\centering
\includegraphics[width=4in]{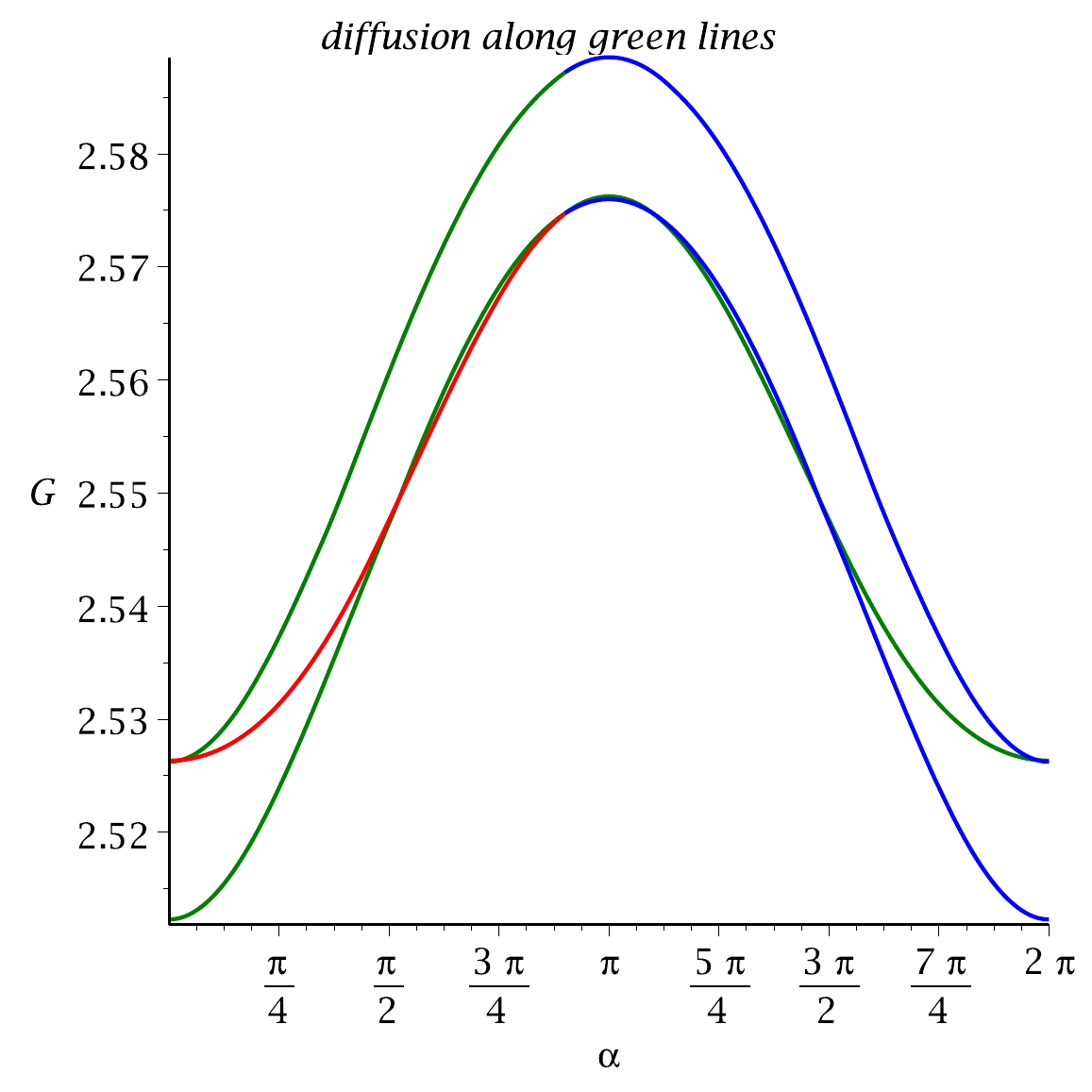}
\caption{Zone of diffusion: Level curves of $\mathcal{L}^{*}_{+}$ ($\mathcal{L}^{*}_{-}$) in blue (red)
and diffusion trajectories in green.
}
\label{TransitionPicture}
\end{figure}

With all these elements, we can finally state our main result
\begin{theorem}\label{Thm:Main}
Let $G_1^*<G_2^*$ large enough and $\EC >0$, $\mu>0$ small enough.
More precisely $C\leq G_1^*<G_2^*\leq c/\EC$ and $0<\mu < \mu^*=\frac{c}{C} \EXP ^{-(8\EC )^{-3} /3}$,
for $C>32$ large enough and $c<1/8$ small enough.
Then, for any finite sequence of values $G_i\in(G_1^*,G_2^*)$, $i=1,\dots,n$,
there exists a trajectory of the \textb{\RPETBP} such that
$G(T_i)=G_i$, $i=1,\dots,n$ for some $0<T_i<T_{i+1}$. In particular,
for any two values $G_1<G_2\in(G_1^*,G_2^*)$,
there exists a trajectory such that
$G(0)<G_1$, and $G(T)>G_2$ for some time $T>0$.
\end{theorem}

\section{Computation of the Melnikov potential: Proof of Theorem \ref{thepropositionmain}}\label{EstMeliPot}
The main difficulty to compute the Melnikov potential is that it is given by an integral~\eqref{cL}
where the coordinates of the separatrix $x_{\HC}$ and $\alpha_{\HC}$ are given implicitly~\eqref{homoclinic:h}
in terms of the time $t$ through the variable $\tau$~\eqref{cambiotau}, whereas
$\RO $ and $f$ are given in terms of $\SO +t$ through the differential equation~\eqref{f} defining the true anomaly $f$.
To evaluate the above Melnikov potential, we will compute its Fourier Cosine series~\eqref{eqn:LFourier1} in the angles
$\SO,\AO$.
\textr{We will detect that there are only five dominant harmonics, $L_{0,0}$, $L_{0,1}$,
$L_{1,-1}$, $L_{1,-2}$ and $L_{1,-3}$, so we will estimate them and bound all the rest.}

The plan of this proof is thus divided in different parts. In subsection~\ref{MPFourier} we Fourier expand the
Melnikov potential $\mathcal{L}$ to find that each of its harmonics $L_{q,k}$ is given by a series in terms of some
constants $\CT _{q}^{n,m}$ and integrals $N(q,m,n)$. General upper bounds for these constants and integrals are given in subsection~\ref{MPBounds},
which provide the upper bounds $B_{q,k}$ for the harmonics $L_{q,k}$ announced in Proposition~\ref{boundLqk}.
Since the upper bounds $B_{q,k}$ are exponentially small for large $G$ and $q\geq 1$, we split the
Fourier expansion~\eqref{eqn:LFourier1} as
\[
 \mathcal{L}=\mathcal{L}_{0}+\mathcal{L}_{1}+\mathcal{L}_{\geq2}
\]
where $\mathcal{L}_{0}$, $\mathcal{L}_{1}$ contain the harmonics of $\mathcal{L}$ of order $0$, $1$ in $\SO $, respectively,
whereas $\mathcal{L}_{\geq2}$ contain the harmonics of higher order, which are readily bounded.
Subsection~\ref{MPAsymp} contains an asymptotic expression for the integrals $N(q,m,n)$ which are necessary for the computation of
$\mathcal{L}_{1}$. Finally, the subsections~\ref{MP1} and~\ref{MP0} are devoted to the computation
to the harmonics of $\mathcal{L}_{1}$, and $\mathcal{L}_{0}$,
respectively, estimating, for each order, the two most dominant ones, and bounding all the rest.

\subsection{Fourier expansion of the Melnikov potential}
\label{MPFourier}

The next Proposition gives formulae for the Fourier coefficients~\eqref{eqn:LFourier1} of the Melnikov potential~\eqref{cL}.
For any integer $n$, $m$, we will use the Fourier expansion of the function
\begin{equation}\label{sumapinyol}
\RO (f(t))^{n}\ \EXP ^{imf(t)} =\sum _{q\in\mathbb{Z}}\CT ^{n,m}_q  \EXP ^{i q t}
\end{equation}
which can be found in \cite{MartinezP94} and \cite[p.~204]{Wintner41}.
Since $\RO $ is an even function and $f$ is and odd function, one readily sees that
the above coefficients are real and indeed they satisfy
\begin{equation*}%\label{simetriesc}
\CT _{-q}^{n,-m}=\CT _{q}^{n,m}=\overline{\CT _{q}^{n,m}} .
\end{equation*}
Once these coefficients $\CT _q^{n,m}$ are introduced we can give explicit formulae for the Fourier
coefficients of the Melnikov potential $\mathcal{L}$.
\begin{proposition}\label{RLFourierLq}
 The Melnikov potential given in~\eqref{cL} or in~\eqref{eqn:LFourier1} can be written as
\begin{equation}\label{LFourierLq}
\mathcal{L}=\sum_{q\in\mathbb{Z}}L_{q}\EXP ^{iq\SO }, \quad \text{where}\quad
L_q=\sum_{k\in \mathbb{Z}}L_{q,k}\EXP ^{ik\AO },
\end{equation}
with
\begin{equation}\label{FourierLqkN}
\begin{aligned}
 L_{q,0}&=\sum_{l\geq 1}\CT _{q}^{2l,0}N(q,l,l) &\\%\label{FourierLqkN:00}\\
 L_{q,1}&=\sum_{l\geq 2}\CT _{q}^{2l-1,-1}N(q,l-1,l)&\\%\label{FourierLqkN:q1}\\
 L_{q,-1}&=\sum_{l\geq 2}\CT _{q}^{2l-1,1}N(q, l,l-1)&\\%\label{FourierLqkN:qm1}\\
 L_{q,k}&=\sum_{l\geq k}\CT _{q}^{2l-k,-k}N(q,l-k,l)&\text{for }k\geq 2\\%\label{FourierLqkN:qk}\\
 L_{q,-k}&=\sum_{l\geq k}\CT _{q}^{2l-k,k}N(q,l,l-k)&\text{for }k\geq 2%\label{FourierLqkN:qmk}
\end{aligned}
\end{equation}
and
\begin{equation}
 N(q,m,n)=\frac{2^{m+n}}{\GO ^{2m+2n-1}} \binom{-1/2}{m}\binom{-1/2}{n}
 \int_{-\infty}^{\infty}\frac{ \EXP ^{iq (\tau+\tau^{3}\!/3) \,\GO ^{3}\!/2}}{(\tau-i)^{2m}(\tau+i)^{2n}}
 d\tau\label{Nqmn}
\end{equation}
\end{proposition}

\begin{proof}
We write Melnikov potential \eqref{cL} as:
\begin{equation}\label{LL1+I}
\mathcal{L}= \mathcal{L}_{\text{main}}+\int_{-\infty}^{\infty}\Biggl[\Bigl(\frac{x_{\HC}^{2}}{2} \Bigl)^{2}
\RO \cos (\alpha_{\HC} -f)-\frac{x_{\HC}^{2}}{2}\Biggl] \,dt ,
\end{equation}
where
$$
\mathcal{L}_{\text{main}}=\int_{-\infty}^{\infty}
\frac{x_{\HC}^{2}}{\bigl[4+x_{\HC}^{4}\RO ^{2}+4x_{\HC}^{2}\RO \cos (\alpha_{\HC}- f)\bigl]^{1/2}}\,dt
$$
can be written as
\begin{multline*}%\label{L1tilde}
\mathcal{L}_{\text{main}}=\int_{-\infty}^{\infty}\frac{x_{\HC}^{2}}{2}\left(1+\frac{x_{\HC}^{2}}{2}
\RO \left(f(t+\SO )\right)\EXP ^{i(\alpha_{\HC}-f(t+\SO ))}\right)^{-1/2}\\
\cdot \left(1+\frac{x_{\HC}^{2}}{2}
\RO \left(f(t+\SO )\right)\EXP ^{-i(\alpha_{\HC}-f(t+\SO ))}\right)^{-1/2}dt .
\end{multline*}
Using the expansion for $z=\dfrac{x_{\HC}^{2}}{2}
\RO \left(f(t+\SO )\right)\EXP ^{\pm i(\alpha_{\HC}-f(t+\SO ))}$
$$
(1+z)^{-\frac{1}{2}}=\sum_{l=0}^{\infty}\binom{-1/2}{l}z^l
$$
\kk{Write\\$|z|<1$?}
which, by~\eqref{homoclinic:h1}, \eqref{r0f}, is valid as long as $|z|=|x_{\HC}^{2}\RO/\!2|\leq 2(1+\EC)/G^2<1$, we get
\begin{equation*}%\label{L1}
\mathcal{L}_{\text{main}}=\sum_{k\geq 0}\sum_{l\geq k}\tilde L_{k}^{l}+\sum_{k< 0}\sum_{l\leq k}\tilde S_{k}^{l}
\end{equation*}
where
\begin{align*}
\tilde L_{k}^{l}&=\frac{1}{2^{2l-k+1}}\ \binom{-1/2}{l-k}\ \binom{-1/2}{l}\ \int_{-\infty}^{\infty}x_{\HC}^{4l-2k+2}\
[\RO (f(t+\SO ))]^{2l-k}\ \EXP ^{ik\alpha_{\HC}}\EXP ^{-ikf(t+\SO )}\ dt;\quad  0\leq k\leq l\\
\tilde S_{k}^{l}&=\frac{1}{2^{-2l+k+1}}\binom{-1/2}{k-l}\binom{-1/2}{-l}\int_{-\infty}^{\infty}x_{\HC}^{-4l+2k+2}
[\RO (f(t+\SO ))]^{-2l+k}\EXP ^{ik\alpha_{\HC}}\EXP ^{-ikf(t+\SO )}dt;\ \ l\leq k\leq -1 .
\end{align*}
With these expressions is easy to see that $\tilde L_0^0$ cancels out the last term in the integral \eqref{LL1+I}
and that $\tilde L_1^1+\tilde S_{-1}^{-1}$ cancels the cosine term, and so
\begin{equation}\label{L1tilde1}
\mathcal{L}=\sum_{l\geq 1}\tilde L_{0}^{l}+\sum_{l\geq 2}\tilde L_{1}^{l}+
\sum_{l\leq -2}\tilde S_{-1}^{l}+\sum_{k\geq 2}\sum_{l\geq k}\tilde L_{k}^{l}+\sum_{k\leq -2}\sum_{l\leq k}\tilde S_{k}^{l} .
\end{equation}
Now  we  perform the change of variable
\[
t=\frac{\GO ^3}{2}\Bigl(\tau+\frac{\tau^3}{3}\Bigl), \qquad
dt=\frac{\GO ^3}{2}(1+\tau^2)\, d\tau
\]
introduced in \eqref{cambiotau}, and we use the formulae for $x_{\HC}$ and
$\alpha_{\HC}$ given in \eqref{homoclinic:h1} and
\eqref{homoclinic:h2}. In particular we will use that
\[
x_{\HC}^2=\frac{4}{\GO ^2(1+\tau ^2)},  \qquad
x_{\HC}^2\, dt= 2\GO\, d \tau, \qquad
\EXP ^{i\alpha_{\HC}}=\frac{\tau-i}{\tau+i}\,\EXP ^{i\AO },
\]
and the expansion in Fourier series given in \eqref{sumapinyol} to obtain
\begin{align}
\tilde L_{k}^{l}&=\EXP ^{ik\AO }\frac{2^{2l-k}}{\GO ^{4l-2k-1}}\ \binom{-1/2}{l}\binom{-1/2}{l-k}
\sum_{q\in \mathbb{Z}}\EXP ^{iq\,\SO } \CT _{q}^{2l-k,-k} \int_{-\infty}^{\infty}
\frac{  \EXP ^{iq (\tau+\tau^{3}\!/3) \,\GO ^{3}\!/2}}{(\tau-i)^{2(l-k)} (\tau+i)^{2l} } d\tau
\nonumber\\
&= \EXP ^{ik\AO } \sum_{q\in \mathbb{Z}}\EXP ^{iq\,\SO } \CT _{q}^{2l-k,-k} N(q,l-k,l),
\qquad  0\leq k\leq l; \label{series:stilde1}\\
\tilde S_{k}^{l}&=\EXP ^{ik\AO }\frac{2^{-2l+k}}{\GO ^{-4l+2k-1}} \binom{-1/2}{-l}
\binom{-1/2}{k-l}\sum_{q\in \mathbb{Z}}\EXP ^{iq\,\SO }\CT _{q}^{-2l+k,-k}
\int_{-\infty}^{\infty}\frac{  \EXP ^{iq (\tau+\tau^{3}\!/3) \,\GO ^{3}\!/2}}{(\tau-i)^{-2l}(\tau+i)^{2(k-l)}}
d\tau\nonumber \\
&= \EXP ^{ik\AO } \sum_{q\in \mathbb{Z}}\EXP ^{iq\,\SO } \CT _{q}^{-2l+k,-k} N(q,-l,k-l),
\qquad l\leq k\leq -1. \label{series:stilde2}
\end{align}
Substituting now equations \eqref{series:stilde1} and  \eqref{series:stilde2} into the expansion \eqref{L1tilde1}
we get
\begin{align*}
 \mathcal{L}&=\ \ \sum_{q\in \mathbb{Z}}\EXP ^{iq\SO }\sum_{l\geq 1} \CT _{q}^{2l,0} N(q,l,l)
+\sum_{q\in \mathbb{Z}}\EXP ^{i(q\SO +\AO )} \sum_{l\geq 2}
\CT _{q}^{2l-1,-1} N(q,l-1,l)\\
&\phantom{=}+\sum_{q\in \mathbb{Z}}\EXP ^{i(q\SO -\AO )} \sum_{l\leq -2}
\CT _{q}^{-2l-1,1} N(q,-l,-l-1)\\
&\phantom{=}+\sum_{q\in \mathbb{Z}}\sum_{k\geq 2}\EXP ^{i(q\SO +k\AO )} \sum_{l\geq k}
\CT _{q}^{2l-k,-k} N(q,l-k,l)\\
&\phantom{=}+\sum_{q\in \mathbb{Z}}\sum_{k\leq -2}\EXP ^{i(q\SO +k\AO )} \sum_{l\leq k}
\CT _{q}^{-2l+k,-k} N(q,-l,k-l).
\end{align*}
Changing now the indexes $l\to -l$ and $k\to -k$ in the third and fifth terms we obtain the desired
formulae~\eqref{FourierLqkN} for the Fourier coefficients $L_{q,k}$.
\end{proof}

\subsection{General upper bounds}
\label{MPBounds}
In view of Proposition~\ref{RLFourierLq} and formulae~\eqref{FourierLqkN},
to compute the dominant part of the Melnikov potential and obtain
effective bounds of the errors we will need to estimate the constants
$\CT _{q}^{n,m}$ defined in \eqref{sumapinyol} and the integrals $N(q,m,n)$
defined in \eqref{Nqmn} for $q\ge 0$ and only for indexes $m,n$ satisfying $n\geq 0$, $m\leq n+1$.
Alternatively to~\eqref{r0f},
it will be very convenient to express the distance $\RO $ between the primaries as
\begin{equation}
 \RO =1-\EC \cos E\label{rEwint}
\end{equation}
in terms of the \textit{eccentric anomaly} $E$, given by the Kepler equation~\cite[p.~194]{Wintner41}
\begin{equation}
t=E-\EC \sin E\label{tEwint}.
\end{equation}
We obtain a general upper bound for the constants~$\CT _{q}^{n,m}$, where the
dominant order in $\EC $ appears explicitly.
\begin{proposition}\label{boundcqmn}
Let $n,m,q\in \mathbb{Z}$, $n,q\geq 0$, $m\leq n+1$.
Then the Fourier coefficients $\CT _{q}^{n,m}$ defined in \eqref{sumapinyol} satisfy
\begin{equation}
\label{Boundscqmn}
\left|\CT _{q}^{n,m}\right|\leq
\begin{cases}
2^{q+n+1}\EXP ^{q\sqrt{1-\EC ^2}}\EC ^{|m-q|} &\mbox{$m\geq 0$}\\
2^{n+1}\EXP ^{q\sqrt{1-\EC ^2}}\dfrac {\EC ^{q-m}}{(1-\EC ^2)^{-m/2}}
%(1+\EC )^{n+1}
&\mbox{$m\leq -1$}
\end{cases}
\end{equation}

\end{proposition}
\begin{proof}
In the integral formula for the Fourier coefficients
\begin{equation}\label{cqnmfourier}
 \CT _{q}^{n,m}=\frac{1}{2\pi}\int _{0}^{2\pi} \RO ^n \EXP ^{i m f}\EXP ^{-iqt}dt
\end{equation}
we change the variable of integration to the eccentric anomaly~\eqref{tEwint} ($dt=\RO \,dE$) to get
\begin{equation}\label{cqnmbechange}
 \CT _{q}^{n,m}=\frac{1}{2\pi}\int _{0}^{2\pi} \left(\RO \EXP ^{i f}\right)^{m} \RO ^{n+1-m}\EXP ^{-iqt}dE .
\end{equation}

To compute $ \CT _{q}^{n,m}$ from~\eqref{cqnmbechange} we will use the identity (see \cite[p.~202]{Wintner41})
\[
\left(\RO \EXP ^{i f}\right)^{\frac{1}{2}}= a \EXP ^{iE/2} - \frac{\EC }{2a} \EXP ^{-iE/2},
\qquad a=\frac{\sqrt{1+\EC}+\sqrt{1-\EC}}{2}
\]
which readily implies
\begin{subequations}\label{rtreifa}
\begin{gather}
\label{rtreifa1}
\RO \EXP ^{i f}=a^2\EXP ^{iE}-\EC +
\frac{\EC ^2}{4a^2}\EXP ^{-iE}=(a\EXP ^{iE/2}-\frac{\EC }{2a}\EXP ^{-iE/2})^2,\\
\label{rtreifa2}
a^2+\frac{\EC^2}{4a^2}=1, \quad a^2-\frac{\EC^2}{4a^2}=\sqrt{1-\EC^2},
\quad a^4 +\frac{\EC ^2}{16a^4}=1-\EC ^2, \quad
\quad a^4 -\frac{\EC ^2}{16a^4}=\sqrt{1-\EC^2}.
\end{gather}
\end{subequations}

To bound the integral~\eqref{cqnmfourier}  for $m\geq 0$ we will consider two different cases:
$0\leq q\leq m$ and $0\leq m < q$.
Let us first consider the case $ 0\leq q\leq m$.
By the analyticity and periodicity of the integral we change the path of integration from
$\Im( E)=0$ to $\Im E = \ln (2a^2/\EC )$
\begin{equation}\label{Eq:UpperPath}
E=u+i\ln \left(\frac{2a^2}{\EC }\right)\qquad u\in[0,2\pi]
\end{equation}
so that
$$
\EXP ^{iE}=\EXP ^{iu-\ln(2a^2\!/\EC )}=\frac{\EC }{2a^2}\EXP ^{iu}
$$
and then, by~\eqref{rEwint},~\eqref{tEwint} and~\eqref{rtreifa1},~\eqref{rtreifa2},
\begin{align*}
\RO \EXP ^{i f}&
=\frac{\EC }{2}\EXP ^{iu}-\EC +\frac{\EC }{2}\EXP ^{-iu}
=\EC (\cos u-1)\\%\label{rufu}\\
\RO&=1-\frac{\EC }{2}\left(\frac{\EC }{2a^2}\EXP ^{iu}+ \frac{2a^2}{\EC }\EXP ^{-iu}\right)
=1-\frac{\EC ^2}{4a^2}\EXP ^{iu}- a^2\EXP ^{-iu}\\
&=1-\left(\frac{\EC ^2}{4a^2}+a^2\right)\cos u +i \left(a^2-\frac{\EC ^2}{4a^2}\right)\sin u
=1-\cos u + i \sqrt{1-\EC ^2} \sin u \\%\label{ru}\\
\EXP ^{-it}&
%=\EXP ^{-i(E-\EC \sin E)}
=\frac{2a^2\EXP ^{-iu}}{\EC }
\exp\left(\frac{\EC ^2}{4a^2}\EXP ^{iu}- a^2 \EXP ^{-iu}\right)
 = \frac{2a^2\EXP ^{-iu}}{\EC }
 \exp\left(-\sqrt{1-\EC ^2} \cos u + i \sin u\right).
\end{align*}
Therefore along the complex path~\eqref{Eq:UpperPath} we have the following bounds
\begin{align*}
\left|\RO \EXP ^{i f}\right|&= \EC (1-\cos u)\leq 2 \EC\\
|\RO|&=\sqrt{\left(1-\cos u\right)^2+\left(1-\EC ^2\right)\sin^2 u}
=\sqrt{2\left(1-\cos u\right)-\EC ^2\sin^2 u}\leq 2\\
\left|\EXP ^{-it}\right|&=
\frac{2a^2}{\EC } \exp\left(-\sqrt{1-\EC ^2} \cos u \right)
\leq \frac{2a^2}{\EC } \EXP ^{\sqrt{1-\EC ^2}}.
\end{align*}
Since $2 a^2 \leq 2$, substituting these bounds in~\eqref{cqnmbechange} we find
directly the desired result~\eqref{Boundscqmn} for $0\leq q \leq m$.

For the the case $ m <q$ we now perform the change of the integration variable through
\begin{equation}\label{Eq:LowerPath}
E=v-i\ln \left(\frac{2a^2}{\EC }\right),\qquad v\in[0,2\pi]
\end{equation}
so that
$$
\EXP ^{iE}=\EXP ^{iv+\ln(2a^2\!/\EC)}=\frac{2a^2}{\EC }\EXP ^{iv}
$$
and then again, by~\eqref{rEwint},~\eqref{tEwint} and~\eqref{rtreifa1},~\eqref{rtreifa2},
\begin{align*}
\RO \EXP ^{i f}&=
\frac{2a^4}{\EC }\EXP ^{iv}-\EC +\frac{\EC ^3}{8a^4}\EXP ^{-iv}=
\frac{2}{\EC}\left( \left(a^4 + \frac{\EC ^4}{16 a ^4}\right)\cos v - \frac{\EC ^2}{2}
+i \left(a^4 - \frac{\EC ^4}{16 a ^4}\right)\sin v\right)\\
&=\frac{2}{\EC} \left(\cos v -\frac{\EC ^2}{2}(1 + \cos v) + i \sqrt{1-\EC ^2} \sin v\right)
\\%\label{rufu}\\
\RO&=1- a^2\EXP ^{iv}- \frac{\EC ^2}{4a^2}\EXP ^{-iv}
=1-\cos v - i \sqrt{1-\EC ^2} \sin v \\%\label{ru}\\
\EXP ^{-it}&
=\frac{\EC \EXP ^{-iv}}{2a^2}
\exp\left(a^2 \EXP ^{iv}- \frac{\EC ^2}{4a^2} \EXP ^{-iv}\right)
 = \frac{\EC \EXP ^{-iv}}{2a^2}
 \exp\left(\sqrt{1-\EC ^2} \cos v + i \sin v\right).
\end{align*}
Therefore
$$
|\RO \EXP ^{i f}|^2= \frac{2}{\EC}\left(1-\frac{\EC ^2 (\cos v+1)}{2}\right)
$$
and consequently, using that $2 a^2 \ge 1$,
along the complex path~\eqref{Eq:LowerPath} we have the following bounds
\begin{equation}\label{Eq:UpperAndLowerBounds}
\frac{2}{\EC}(1-\EC ^2)^{1/2}\leq \left|\RO \EXP ^{i f}\right|\leq \frac{2}{\EC}(1+\EC ^2)^{1/2} \le
\frac{4}{\EC}, \quad
|\RO|\leq 2, \quad
\left|\EXP ^{-it}\right| \leq \frac{\EC }{2a^2} \EXP ^{\sqrt{1-\EC ^2}}\le
\EC  \EXP ^{\sqrt{1-\EC ^2}}.
\end{equation}
%Since $2 a^2 \geq 1$,
Substituting the above upper bounds~\eqref{Eq:UpperAndLowerBounds} in~\eqref{cqnmbechange}
we find the desired result~\eqref{Boundscqmn} for $0\leq m <q$.
In the case $m\leq -1$ we use the above lower bounds for $|\RO \EXP ^{i f}|$ to get~\eqref{Boundscqmn}.

%bound directly the integral~\eqref{cqnmbechange} for $E\in[0,2\pi]$.
%Since $|\EXP ^{if}|=|\EXP ^{-it}|=1$ we have
%$$
%|\CT _{q}^{n,m}|\leq \frac{1}{2\pi}\int _{0}^{2\pi} |r|^{n+1} dE
%$$
%by noticing that $|r|\leq (1+\EC )$ we conclude the proof of the bounds for the $\CT _{q}^{n,m}$.
%The bounds for $q<0$ follow from the simmetries: $\CT _{-q}^{n,-m}=\CT _{q}^{n,m}=\overline{\CT }_{q}^{n,m}$ .
\end{proof}

As we can see from equations \eqref{FourierLqkN} the Fourier coefficients of the Melnikov potential
$\mathcal{L}$ depend also on the function $N(q,m,n)$
defined in \eqref{Nqmn}, so to bound (or to compute) these Fourier coefficients we need to bound
(or to compute) $N(q,m,n)$.
%that the next result, proved in appendix \ref{proofs} will be useful

Introducing the integral
\[
I(q,m,n)=\int_{-\infty}^{\infty}\frac{ \EXP ^{iq\GO ^{3}(\tau+\tau^{3}\!/3)\!/2}   }
{(\tau-i)^{2m}(\tau+i)^{2n}} d\tau
\]
$N(q,m,n)$ can be written as
\[
N(q,m,n)=\frac{2^{m+n}}{\GO ^{2m+2n-1}} \binom{-1/2}{m}\binom{-1/2}{n} I(q,m,n) .
\]
We will denote
\begin{equation}\label{htau}
 h(\tau)=i\Bigl(\frac{\tau^{3}}{3}+\tau\Bigl)
\end{equation}
the variable term in the exponencial of the integral, so that
\begin{equation}
I(q,m,n)=\int_{-\infty}^{\infty}\frac{\EXP ^{q\GO ^{3}h(\tau)\!/2}}{(\tau-i)^{2m}(\tau+i)^{2n}} d\tau
\label{Ia} .
\end{equation}

Since the integral $I(q,m,n)$ involves an exponential function with a large parameter $\GO ^{3}$ in front of the exponent,
we will apply the method of steepest descent~\cite[\S 2.5--6]{Erdelyi56}.
In particular on a complex path with $\Im(h(\tau))=0$.
So, let us define the path (see Figure~\ref{fig:Gamma}):
\begin{equation}\label{Gammapath}
\Gamma=\Gamma_{1}\cup\Gamma_{2}\cup\Gamma_{3}\cup\Gamma_{4}\cup\Gamma_{5}
\end{equation}
where $0<\varepsilon<1$, $\tau ^*$ is a point such that $\Im(h(\tau^*))=0$ that will be fixed in Lemma~\ref{Fmn1} as $|\tau^*-i|=1/2$, and
%$\varepsilon >0$ and $c$ is taken such that $c\geq 1$ and $c\varepsilon<1$ :
\begin{align}
 \Gamma_{1}&=\{\tau\in \mathbb{C}:\Im(h(\tau))=0\}\cap \{\tau\in \mathbb{C}: \Re(\tau)\leq \Re(-\bar\tau ^{*})\}\notag\\
 \Gamma_{5}&=\{\tau\in \mathbb{C}:\Im(h(\tau))=0\}\cap \{\tau\in \mathbb{C}: \Re(\tau)\geq \Re(\tau^{*})\}\notag\\
 \Gamma_{2}&=\{\tau\in \mathbb{C}:\Im(h(\tau))=0\}\cap \{\tau\in \mathbb{C}: \Re(-\bar\tau ^{*}) \leq \Re(\tau)\leq 0 \}\cap \{\tau\in \mathbb{C}: |\tau-i|\geq c\, \varepsilon\}\notag\\
 \Gamma_{4}&=\{\tau\in \mathbb{C}:\Im(h(\tau))=0\}\cap \{\tau\in \mathbb{C}: 0 \leq \Re(\tau)\leq \Re(\tau^{*})\}\cap \{\tau\in \mathbb{C}: |\tau-i|\geq c\,\varepsilon\}\notag\\
 \Gamma_{3}&=\{\tau\in \mathbb{C}:\Im(h(\tau))\leq 0\}\cap\{\tau\in \mathbb{C}: |\tau-i|= c\,\varepsilon\} .
\label{Gamma3}
\end{align}
\begin{wrapfigure}[17]{l}{70mm}
\includegraphics[width=65mm]{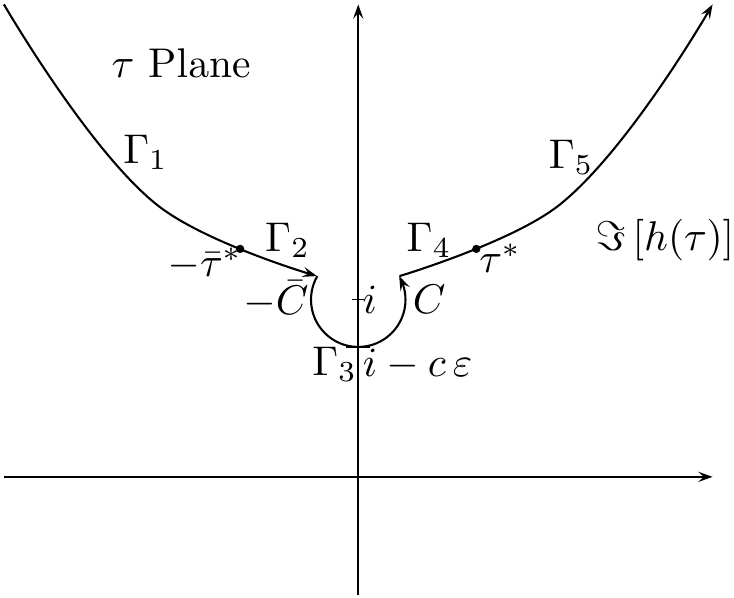}
\caption{The complex path $\Gamma$}
\label{fig:Gamma}
\end{wrapfigure}

By the Cauchy-Goursat Theorem plus a limit argument, the integral $I(q,m,n)$, defined in~\eqref{Ia}
over the real axis, is equal to the one taken over the path $\Gamma$ thinking of $\tau$ as a complex number
(see \cite{LlibreS80}).
In fact, by the same argument, its value depends neither on $\varepsilon$ nor on $\tau ^*$.

The positive branch of the hyperbola defined by $\Im(h(\tau))=0$ intersects the circumference of radius
$\varepsilon$ in two points $C_\eps$ and $-\overline{C_\eps}$ given by
\begin{equation}
\label{C}
C_\eps = \Gamma_3\cap\Gamma_4 \qquad
-\overline  C_\eps=\Gamma_3\cap\Gamma_2
\end{equation}

Since the integral over $\Gamma$ does not depend on $\varepsilon$, we will choose a particular value of
$\varepsilon$ to bound $I(q,m,n)$ and
consequently $N(q,m,n)$ defined in~\eqref{Nqmn}.
Later on, in Proposition~\ref{lema1refinado2N}, we will just compute the $\varepsilon$-independent terms of this integral.

It is not difficult to see that, if we define the function
\begin{equation}\label{cambiou}
 u(\tau)=h(i)-h(\tau)=-\frac{2}{3}-i\Bigl(\frac{\tau^{3}}{3}+\tau\Bigl)=(\tau - i)^2-\frac{i}{3}(\tau - i)^3,
\end{equation}
then
$$
 u(\Gamma_1\cup\Gamma_{2}),\, u(\Gamma_{4}\cup\Gamma_{5})\subset \mathbb{R}^+_0.
$$
Moreover, if $\tau^-\in\Gamma_1\cup\Gamma_{2}$ then $\tau^+=-\bar\tau^-\in \Gamma_{4}\cup\Gamma_{5}$ and
$$
u(\tau^-)=u(\tau^+).
$$
On the other hand one can see that $u$ is an increasing function while moving along
$\Gamma_1\cup\Gamma_{2}$ or $\Gamma_4\cup\Gamma_{5}$
in the direction of increasing imaginary part.
Therefore $u$ has two inverses in $\mathbb{R}_0^+$: $\tau^{+}$ and $\tau_{-}$.
Before writing them down let us
notice that the point $C_\eps$ defined in \eqref{C} can be written as
\begin{equation}
C_\eps=i+\varepsilon\,\EXP ^{i\theta_{\varepsilon}}\quad \text{with}\quad \theta_{\varepsilon}\in(0,\pi/2)
\label{Ce}
\end{equation}
and has the following expression in the  coordinates $u$ defined in \eqref{cambiou}
 \begin{equation}\label{ke}
 u(C_\eps)=|u(C_\eps)|=\varepsilon^{2}\,\bigl|1-\frac{\varepsilon}{3}i\EXP ^{i\theta_{\varepsilon}}\bigl|
=\varepsilon^{2}\,\ k_{\varepsilon}%\leq 2 \ \epsilon^{2}
\end{equation}
with
\[
 k_{\varepsilon}=\bigl|1-\frac{\varepsilon}{3}i\EXP ^{i\theta_{\varepsilon}}\bigl|
=\sqrt{\Bigl(1+\frac{\varepsilon}{3}\sin  \theta_{\varepsilon}\Bigl)^{2}
+\Bigl(\frac{\varepsilon}{3}\cos  \theta_{\varepsilon}\Bigl)^{2} }\geq 1 ,
\]
since by construction, $\theta_{\varepsilon}\in(0,\pi/2)$ and then $0<\sin \theta_{\varepsilon}$.

Now, we can write the inverses of the function $u$
\begin{align*}
\tau^{+}: &[u(C_\eps),+\infty)\longrightarrow\Gamma_{4}\cup\Gamma_{5}
&\tau^{-}:&[u(C_\eps),+\infty)\longrightarrow \Gamma_{1}\cup\Gamma_{2}\\
          &\qquad \quad u\longmapsto  \xi(u) + i\eta(u) ,
&         &\qquad \     u\longmapsto  -\xi(u) + i\eta(u) .
\end{align*}

The change \eqref{cambiou} is useful over $\Gamma_{1}\cup\Gamma_{2}$ and $\Gamma_{4}\cup\Gamma_{5}$,
thus performing this change
in \eqref{Nqmn}, we have that for any $\varepsilon > 0$
\begin{equation}\label{Su}
N(q,m,n)=\frac{d_{m,n}\EXP ^{-q\frac{\GO ^{3}}{3}}}{\GO ^{2m+2n-1}}
\Biggl[\int_{u(C_\eps)}^{\infty} [F_{m,n}^+(u)- F_{m,n}^-(u)]\EXP ^{-q\GO ^3 u\!/2}du
%-\int_{C(u)}^{\infty} {^-}F_{k}^{l}(u)\EXP ^{k\frac{\GO ^{3}}{2} u}du
+ (-i)\EXP ^{q\frac{\GO ^{3}}{3}}\int_{\Gamma_{3}}f_{m,n}^{q}(\tau)d\tau\Biggl]
\end{equation}
where
\begin{align}
 d_{m,n}&=i\,2^{m+n} \binom{-1/2}{n}\binom{-1/2}{m}\label{dmn}\\
 F_{m,n}^{\pm}(u)&=\frac{1}{(\tau^{\pm}(u)-i)^{2m+1}(\tau^{\pm}(u)+i)^{2n+1}}\label{Fmn}\\
f_{m,n}^{q}(\tau)&=\frac{\EXP ^{q\frac{\GO ^{3}}{2}h(\tau)}}{(\tau-i)^{2m}(\tau+i)^{2n}}, \label{fmn}
\end{align}
and $h(\tau)$ is given in \eqref{htau}. To give a bound for $N(q,m,n)$ given by~\eqref{Su}, some estimates for
$d_{m,n}$ and $F_{m,n}$ are needed. We begin with the constants $d_{m,n}$.
\begin{lemma}\label{binombound}
Let $m,n\in\mathbb{Z}$, $m,n\geq 0$ and $d_{m,n}$ be defined by equation \eqref{dmn}. Then
\[
|d_{m,n}|\leq \EXP ^{-1/2} 2^{m+n}\qquad \text{if $m+n>0$ }
\]
\end{lemma}
\begin{proof}
Let $s\in \mathbb{N}$, then
\begin{align*}
\biggl|\binom{-1/2}{s}\biggl|&= \biggl|\frac{(-1)^s}{s!}\Bigl(\frac{1}{2}\Bigl)
\Bigl(\frac{1}{2}+1\Bigl)\cdots\Bigl(\frac{1}{2}+s-1\Bigl)\biggl|
=\frac{1}{2^s}\Bigl[1\cdot \frac{3}{2}\cdots \frac{2s-1}{s}]\Bigl]\\
&\leq \frac{1}{2^s}\Bigl(2-\frac{1}{s}\Bigl)^s
=\Bigl(1-\frac{1}{2s}\Bigl)^s
\leq \lim_{s\to \infty}\Bigl(1-\frac{1}{2s}\Bigl)^{s}
=\EXP ^{-1/2}
\end{align*}
\end{proof}
The next Lemma gives information about the functions $F^\pm_{m,n}(u)$.
\begin{lemma}\label{constd}
The function $F^\pm_{m,n}(u)$ defined in~\eqref{Fmn} has the expansion
\begin{equation}\label{defd}
 F^\pm_{m,n}(u)=(\pm\sqrt{u})^{-2m-1}\sum_{j=0}^{\infty}d_j^{m,n}(\pm\sqrt{u})^{j}
\end{equation}
where the coefficients $d_j^{m,n}$ satisfy
\begin{equation}\label{coeficientsdj}
d_0^{m,n}=1/(2i)^{2n+1}, \quad |d_j^{m,n}| \le \left(\frac{4}{3}\right)^{m}\left(\frac{3}{2}\right)^{\frac{j+3}{2}} .
\end{equation}
%and in particular $d_0^{n,m}=1/(2i)^{2n+1}$.
Consequently, the series~\eqref{defd} is convergent for $|\sqrt{u}|<\sqrt{2/3}$.
%Equation \eqref{defd} defines the constants $d_j^{n,m}$,
\end{lemma}
\begin{proof}
Let us introduce the function
$$
T^\pm_{m,n}(x):=(\pm)x^{2m+1}F^\pm_{m,n}(x^2)=\sum _{j=0}^{\infty} d_j^{m,n}(\pm x)^j ,
$$
which is well defined since $u=x^2$ is a good change of variables in $\mathbb{R}^+$ and has the two inverses
$x=\pm \sqrt{u}$.
To bound the coefficients $d_j^{m,n}$ we use Cauchy formula:
$$
(\pm1)^jd_j^{m,n}= \frac{1}{2\pi i}\int_{|x|=\varepsilon} \frac{T^\pm _{m,n}(x)}{x^{j+1}}dx
=\frac{-1}{2\pi i}\int_{|x|=\varepsilon} \frac{F^\pm _{m,n}(x^2)}{x^{j-2m}}dx .
$$
Applying the  change of variables
\begin{equation}\label{canvixmesmenys}
x= \pm\sqrt{(\tau-i)^2-\frac{i}{3}(\tau-i)^3}=\pm (\tau-i)\sqrt{(1-\frac{i}{3}(\tau-i))}=
\pm\frac{\tau-i}{\sqrt{3}}(\sqrt{2-i\tau}),
\end{equation}
we obtain
\begin{eqnarray*}
(\pm1)^jd_j^{m,n}&=&
\mp \frac{1}{2\pi i}\int_{|\tau-i|=\rho} \frac{(\tau -i)^{2m-j}}{3^{\frac{2m-j}{2}}} (2-i\tau)^{\frac{2m-j}{2}}
\frac{1}{(\tau -i)^{2m+1}(\tau +i)^{2n+1}}
\frac{3(1-i\tau)}{2\sqrt{3}(2-i\tau)^{\frac{1}{2}}}d\tau\\
&=&\mp \frac{1}{2}\frac{i^{\frac{j+1}{2}-m}}{2\pi 3^{m-\frac{j+1}{2}}}
\int_{|\tau-i|=\rho} \frac{d\tau }{ (\tau -i)^{j+1} (\tau+i)^{2n}(\tau +2i)^{\frac{j+1-2m}{2}  }   } .
\end{eqnarray*}
Now, taking $\rho=1$ and using that $|\tau +i|\ge 1$ and that $2\le |\tau+2i|\le 4$ we have
$$
|d_j^{m,n}| \le \left(\frac{4}{3}\right)^{m}\left(\frac{3}{2}\right)^{\frac{j+3}{2}} ,
$$
which is the desired bound.
From this bound it is clear that the series defining $T^\pm_{m,n}(x)$ is convergent for $|x|<\sqrt{2/3}$ and
therefore the one for
$F_{m,n}^\pm(u)$ is convergent for $\sqrt{u}<\sqrt{2/3}$.
\end{proof}

From equation \eqref{defd} we have
\begin{equation}\label{defdg}
 F^\pm_{m,n}(u)=(\pm\sqrt{u})^{-2m-1}\sum_{j=0}^{2m}d_j^{m,n}(\pm\sqrt{u})^{j} + g^{\pm}_{m,n}(\pm\sqrt{u}),
\end{equation}
where the regular part of the function $F^\pm_{m,n}(u)$ is given by
\begin{equation}\label{gmn}
g^{\pm}_{m,n}(\pm\sqrt{u})=(\pm\sqrt{u})^{-2m-1}\sum_{j=2m+1}^{\infty}d_j^{m,n}(\pm\sqrt{u})^j
\end{equation}
and $d_j^{m,n}$ are defined by equation \eqref{defd} and satisfy bounds \eqref{coeficientsdj}.
The next Lemma bounds $g^{\pm}_{m,n}$ inside its domain of convergence.
\begin{lemma}\label{gmnbound}
 Let $g^{\pm}_{m,n}(\pm\sqrt{u})$ as in equation \eqref{gmn}, $0<\beta<1$ and $0<\sqrt{u}<\beta\sqrt{2/3}$.
Then
$$
\left|g^{\pm}_{m,n}(\pm\sqrt{u})\right|< \frac{9}{1-\beta}2^{m-2}.
$$
\end{lemma}
\begin{proof}
It is clear from equation \eqref{gmn} that
 \begin{equation*}
g^{\pm}_{m,n}(\pm\sqrt{u})=\sum_{s=0}^{\infty}d_{s+2m+1}^{m,n}(\pm\sqrt{u})^{s} .
\end{equation*}
Since by hypothesis $0<\sqrt{u}<\beta\sqrt{2/3}$ with $\beta <1$, we can apply Lemma~\ref{constd} to get
\begin{align*}
|g^\pm_{m,n}(\pm\sqrt{u})|&\leq
\left(\frac{4}{3}\right)^m \left(\frac{3}{2}\right)^{\frac{2m+4}{2}}
\sum_{s=0}^{\infty}\left(\frac{3}{2}\right)^{\frac{s}{2}}(\sqrt{u})^{s}\\
&\leq  \left(\frac{4}{3}\right)^m \left(\frac{3}{2}\right)^{\frac{2m+4}{2}}
\sum_{s=0}^{\infty}\left(\frac{3}{2}\right)^{\frac{s}{2}}\left(\beta \sqrt{2/3}\right)^{s}
=\frac{9}{1-\beta}2^{m-2}
\end{align*}
which proves the Lemma.
\end{proof}

We are now in conditions to give a general bound for $N(q,m,n)$ for $q\geq 1$.

\begin{proposition}\label{boundN}
Let $N(q,m,n)$ as defined in \eqref{Su} for $q\geq 1$, $m,n\geq 0$, $m+n> 0$, $\GO > 1$ . Then
$$
|N(q,m,n)|\leq 2^{n+m+3}\, \EXP ^{q} \, \GO ^{m-2n-1/2} \, \EXP ^{-q\GO ^{3}\!/3} .
$$
\end{proposition}

\begin{proof}
We will bound the integrals of $N(q,m,n)$ in~\eqref{Su} choosing
$$
\varepsilon=\GO ^{-3/2},
%\quad c=1,
\quad \GO >1 .
$$
%with $\GO >c^{2/3}$.
We write down then, using \eqref{Ce}, \eqref{ke}, \eqref{Fmn},  and that $k_\eps >1$,
\begin{align}
&\left|\int_{u(C_\eps)}^{\infty}F^\pm_{m,n}(u)\EXP ^{-q\GO ^{3}\!u\!/2}\ du \right|
\leq \int_{\GO ^{-3}k_{\epsilon}}^{\infty}|F^\pm_{m,n}(u)|\EXP ^{-q\GO ^{3}\!u\!/2}\ du
%\quad\text{(by means of \eqref{ke})}\notag\\
\leq \int_{\GO ^{-3}            }^{\infty}     |F^\pm_{m,n}(u)|\EXP ^{-q\GO ^{3}\!u\!/2}\ du
%\quad\text{($k_{\epsilon}\geq 1$)}
\notag\\
&\leq  \left|F^\pm_{m,n}(u(C_\eps))\right|\int_{\GO ^{-3}            }^{\infty}   \EXP ^{-q\GO ^{3}\!u\!/2}\ du
%\notag\\&
\leq \frac{\GO ^{3m+\frac{3}{2}}}{\left(2-(\GO ^{-3/2})\right)^{2n+1}}\frac{2\EXP ^{-q/2}}{q\GO ^3}
%\notag\\&
\leq 2 \GO ^{3m-3/2}. \label{shorties}
\end{align}

It only remains to bound the last integral of \eqref{Su} where the integrand is given in \eqref{fmn} and the domain
$\Gamma_3$ in \eqref{Gamma3}.
The path $\Gamma_3$ can be parameterized by
\begin{equation}\label{cambiottheta}
\tau=i+\,\GO ^{-3/2}\rm{e}^{i\theta}\quad \text{with } \theta\in[\theta_1,\theta_2]=[\pi-\theta_\eps,\theta_\eps],
\end{equation}
with $\theta _\eps $ given in \eqref{Ce}.
If we define
$$
\tilde h(\theta)=h(\tau(\theta))=i\left(\frac{\tau(\theta)^3}{3}+\tau(\theta)\right),
$$
a straightforward computation using \eqref{cambiou} shows that
$$
\tilde h(\theta)=-\frac{2}{3}-\GO ^{-3}\Bigl(\EXP ^{2i\theta}+\frac{1}{3i}\GO ^{-\frac{3}{2}}\EXP ^{3i\theta} \Bigl)
$$
and then, as $\GO >1$,
\begin{align}
\bigl| \EXP ^{q\GO ^3\tilde h(\theta)\!/2}\bigl|&=\EXP ^{-q\GO ^3\!/3}
\EXP ^{-q\left(\cos 2\theta+\GO ^{-3/2}\sin 3\theta\!/3\right)\!/2}
%\notag\\&
\leq \EXP ^{-q\GO ^3\!/3}\EXP ^{\frac{q}{2}(1+\frac{1}{3}\GO ^{-3\!/2})}
%\notag\\&
\leq \EXP ^{-q\GO ^3\!/3}\EXP ^{q}.
%\quad \text{(for $\GO \geq (c^2/9)^{1/3}$ )
\label{cotaexptheta}
\end{align}
Note that, by  \eqref{cambiottheta}, over $\Gamma_3$ we have that $|\tau -i|= \GO ^{-3\!/2}<1$
and therefore $|\tau+i|>1$,
and we can bound the last integral of \eqref{Su} using  \eqref{cotaexptheta}:
\begin{align}
&\left|
\int_{\Gamma_3}\frac{\EXP ^{q\GO ^{3}h(\tau)\!/2}}{(\tau-i)^{2m}(\tau+i)^{2n}}\, d\tau \right|
= \left|\int_{\theta_1}^{\theta_2}
\frac{\EXP ^{q\GO ^{3}\tilde h(\theta)\!/2}}{(\tau(\theta)-i)^{2m}(\tau(\theta)+i)^{2n}}
i\,\GO ^{-3\!/2}\EXP ^{i\theta}\, d\theta\right|\notag\\
&\leq \int_{\theta_1}^{\theta_2}
\frac{\left|\EXP ^{q\GO ^{3}\tilde h(\theta)}\!/2\right|}{(\GO ^{-3/2})^{2m}}\GO ^{-3\!/2}\, d\theta
%\quad\text{(by \eqref{cambiottheta} and \eqref{cotaaux})}\notag\\&
\leq \int_{\theta_1}^{\theta_2}\frac{\EXP ^{-q\GO ^3\!/3}\EXP ^{q}}{\GO ^{-3m}}\GO ^{-3\!/2}\, d\theta
%\quad \text{(by \eqref{cotaexptheta})}\notag\\
%&\leq K\gamma^{\max\{m,n\}}\GO ^{3m-\frac{3}{2}}\rm{e}^{-\frac{q}{3}\GO ^3}\rm{e}^{qc^2}\label{rondie}
%&\leq 2\pi \GO ^{3m-3/2}\EXP ^{-\frac{q}{3}\GO ^3}\EXP ^{q}\notag \\&
\leq \pi \GO ^{3m-3/2}\EXP ^{-q\GO ^3\!/3}\EXP ^{q}.\label{rondie}
\end{align}

From Lemma \ref{binombound} and the bounds \eqref{shorties} and \eqref{rondie},
we can finally bound $N(q,m,n)$ given by equation \eqref{Su} as follows
\begin{align*}
|N(q,m,n)|&\leq \EXP ^{-1/2}  2^{m+n}\EXP ^{-q\GO ^{3}\!/3}\GO ^{m-2n-1\!/2} \Bigl(4 +\pi\EXP ^{q}\Bigl)
%\\
%&\leq 2\pi \EXP ^{-1/2} \EXP ^{qc^2} 2^{m+n}\EXP ^{-q\frac{\GO ^{3}}{3}}\GO ^{-2m-2n+1}\Bigl( \Bigl(\frac{2}{3}\Bigl)^{\min\{m,n\}}\GO ^{3m-3/2}+ \GO ^{-3} + \GO ^{3m-3/2}\Bigl)\\
%&&\leq 2\pi \EXP ^{-1/2} \EXP ^{qc^2} 2^{m+n}\EXP ^{-q\frac{\GO ^{3}}{3}}\GO ^{m-2n-\frac{1}{2}}\bigl(\Bigl(\frac{2}{3}\Bigl)^{\min\{m,n\}}+ 1 + 1\bigl)\\&
\leq   2^{m+n+3}\EXP ^{q}\EXP ^{-q\GO ^{3}\!/3}\GO ^{m-2n-1\!/2}.
\end{align*}
\end{proof}

From this Proposition and the one estimating the constants $\CT _{q}^{n,m}$, we can provide general estimates for
the Fourier coefficients $L_{q,k}$ of the Melnikov potential for $q\geq 1$.
\begin{proposition}\label{boundLqkbis}
Assume $\GO \geq 32$. Then for $q\ge 1$,  $k\geq 2$, the Fourier coefficients
of the Melnikov potential~\eqref{eqn:LFourier1} verify the following bounds:
\begin{align*}
 |L_{q,0}|&\leq 2^{9}  \left(2 \EXP ^{2}\right)^{q} \EC ^q \, \GO ^{-3/2}\,\EXP ^{-q \GO ^3\!/3} \\
 |L_{q,1}|&\leq 2^{11}  \EXP ^{2q}  \frac{\EC ^{q+1}}{\sqrt{1-\EC ^2}} \, \GO ^{-7/2} \EXP ^{-q \GO ^3\!/3}\\
 |L_{q,- 1}|&\leq  2^{9} \left(2 \EXP ^{2}\right)^{q}\, \EC ^{q-1}\, \GO ^{-1/2} \EXP ^{-q \GO ^3\!/3}\\
 |L_{q,k}|&\leq   2^{2k+5} \EXP ^{2q} \frac{\EC ^{q+k}}{(\sqrt{1-\EC ^2})^k}\, \GO ^{-2k-1/2}  \EXP ^{-q \GO ^3\!/3}\\
 |L_{q,-k}|&\leq 2^5  2^{2k} \left(2 \EXP ^{2}\right)^{q} \EC ^{|k-q|} \GO ^{k-1/2} \EXP ^{-q \GO ^3\!/3}
\end{align*}

\end{proposition}
\begin{proof}
 From equations \eqref{FourierLqkN} and Propositions \ref{boundcqmn} and \ref{boundN} we have
\begin{align*}
 |L_{q,0}|&\leq  2^4 \EXP ^{q} \EXP ^{-q \GO ^3\!/3} (2 \EC \EXP ^{\sqrt{1-\EC ^2}})^{q}  \GO ^{-1/2}
 \sum_{l\geq 1} (2^4 \GO ^{-1})^{l}\\
 |L_{q,1}|&\leq 2^2 \EXP ^{q} \EXP ^{q\sqrt{1-\EC ^2}}\frac{\EC ^{q+1}}{\sqrt{1-\EC ^2}}
 \EXP ^{-q \GO ^3\!/3}  \GO ^{-3/2}
 \sum_{l\geq 2} ( 2^4 \GO ^{-1})^{l}\\
 |L_{q,- 1}|&\leq  \EXP ^{q} \EXP ^{-q \GO ^3\!/3} \EXP ^{q\sqrt{1-\EC ^2}} 2^q \EC ^{|1-q|} \GO ^{3/2}
 \sum_{l\geq 2} (2^4 \GO ^{-1})^{l}\\
 |L_{q, k}|&\leq  2^{4-k} \EXP ^{q}\EXP ^{q\sqrt{1-\EC ^2}}
\frac{\EC ^{q+k}}{(\sqrt{1-\EC ^2})^k}
 \EXP ^{-q \GO ^3\!/3}  \GO ^{-k-1/2}  \sum_{l\geq k} ( 2^4 \GO ^{-1})^{l}\\
 |L_{q,-k}|&\leq 2^4 2^{-2k}  \EXP ^{q} \EXP ^{-q \GO ^3\!/3} \EXP ^{q\sqrt{1-\EC ^2}} 2^q \EC ^{|k-q|}
 \GO ^{2k-1/2} \sum_{l\geq k} (2^4 \GO ^{-1})^{l}.
\end{align*}
Since by hypothesis $2^4 /\GO \leq 1/2$, all these series converge
and the Proposition is proven using that  $0\le \EC \leq 1$ and that $\EXP ^{q\sqrt{1-\EC ^2}}\le \EXP ^{q}$.
\end{proof}

The Melnikov potential $\mathcal{L}$~\eqref{cL} has a Fourier Cosine series~\eqref{eqn:LFourier1} which can be split
with respect to the variable $\SO $ as $\mathcal{L}=\mathcal{L}_0+\mathcal{L}_1+\mathcal{L}_2+\cdots$, like
in~(\ref{eqn:LFourier2}-\ref{eqn:LFourier3}), as well as
a complex Fourier series~\eqref{LFourierLq} $\mathcal{L}=\sum_{q\in\mathbb{Z}}L_{q}\EXP ^{iq\SO }$.
Both series are related through $\mathcal{L}_0=L_{0}$ and
$\mathcal{L}_q=2\Re\left\{\EXP ^{iq\SO } L_{q}\right\}$ for $q\geq 1$. In the next Lemma we see
that the terms
\[\mathcal{L}_{\geq 2}(\AO ,\GO ,\SO ;\EC)= 2\sum_{q\geq 2}\sum_{k\in \mathbb{Z}}L_{q,k}\cos(q\SO +k\AO )
=\mathcal{L}_2+\mathcal{L}_3+\mathcal{L}_4+\cdots
\]
of second order with respect to $\SO $
satisfy a very exponentially small bound for large $\GO $.

\begin{lemma}\label{boundSumLqk}
Assume $\GO \geq 32$, $ \EC  \GO \leq 1/8$. Then for $q\geq 2$
\begin{align*}
 |L_q|\leq \sum_{k\in \mathbb{Z}}|L_{q,k}|&\leq
%K \EXP ^{-q\GO ^3\frac{2}{9}} \biggl[ 2^{3q}  \EXP ^{q\sqrt{1-\EC ^2}}\GO ^{q-1/2}\biggl]
2^{13} \EXP ^{-q \GO ^3\!/3} (\EXP ^{2} 2^{3} \GO ) ^{q}\GO ^{-1/2}\\
\left|\mathcal{L}_{\geq 2}(\AO ,\GO ,\SO ;\EC)\right|&\leq
2^{28}\GO ^{3/2}\EXP ^{-2\GO ^3\!/3}.
\end{align*}
\end{lemma}

\begin{proof}
From Proposition~\ref{boundLqkbis} we have, using that $\dfrac{\EC }{\sqrt{1-\EC ^2}}\leq 1$:
 \begin{align*}
\sum_{k\in \mathbb{Z}}|L_{q,k}|&\leq
|L_{q,0}|+|L_{q,1}|+|L_{q,-1}|+\sum_{k\geq 2}\bigl(|L_{q,k}|+|L_{q,-k}|\bigl)\\
&\leq  \EXP ^{-q \GO ^3\!/3} \EXP ^{2q}
\biggl[2^9 2^q \EC ^q \GO ^{-3/2}
+2^{11} \EC ^q \GO ^{-7/2}
+2^{9} 2^{q}\EC ^{q-1} \GO ^{-1/2}\\
& +
2^{5}\sum_{k\geq 2}\Bigl(2^k
\EC ^q \GO ^{-2k-1/2}+
2^{2k+q} \EC ^{|k-q|}\GO ^{k-1/2}\Bigl)\biggl]\\
&\leq
\EXP ^{-q \GO ^3\!/3}\EXP ^{2q}
\biggl[
2^{10} 2^q \EC ^{q-1}\GO ^{-1/2}
+2^4 \EC ^q \GO ^{-7/2}
+ 2^5 \GO ^{-1/2}\EC ^q\sum_{k=2}^{\infty}(2\GO ^{-2})^k\\
&+ 2^5 \GO ^{-1/2} 2^{q}\EC ^q\sum_{k=2}^{q-1}(4\GO \EC ^{-1})^k
+2^5 \GO ^{-1/2} \EC ^{-q} 2^{q}\sum_{k=q}^{\infty}(4\EC \GO )^k\biggl].
 \end{align*}
Using now that $\EC \GO \leq 1/8$, we obtain the required bound for
$\sum_{k\in \mathbb{Z}}|L_{q,k}|$:
 \begin{align*}
  \sum_{k\in \mathbb{Z}}|L_{q,k}|
&\leq
\EXP ^{-q \GO ^3\!/3}\EXP ^{2q}
\biggl[ 2^{10} 2^q \EC ^{q-1}\GO ^{-1/2}
+2^{11} \EC ^q \GO ^{-7/2}
+ 2^{8}\EC ^q \GO ^{-9/2}\\
&+ 2^4 2^{3q} \EC \GO ^{q-3/2}
+2^6  \GO ^{q-1/2}2^{3q}\biggl]\\
&\leq
2^{10} \EXP ^{-q \GO ^3\!/3} \EXP ^{2q} 2^{3q} \GO ^{q-1/2}
 \biggl[ 2^{-2q}\EC ^{q-1}\GO ^{-q} +2^{-3q} \EC ^q \GO ^{-3-q}\\
&+ 2^{-3q-4} \EC ^q \GO ^{-4-q} +\frac{1}{2^6}
\EC \GO ^{-1} +\frac{1}{4}\biggl]
\leq
2^{13} \EXP ^{-q \GO ^3\!/3} (\EXP ^{2} 2^{3} \GO ) ^{q}\GO ^{-1/2}.
\end{align*}

To get the bound for $\left|\mathcal{L}_{\geq 2}\right|$, we sum for $q\geq 2$,
\begin{align*}
\left|\mathcal{L}_{\geq 2}\right|&\leq 2^{13} \GO ^{-1/2} \sum_{q\geq 2} \biggl[
 \EXP ^{-\GO ^3\!/3} \EXP ^{2} 2^{3} \GO \biggl]^{q}\leq
2^{20}\EXP ^{4}\GO ^{3/2}\EXP ^{-2\GO ^3\!/3}
\end{align*}
where the last bound holds as long as
\begin{equation*}%\label{condextra}
 \EXP ^{-2\GO ^3\!/3} \EXP ^{2} 2^{3} \GO \leq 1/2
\end{equation*}
which is true for every $\GO \geq 32$. Now, using that $\EXP <4$ we get the result.
\end{proof}

\subsection{Aymptotic estimate for $N(q,m,n)$}
\label{MPAsymp}
To estimate the term $\mathcal{L}_{1}$ we will need an asymptotic expression for $N(q,m,n)$, which is given in
the next Proposition.
\begin{proposition}\label{lema1refinado2N}
For $n+m>0$ let $d_j^{m,n}$ the constants $d_j^{m,n}$ defined by equation \eqref{defd} and $d_{n,m}$ given by equation \eqref{dmn}.
Then for $q\geq 1$ and  $\GO  > 1$ we have
$$
N(q,m,n)=\frac{d_{m,n}\EXP ^{-q\GO ^{3}\!/3}}{\GO ^{2m+2n-1}}
\Biggl[\sum_{s=0}^{m}(-1)^s\sqrt{\pi}\frac{2^{3\!/2}q^{s-1\!/2}}{(2s-1)!!}d_{2m-2s}^{m,n}
\GO ^{3s-3\!/2}+T^q_{m,n}+R^q_{m,n}\Biggl]
%N(q,m,n)=\frac{d_{m,n}\EXP ^{-q\frac{\GO ^{3}}{3}}}{\GO ^{2m+2n-1}}\Biggl[2d_{2m}^{n,m}(\sqrt{q\delta})^{-1}\Gamma\bigl(\frac{1}{2}\bigl)+\sum_{s=1}^{m}\frac{d_{2m-2s}^{n,m}2(\sqrt{q\delta})^{2s-1}}{(-s-\frac{1}{2}+1)(-s-\frac{1}{2}+2)\cdots(-\frac{1}{2})}\Gamma\Bigl(\frac{1}{2}\Bigl)+E_T+E_R\Biggl]
$$
where
\begin{equation*}
|T^q_{m,n}| \leq 45 \, 2^{2m+2}\cdot \GO ^{-3}
%K \gamma_4^m \GO ^{-3}
\qquad |R^q_{m,n}|\leq 18  \, q^{m-1}\GO ^{3m-3} .
\end{equation*}
When $s=0$ the factor $1/(2s-1)!!$ in the formula above should be replaced by $1$.
\end{proposition}

To prove this Proposition we will proceed as in the proof of Proposition~\ref{boundN} changing the path of integration to the path
$\Gamma$ defined in \eqref{Gammapath} leading to the integral \eqref{Su}.
The important fact is that the integral \eqref{Su} does not depend on $\varepsilon$.
So, we will compute only the $\varepsilon$-independent terms of that integral.
The rest of this subsection is dedicated to the proof of Proposition~\ref{lema1refinado2N}.

\begin{lemma}\label{Fmn1}
For $0<\varepsilon< 1$ let $u(C_\eps)$ be as in equation \eqref{ke} and $F_{m,n}^{\pm}$ as defined by \eqref{Fmn}.
% and $0<\beta<1$, then if $u^{**}=2/3$ as given in Lemma \ref{Cauchyest} and  $0<\sqrt{u}<\sqrt{u_*}=\beta \sqrt{u^{**}}<\sqrt{u^{**}}$
For any $\varepsilon>0$ small enough we have, if $\GO  > 1$
 \begin{equation*}
\int_{u(C_\eps)}^{\infty} F_{m,n}^{\pm}(u)\EXP ^{-q\GO ^3 u\!/2}du
=\sum_{j=0}^{2m}\int_{u(C_\eps)}^{\infty}\EXP ^{-q\GO ^3 u\!/2}d_{j}^{m,n}(\pm\sqrt{u})^{-2m-1+j}du
+ \widehat E
%\widehat E_1 + \widehat E_2+ \widehat E_3
\end{equation*}
where the constants $d_j^{m,n}$ are defined by equation \eqref{defd} and $\widehat E$ satisfies
$$
|\widehat E|\leq 45 \, 2^{2m+2}\, \GO ^{-3} .
$$
\end{lemma}

\begin{proof}
Let us take $\sqrt{u_*}=\beta \sqrt{2/3}$ with $\beta= -1+\frac{\sqrt{11}}{4}\sqrt{3+\sqrt{11}/2}\simeq 0.79$.
A simple calculation using \eqref{cambiou} shows that $|\tau ^\pm(u^*)-i|=1/2$.
By definition, for $\varepsilon>0$ small enough we have that $0<u(C_\eps)<u_*<\sqrt{u_*}<\sqrt{2/3}$, so
$$
\int_{u(C_\eps)}^{\infty} F_{m,n}^{\pm}(u)\EXP ^{-q\GO ^3 u\!/2}du=
\int_{u(C_\eps)}^{u_*} F_{m,n}^{\pm}(u)\EXP ^{-q\GO ^3 u\!/2}du+\widehat E_1
$$
with
$$
\widehat E_1=\int_{u_*}^{\infty} F_{m,n}^{\pm}(u)\EXP ^{-q\GO ^3 u\!/2}du,
$$
which can be bounded as
%since $0<u_*<u^*$ by Lemma \ref{r} and
%by the triangle inequality
\begin{align*}
|\widehat E_1|&=\left|\int_{u_*}^{\infty}F^\pm_{m,n}(u)\EXP ^{-q\GO ^{3}\!u\!/2}\ du \right|
\leq \int_{u_*}^{\infty}\frac{\EXP ^{-q\GO ^{3}\!u\!/2}}{|(\tau^{\pm}(u)-i)^{2m+1}(\tau^{\pm}(u)+i)^{2n+1}|}\ du \\
&\leq \frac{2\EXP ^{-q\frac{\GO ^{3}}{2}u_*}}{q\GO ^{3}}\frac{1}{|\tau^{\pm}(u_*)-i|^{2m+1}}
\frac{1}{|\tau^{\pm}(u_*)+i|^{2n+1}} \\&
\leq  2^{2m+2} \GO ^{-3}\EXP ^{-q\frac{\GO ^{3}}{2}u_*}
%\Bigl( \frac{1}{\rho}\Bigl)^{2m+1}\notag\\&
\leq  2^{2m+2}\GO ^{-3}.
\end{align*}
By Lemma \ref{constd} and equation \eqref{defdg} we have
$$
\int_{u(C_\eps)}^{u_*} F_{m,n}^{\pm}(u)\EXP ^{-q\GO ^3 u\!/2}du=
\sum_{j=0}^{2m}\int_{u(C_\eps)}^{u_*}d_j^{m,n}\EXP ^{-q\GO ^3 u\!/2}(\pm\sqrt{u})^{-2m-1+j}du+\widehat E_2
$$
where
$$
\widehat E_2=\int_{u(C_\eps)}^{u_*}g^{\pm}_{m,n}(\pm\sqrt{u})\EXP ^{-q\GO ^3 u\!/2}du.
$$
Using that $\sqrt{u_*}=\beta \sqrt{2/3}$, by Lemma \ref{gmnbound} we have that, for any $\varepsilon>0$ small enough,
\begin{align*}
|\widehat E_2|&\leq \int_{u(C_\eps)}^{u_*}|g^{\pm}_{m,n}(\pm\sqrt{u})|\EXP ^{-q\GO ^3 u\!/2}du
%\\&
\leq 9 \, \frac{2^{m-2}}{1-\beta}\int _{0}^{\infty}\EXP ^{-q\GO ^3 u\!/2}du\\
&\leq   9 \, \frac{2^{m-1}}{q(1-\beta)}  \GO ^{-3}
%\\
%&\leq  \frac{2}{(1-\beta)}\Bigl(\frac{4}{3}\Bigl)^{m}(\sqrt{\frac{2}{3}})^{-2m}\GO ^{-3}\\&
\leq   9 \, \frac{2^{m-1}}{1-\beta}  \GO ^{-3}.
%&\leq 2 \gamma_1^{m}(r\beta)^{-2m-1}\frac{1}{(1-\beta)}\GO ^{-3}
\end{align*}

Finally,
\begin{equation*}%\label{cu*infty}
\int_{u(C_\eps)}^{u_*}d_j^{m,n}\EXP ^{-q\GO ^3 u\!/2}(\pm\sqrt{u})^{-2m-1+j}du=
 \int_{u(C_\eps)}^{\infty}d_j^{m,n}\EXP ^{-q\GO ^3 u\!/2}(\pm\sqrt{u})^{-2m-1+j}du+\widehat E_3 (j),
\end{equation*}
where
$$
\widehat E_3(j)=- \int_{u_*}^{\infty}d_j^{m,n}\EXP ^{-q\GO ^3 u\!/2}(\pm\sqrt{u})^{-2m-1+j}du.
$$
%observe that $E_3(m,s)$ is independent of $\varepsilon$.
We can bound $\widehat E_3(j)$ thanks to the inequalities of Lemma \ref{constd}:
\begin{align*}
 |\widehat E_3(j)|&\leq |d_{j}^{m,n}|(\sqrt{u_*})^{-2m-1+j}\int_{u_*}^{\infty}\EXP ^{-q\GO ^3 u\!/2}du
 %\\&
 \leq |d_{j}^{m,n}|(\sqrt{u_*})^{-2m-1+j}2\EXP ^{-q\frac{\GO ^{3}}{2} u_*}\frac{\GO ^{-3}}{q}\\
&\leq 2|d_{j}^{m,n}|\left(\beta\sqrt{2/3}\right)^{-2m-1+j}\GO ^{-3}
%\\&
\leq 2\left(\frac{4}{3}\right)^m \left(\frac{3}{2}\right)^{ \frac{j+3}{2} }
\left(\beta\sqrt{2/3}\right)^{-2m-1+j}\GO ^{-3}
%\qquad \text{(by Lemma \ref{Cauchyest})}\\
%&\leq 2\frac{1}{(\sqrt{u^{**}})^{2m}} \sqrt{2/3} \Bigl(\frac{4}{3}\Bigl)^m\beta^{-2s-1}(\sqrt{u^{**}})^{-1}\GO ^{-3}
\\
&
=9 \,  2^{m-1}\beta^{-2m-1+j}\GO ^{-3}.
\end{align*}
Denoting now $\widehat E_3= \sum _{j=1}^{2m} \widehat E_3(j)$, we have
\begin{align*}
|\widehat E_3|
%=&\sum_{j=0}^{2m}\int_{u_*}^{\infty}d_j^{m,n}\EXP ^{-q\GO ^3 u\!/2}(\pm\sqrt{u})^{-2m-1+j}du
\le 9 \, 2^{m-1}\GO ^{-3} \sum_{j=0}^{2m} \beta^{-2m-1+j}\le  9 \, 2^{m-1}\GO ^{-3} \frac{ \beta^{-2m-1}}{1-\beta}.
\end{align*}
Now the Lemma is proven using that $1/\beta<\sqrt{2}$ and
$$
|\widehat E|=|\widehat E_1+\widehat E_2+\widehat E_3|
$$
\end{proof}

The next Lemma is a straightforward application of the last one.
\begin{lemma}\label{Fmn2}
For $0<\varepsilon< 1$ let $u(C_\eps)$ be as in equation \eqref{ke} and $F_{m,n}^{\pm}$ as in \eqref{Fmn}.
Then for any $\varepsilon>0$ small enough we have, if $\GO  > 1$
%Let $0<\varepsilon< 1$ and $u(C_\eps)$ be as in equation \eqref{Su}, $F_{m,n}^{\pm}$ defined by \eqref{Fmn} and $0<\beta<1$, then if $u^{**}$ is %given by
%Lemma \ref{Cauchyest} and  $0<\sqrt{u}<\sqrt{u_*}=\beta \sqrt{u^{**}}<\sqrt{u^{**}}$
%for any $\varepsilon>0$ small enough we have
 \begin{equation*}
\int_{u(C_\eps)}^{\infty} \bigl[F_{m,n}^{+}(u)-F_{m,n}^{-}(u)\bigl]\EXP ^{-q\GO ^3 u\!/2}du =
2\sum_{s=0}^{m}\int_{u(C_\eps)}^{\infty}\EXP ^{-q\GO ^3 u\!/2}d_{2m-2s}^{m,n}(\sqrt{u})^{-2s-1}du
+2\widehat E
%_1 +2\widehat E_2 +2\widehat E_3,
\end{equation*}
where $\widehat E
%_1$,  $\widehat E_2$ and $\widehat E_3
$ is the same as in Lemma \ref{Fmn1}.
\end{lemma}
\begin{proof}
By Lemma \ref{Fmn1} we have
\begin{multline*}
\int_{u(C_\eps)}^{\infty} \bigl[F_{m,n}^{+}(u)-F_{m,n}^{-}(u)\bigl]\EXP ^{-q\GO ^3 u\!/2}du\\
=\sum_{j=0}^{2m}\int_{u(C_\eps)}^{\infty}\EXP ^{-q\GO ^3 u\!/2}d_{j}^{m,n}[1-(-1)^{-2m-1+j}](\sqrt{u})^{-2m-1+j}du
+ 2\widehat E
%_1 + 2\widehat E_2+ 2\widehat E_3
\end{multline*}
and the terms in the sum are not zero only for $-2m-1+j=-2s-1$ with $s=0,\dots,m$.
This observation proves the Lemma.
\end{proof}

\begin{lemma}\label{Fmn3}
Let $0<\varepsilon< 1$ and $u(C_\eps)$ be as in equation \eqref{Su}.
%, $0<\beta<1$, then if $u^{**}=2/3$ as given in Lemma \ref{Cauchyest} and  $\sqrt{u_*}=\beta \sqrt{u^{**}}<\sqrt{u^{**}}$,
Then the $\varepsilon$-independent term of
\begin{equation*}
 \int_{u(C_\eps)}^{\infty}\EXP ^{-q\GO ^3 u\!/2}d_{2m-2s}^{m,n}(\sqrt{u})^{-2s-1}du
\end{equation*}
is
\begin{equation*}
%\frac{d_{2m-2s}^{n,m}(\sqrt{q\delta})^{2s-1}}{(-s-\frac{1}{2}+1)(-s-\frac{1}{2}+2)\cdots(-\frac{1}{2})}\Gamma\bigl(\frac{1}{2}\bigl)+E_3\\
(-1)^s2^{s+3\!/2}(2s+1)\frac{(s+1)!}{(2s+2)!}d_{2m-2s}^{m,n}q^{s-1\!/2}\GO ^{3s-3\!/2}\,\Gamma(1/2)
%+\widehat E_3(m,s)
\end{equation*}
%with
%$$|\widehat E_3(m,s)|\leq  2^{m+1}\beta^{-2s-1}\GO ^{-3}.$$
\end{lemma}
\begin{proof}
% First we write down
%\begin{equation}\label{cu*infty}
 %\int_{u(C)}^{u_*}\EXP ^{-q\GO ^3 u\!/2}d_{2m-2s}^{n,m}(\sqrt{u})^{-2s-1}du= \int_{u(C)}^{\infty}\EXP ^{-q\GO ^3 u\!/2}d_%{2m-2s}^{n,m}(\sqrt{u})^{-2s-1}du+\widehat E_3
%\end{equation}
%where$$\widehat E_3(m,s)=- \int_{u_*}^{\infty}\EXP ^{-q\GO ^3 u\!/2}d_{2m-2s}^{n,m}(\sqrt{u})^{-2s-1}du$$
%observe that $E_3(m,s)$ is independent of $\varepsilon$. Let us bound $\widehat E_3(m,s)$
%\begin{align*}
% |\widehat E_3(m,s)|&\leq |d_{2m-2s}^{n,m}|(\sqrt{u_*})^{-2s-1}\int_{u_*}^{\infty}\EXP ^{-q\GO ^3 u\!/2}du\\
%&\leq |d_{2m-2s}^{n,m}|(\sqrt{u_*})^{-2s-1}2\EXP ^{-q\frac{\GO ^{3}}{2} u_*}\frac{\GO ^{-3}}{q}\\
%&\leq 2|d_{2m-2s}^{n,m}|(\beta\sqrt{u^{**}})^{-2s-1}\GO ^{-3}\\
%&\leq 2\frac{1}{(\sqrt{u^{**}})^{2m-2s}} \sqrt{2/3} \Bigl(\frac{4}{3}\Bigl)^m(\beta\sqrt{u^{**}})^{-2s-1}\GO ^{-3}\qquad \text{(by Lemma \ref{Cauchyest})}\\
%&\leq 2\frac{1}{(\sqrt{u^{**}})^{2m}} \sqrt{2/3} \Bigl(\frac{4}{3}\Bigl)^m\beta^{-2s-1}(\sqrt{u^{**}})^{-1}\GO ^{-3}\\
%&=2^{m+1}\beta^{-2s-1}\GO ^{-3},
%\end{align*}
%\begin{comment}
%On the other hand, integrating by parts
%\begin{multline*}
%\int_{u(C)}^{\infty}\EXP ^{-q\GO ^3 u\!/2}d_{2m-2s}^{n,m}(\sqrt{u})^{-2s-1}du=\\
%d_{2m-2s}^{n,m}\Biggl[-\frac{1}{-s-\frac{1}{2}+1}(C(u))^{\frac{1}{2}(-2s-1)+1}\EXP ^{-q\delta C(u)}+
%\frac{q \delta}{-s-\frac{1}{2}+1}\int_{C(u)}^{\infty}\EXP ^{-q\delta u}u^{\frac{1}{2}(-2s-1)+1}\ du\Biggl]
%\end{multline*}
%\end{comment}

By equation \eqref{ke} we know that $u(C_\eps)=O( \varepsilon^2)$ and then the following definitions make sense,
calling $\delta =\GO ^{3}\!/2$:
\begin{align*}
I_{p,s}(\varepsilon)&=\int_{u(C_\eps)}^{\infty} \EXP ^{-q\delta u}\,u^{p-(2s+1)\!/2}\, du\\
f_{p,s}(\varepsilon)&=u(C_\eps)^{p-(2s+1)\!/2}\,\EXP ^{-q\delta u(C_\eps)}.
\end{align*}
Using this notation and integrating by parts we have
\begin{align}\label{recIp}
I_{p-1,s}(\varepsilon)&=\frac{q\delta}{p-s-1/2}\,\int_{u(C_\eps)}^{\infty}\EXP ^{-q\delta u}u^{p-(2s+1)\!/2}\, du-\frac{u(C_\eps)^{p-(2s+1)\!/2}\,\EXP ^{-q\delta u(C_\eps)}}{p-s-1/2}\notag\\
&=\frac{1}{p-s-1/2}\ \left(q\delta I_{p,s}(\varepsilon)-f_{p,s}(\varepsilon)\right)
\end{align}
and also
\begin{equation}\label{CinftyI0}
\int_{u(C_\eps)}^{\infty}\EXP ^{-q\GO ^3 u\!/2}d_{2m-2s}^{m,n}(\sqrt{u})^{-2s-1}du=d_{2m-2s}^{m,n}I_{0,s}(\varepsilon) .
\end{equation}
Now, for $s> 0$ we use recursively equation \eqref{recIp} $s$ times to get
\begin{align*}
I_{0,s}(\varepsilon)=&\frac{(q\delta)^{s}}{(-s-1\!/2+1)(-s-1\!/2+2)\cdots(-1\!/2)}I_{s,s}(\varepsilon)\\
&-\sum_{p=1}^{s}\frac{(q\delta)^{p-1}f_{p,s}(\varepsilon)}{(-s-1\!/2+1)\cdots(-s-1\!/2+p)} .
\end{align*}
The $\varepsilon$-independent term of $I_{0,s}(\varepsilon)$ is given by
\begin{multline*}
 \frac{(q\delta)^{s}}{(-s-1\!/2+1)(-s-1\!/2+2)\cdots(-1\!/2)}\lim_{\varepsilon\to 0}I_{s,s}(\varepsilon)\\
                 =\frac{(q\delta)^{s}}{(-s-1\!/2+1)(-s-1\!/2+2)\cdots(-1\!/2)}\frac{1}{\sqrt{q\delta}}\  \Gamma(1/2)\\
                 =\frac{(\sqrt{q\delta})^{2s-1}}{(-s-1\!/2+1)(-s-1\!/2+2)\cdots(-1\!/2)}\, \Gamma(1/2).
\end{multline*}
Then the $\varepsilon$-independent term of the integral in equation \eqref{CinftyI0} is
$$
\frac{d_{2m-2s}^{m,n}(\sqrt{q\delta})^{2s-1}}{(-s-1\!/2+1)(-s-1\!/2+2)\cdots(-1\!/2)}\ \Gamma(1/2)
$$
when $s>0$.

In the same way, we have that the $\varepsilon$-independent term of
$$
I_{0,0}(\varepsilon)=\int_{u(C_\eps)}^{\infty}\EXP ^{-q\GO ^3 u\!/2}d_{2m}^{m,n}(\sqrt{u})^{-1}du
$$
is $d_{2m}^{m,n}(\sqrt{q\delta})^{-1}\,\Gamma(1/2)$.
Therefore the Lemma is proved if we notice that
\begin{align*}
& (-s-1\!/2+1)(-s-1\!/2+2)\cdots(-1\!/2)
%=\frac{1}{2^s}(-2s+1)(-2s+3)\cdots(-1)\\
%&
=\frac{(-1)^s}{2^s}(2s-1)(2s-3)\cdots 1\\
&=\frac{(-1)^s}{2^s}\frac{(2s+1)!!}{2s+1}=\frac{(-1)^s}{2^{2s+1}(2s+1)}\frac{(2s+2)!}{(s+1)!}\end{align*}
where we have used that
$$
(2s+1)!!=\frac{(2s+2)!}{2^s(s+1)!} \, .
$$
%we get
%$$\bigr(-s-1\!/2+1\bigr)\bigr(-s-1\!/2+2\bigr)\cdots\bigr(-1\!/2\bigr)=\frac{(-1)^s}{2^{2s+1}(2s+1)}\frac{(2s+2)!}{(s+1)!}$$
This expression allow us to write the cases $s>0$ and $s=0$ in one equation which completes the proof.
\end{proof}

Next Lemma is a straightforward application of Lemmas \ref{Fmn2} and \ref{Fmn3}.
\begin{lemma}\label{Fmn4}
Let $u(C_\eps)$ given in equation \eqref{ke} and $F_{m,n}^{\pm}$ defined by \eqref{Fmn}, then
the $\varepsilon$-independent terms of
$$
\int_{u(C_\eps)}^{\infty} \bigl[F_{m,n}^{+}(u)-F_{m,n}^{-}(u)\bigl]\EXP ^{-q\GO ^3 u\!/2}du
$$
are given by
$$
%2d_{2m}^{n,m}(\sqrt{q\delta})^{-1}\,\Gamma(1/2)+\sum_{s=1}^{m}\frac{2d_{2m-2s}^{n,m}(\sqrt{q\delta})^{2s-1}}{(-s-1\!/2+1)(-s-1\!/2+2)\cdots(-1\!/2)}\,\Gamma(1/2)+ E_T\\
\sum_{s=0}^{m}(-1)^s 2^{s+5\!/2}(2s+1)\frac{(s+1)!}{(2s+2)!}d_{2m-2s}^{m,n}q^{s-1\!/2}\GO ^{3s-3\!/2}\,\Gamma(1/2)+2 \widehat E
%T^q_{m,n}
$$
where $\widehat E$ is the same as in Lemma \ref{Fmn1}.
%$$
%|T^q_{m,n}|\leq 9\, (4\sqrt{2}+\frac{2}{1-\beta}) (\frac{2}{\beta^2})^m\GO ^{-3}
%$$and $\beta\simeq 0.79$ is given in Lemma \ref{Fmn1}.
%$$
%\beta=\biggl(-1+\frac{\sqrt{11}}{4}\sqrt{3+\frac{\sqrt{11}}{2}}\biggl)^{1/2},
%$$
\end{lemma}

\begin{lemma}\label{residuo}
Let $f^q_{m,n}$ be defined in equation \eqref{fmn}, then
 \begin{equation}\label{eresiduo}
\mathrm{Res}(f^q_{m,n}(\tau),i)=  2i\, \EXP ^{-q\GO ^3\!/3}
\sum_{l=0}^{m-1}\frac{1}{l!}\left(\frac{-q\GO ^3}{2}\right)^l d_{2m-1-2l}^{m,n}
\end{equation}
\end{lemma}
\begin{proof}
We use the definition of $f^q_{m,n}$ given in \eqref{fmn}, with $h(\tau)$ given in \eqref{htau},
or equivalently by \eqref{cambiou}.
$$
h(\tau)=-2/3-(\tau-i)^{2}+i(\tau-i)^{3}/3
$$
Taking any $\delta>0$ small enough, we have
$$
\mathrm{Res}(f^q_{m,n}(\tau),i)=
\frac{1}{2\pi i} \int_{|\tau -i|=\delta} f^q_{m,n}(\tau)d\tau=
\frac{1}{2\pi i} \int_{|\tau -i|=\delta} \frac{\EXP ^{q\frac{\GO ^3}{2}h(\tau)}}
{(\tau-i)^{2m}(\tau+i)^{2n}}\,d\tau .
$$
We use again one of the changes \eqref{canvixmesmenys}, for instance
$$
x = \sqrt{h(i)-h(\tau)}
%\sqrt{-\frac{2}{3}-h(\tau)}= (\tau-i)\sqrt{(1-\frac{i}{3}(\tau-i))}
=\frac{\tau-i}{\sqrt{3}}(\sqrt{2-i\tau}),
$$
to obtain
\begin{align*}
\mathrm{Res}\left(f^q_{m,n}(\tau),i\right)
&=
\frac{\EXP ^{-q\GO ^3\!/3}}{\pi } \int_{|x|=\bar \delta} \frac{\EXP ^{-q\GO ^3 x^2\!/2}}
{(\tau_+(x)-i)^{2m+1}(\tau_+(x)+i)^{2n+1}}\, x\, d\tau\\
&=
2i\,\EXP ^{-q\GO ^3\!/3}\, \mathrm{Res}\left(xF^{n,m}_+(x^2)\EXP ^{-q\frac{\GO ^3}{2}x^2},0\right).
\end{align*}
We can now use the Taylor expansion of the function $F^{n,m}_+(x^2)=\sum _{j\geq 0} d_j^{m,n}x^{j-2m-1}$
and the expansion of
$\EXP ^{-q\GO ^3x^2\!/2}= \sum _{l\ge 0}\left(-q\GO ^3x^2\!/2\right)^l\!/l!$ to obtain the desired
formula~\eqref{eresiduo}.

\end{proof}

From this Lemma one and the bounds for  $d_j^{m,n}$ given in \eqref{coeficientsdj}, we have
\begin{align}
\left|\mathrm{Res}\left(f^q_{m,n}(\tau),i\right)\right|&\leq
3 \, 2^m\EXP ^{-q \GO ^3\!/3}\sum _{l=0}^{m-1}\frac{1}{l!}\left(\frac{q\GO ^3}{3}\right)^{l}\notag\\
&\leq
3 \, 2 ^{m+1}\EXP ^{-q \GO ^3\!/3}\left(\frac{q\GO ^3}{3}\right)^{m-1} =
\frac{2^{m+1} q^{m-1}\GO ^{3m-3}}{3^{m-2}}\EXP ^{-q \GO ^3\!/3} .
\label{fitares}
\end{align}

We are finally in conditions to prove Proposition \ref{lema1refinado2N}. $N(q,m,n)$ is given in \eqref{Su}, and since it does not depend on $\varepsilon$
we can apply Lemmas \ref{Fmn4} and \ref{residuo} and the bound above \eqref{fitares} to obtain
\begin{equation*}
N(q,m,n)=\frac{d_{m,n}\EXP ^{-q\frac{\GO ^{3}}{3}}}{\GO ^{2m+2n-1}}
\Biggl[\sum_{s=0}^{m}(-1)^s 2^{s+\frac{5}{2}}(2s+1)\frac{(s+1)!}{(2s+2)!}d_{2m-2s}^{m,n}q^{s-1\!/2}\GO ^{3s-\frac{3}{2}}\,\Gamma(1/2)
 +T^q_{m,n}+R^q_{m,n}\Biggl]
\end{equation*}
where by Lemma \ref{Fmn4}
$$
|T^q_{m,n}|=2 \widehat E   \leq 45 \, 2^{2m+2}\cdot \GO ^{-3} .
%   9\, (4\sqrt{2}+\frac{2}{1-\beta}) (\frac{2}{\beta^2})^m\GO ^{-3},
$$
and
$$
R^q_{m,n}=(-i)\EXP ^{q\frac{\GO ^{3}}{3}}\int_{\Gamma_{3}}f_{m,n}^{q}(\tau)d\tau
$$
is bounded by Lemma \ref{residuo}
$$
|R^q_{m,n}|\leq \frac{2^{m+1} q^{m-1}}{3^{m-2}}\GO ^{3m-3} <18 \, q^{m-1} \GO ^{3m-3} .
$$
Using that
$2^{s+1}(s+1)!(2s+1)!!=(2s+2)! $ to show that
$$
\frac{(2s+1)(s+1)!}{(2s+2)!}=\frac{1}{2^{s+1}(2s-1)!!}
$$
the formula for $N(q,m,n)$ of Proposition \ref{lema1refinado2N} follows. Due to the fact that the right hand side
of this last expression is not defined for $s=0$
but the left hand side is and is equal to one, we need to point out that for $s=0$,
the term $1/(2s-1)!!$ in the final formula should be replaced by $1$.

\subsection{Asymptotic estimate of $\mathcal{L}_1$}
\label{MP1}
\textr{Let us first compute the coefficients $\CT _1^{n,m}$ which enter in the dominant terms of $\mathcal{L}_1$, more precisely $\CT _1^{3,1}$, $\CT _1^{2,2}$ and $\CT _1^{3,3}$}.
In passing, we will also compute $\CT _0^{2,0}$ and $\CT _0^{3,1}$, which will enter in the dominant terms of $\mathcal{L}_0$.
\begin{lemma}\label{ccalc}
 Let $\CT _q^{n,m}$ be defined by \eqref{sumapinyol}. Then
\[
 \CT _1^{3,1}=1+Q_1,\quad
 \CT _1^{2,2}=-3\EC +Q_2,\quad\
\CT _0^{2,0}=1+Q_3,\quad
\CT _0^{3,1}=-\frac{5}{2}\, \EC +Q_4,\quad
\textr{\CT _1^{3,3}=\frac{57}{8}\, \EC ^2 +Q_5,}
\]
with
$$
|Q_i|\le 98 \EC ^2, \quad i=1,2,3,4, \quad \textr{|Q_5|\le 98 \EC ^3,}
$$
%\begin{align*}
% |Q_1|&\leq 98 \EC ^2\\|Q_2|&\leq 50\EC ^2\\|Q_3|&\leq 4\EC ^2\\|Q_4|&\leq 19 \EC ^2\end{align*}
\end{lemma}
\begin{proof}
 From the definition given in \eqref{sumapinyol} plus the change of variable $t=E-\EC \sin E$ we have
\begin{align*}
 \CT _1^{3,1}&=\frac{1}{2\pi}\int _{0}^{2\pi} \left(\RO \EXP ^{if(E)}\right)\RO ^{3}\EXP ^{-it}dE,&
\CT _1^{j,j}&=\frac{1}{2\pi}\int _{0}^{2\pi} \left(\RO \EXP ^{if(E)}\right)^j\RO \EXP ^{-it}dE,\quad j=2,3\\
\CT _0^{2,0}&=\frac{1}{2\pi}\int _{0}^{2\pi} \RO ^3dE,&
\CT _0^{3,1}&=\frac{1}{2\pi}\int _{0}^{2\pi} \left(\RO \EXP ^{if(E)}\right)\RO ^3dE .
\end{align*}
 From equations \eqref{rtreifa} we have
\begin{align}
 \CT _1^{3,1}&=\frac{1}{2\pi}\int _{0}^{2\pi} [a^2\EXP ^{iE}-\EC +\frac{\EC ^2}{4a^2}\EXP ^{-iE}](1-\EC  \cos E)^{3}\EXP ^{-iE}\EXP ^{i\EC \sin E}dE\label{c131calc}\\
\CT _1^{j,j}&=\frac{1}{2\pi}\int _{0}^{2\pi} [a^2\EXP ^{iE}-\EC +\frac{\EC ^2}{4a^2}\EXP ^{-iE}]^j(1-\EC  \cos E)\EXP ^{-iE}\EXP ^{i\EC \sin E}dE ,\quad j=2,3\label{c122calc}\\
\CT _0^{2,0}&=\frac{1}{2\pi}\int _{0}^{2\pi} (1-\EC  \cos E)^{3}dE\label{c020calc}\\
\CT _0^{3,1}&=\frac{1}{2\pi}\int _{0}^{2\pi} [a^2\EXP ^{iE}-\EC +\frac{\EC ^2}{4a^2}\EXP ^{-iE}](1-\EC  \cos E)^{3}dE \, .\label{c031calc}
\end{align}

In what follows we will use (see~\eqref{rtreifa2}) that
\begin{equation}\label{desis}
 0\leq \EC  \leq 1,\qquad
 \frac{1}{2}\leq a^2 \leq 1, \qquad
 |a^2-1|
 %&=\bigl|\frac{\sqrt{1-\EC ^2}-1}{2}\bigl|=\bigl|\frac{-\EC ^2}{2(\sqrt{1-\EC ^2}+1)}\bigl|
 \leq \frac{\EC ^2}{2},\qquad
a^2+\frac{\EC ^2}{4a^2}=1 .
\end{equation}

To compute $\CT _1^{3,1}$ we use equation \eqref{c131calc}. It is easy to see that
%\begin{subequations}\label{E_123}
%\begin{gather}\notag%\label{E_123a}
\begin{equation}\label{E_123b}
\begin{array}{rcl}
a^2\EXP ^{iE}-\EC +\frac{\EC ^2}{4a^2}\EXP ^{-iE}&=&\EXP ^{iE}-\EC +\bar E_1,\\
 (1-\EC \cos E)^3=1-3\EC  \cos E + \bar E_2, \quad
\EXP ^{i\EC \sin E}&=&1+i\EC \sin E+\bar E_3,
\end{array}
\end{equation}
%\end{gather}
%\end{subequations}
where
\[
\bar E_1= (a^2-1)\EXP ^{iE}+\frac{\EC ^2}{4a^2}\EXP ^{-iE},\qquad
 \bar E_2= 3\EC ^2\cos^2 E- \EC ^3\cos^3 E,\qquad
 \bar E_3=\frac{1}{2}\sum_{j=0}^{\infty}2\frac{(i\EC \sin E)^{j+2}}{(j+2)!} ,
\]
satisfy
\[
\left|\bar E_1\right|\leq \frac{\EC ^2}{2}+\frac{\EC ^2}{2}=\EC ^2,\qquad
\left|\bar E_2\right|\leq 4\EC ^2,\qquad
\left|\bar E_3\right|\leq \frac{\EC ^2}{2}\EXP ^{\EC }\leq \EC ^2 \frac{\EXP }{2}\leq 2\EC ^2 .
\]
From the equation \eqref{c131calc} defining $\CT _1^{3,1}$ plus equations \eqref{E_123b}, $\CT _1^{3,1}$ is the Fourier coefficient of order 1 of the function
\begin{multline*}
(\EXP ^{iE}-\EC +\bar E_1)(1-3\EC  \cos E + \bar E_2)(1+i\EC \sin E+\bar E_3)=\\
\EXP ^{iE}-\EC -3\EC \cos E \EXP ^{iE}+i\EC \sin E \EXP ^{iE}+\tilde Q_1(E)
\end{multline*}
where
\begin{align*}
\tilde Q_1(E)=&\bar E_1-3\EC ^2\cos E-3\EC \bar E_1\cos E+\bar E_2(\EXP ^{iE}-\EC +\bar E_1)
-i\EC ^2\sin E-3i\EC ^2\cos E\sin E \EXP ^{iE}\\
&-3i\EC ^3\cos E\sin E-3i\EC ^2\sin E\cos E \bar E_2+i\EC \sin E\bar E_2(\EXP ^{iE}-\EC +\bar E_1)\\
&+\bar E_3 (\EXP ^{iE}-\EC + \bar E_1-3\EC \cos E \EXP ^{iE}-3\EC ^2\cos E-3\EC \bar E_1\cos E+\bar E_2(\EXP ^{iE}-\EC +\bar E_1)),
\end{align*}
which implies that, up to order one in $\EC $, the Fourier coefficient $\CT _1^{3,1}$ is exactly $1$.
From the bounds for $\bar E_1$, $\bar E_2$ and $\bar E_3$ we find
$|\tilde Q_1(E)|\leq 98 \EC ^2$, which implies the Lemma for $ \CT _1^{3,1}$.

From equation~\eqref{E_123b}, it is easy to see that
\begin{equation*}%\label{E_456}
 \left(a^2\EXP ^{iE}-\EC +\frac{\EC ^2}{4a^2}\EXP ^{-iE}\right)^2=\left(\EXP ^{iE}-\EC +\bar E_1\right)^2=\EXP ^{2iE}-2\EC \EXP ^{iE}+\bar E_4
\end{equation*}
where
\begin{align*}
 \bar E_4&= \EC ^2 + 2\bar E_1(\EXP ^{iE}-\EC )+\bar E_1^2
\end{align*}
can be bounded, in regard of equations \eqref{desis} and the bound for $\tilde E_1$, as
\begin{align*}
 |\bar E_4|&\leq \EC ^2+2\EC ^2(1+\EC )+\EC ^4\leq 6 \EC ^2.
\end{align*}
Using equation \eqref{E_123b}, we see from equation \eqref{c122calc} that $\CT _1^{2,2}$ is the Fourier coefficient of order 1 of the function
\begin{multline*}
(\EXP ^{2iE}-2\EC \EXP ^{iE}+\bar E_4)(1-\EC \cos E)(1+i\EC \sin E+\bar E_3)=\\\EXP ^{2iE}-\EC \cos E\EXP ^{2iE}-2\EC \EXP ^{iE}+i\EC \sin E\EXP ^{2iE}+\tilde Q_2(E)
\end{multline*}
where
\begin{align*}
 \tilde Q_2(E)&=2\EC ^2\cos E\EXP ^{iE}+\bar E_4-\EC \bar E_4\cos E\\
&=i\EC \sin E(- \EC \cos E\EXP ^{2iE}-2\EC \EXP ^{iE}+2\EC ^2\cos E\EXP ^{iE}+\bar E_4 -\EC \bar E_4\cos E)\\
&=\bar E_3( \EXP ^{2iE}-\EC \cos E\EXP ^{2iE}-2\EC \EXP ^{iE}+2\EC ^2\cos E\EXP ^{iE}+\bar E_4 -\EC \bar E_4\cos E) .
\end{align*}
From the above expressions we conclude that, up to order one in $\EC $, the Fourier coefficient
$c_1^{2,2}$ is exactly $-3\EC $,
and from the bounds for $\bar E_4$ and $\bar E_3$ we find that $|\tilde Q_2(E)|\leq 50\EC ^2$
which implies the Lemma for $\CT _1^{2,2}$.

\textr{An analogous reasoning gives the value and the bounds for $\CT _1^{3,3}$ using
\[
\left(a^2\EXP ^{iE}-\EC +\frac{\EC ^2}{4a^2}\EXP ^{-iE}\right)^3=
\frac{15}{4}a^2\EC ^2\EXP ^{iE}-3a^4\EC \EXP ^{2iE}+ a^6 \EXP ^{3iE}+\tilde E_{4,1}, \quad |\tilde E_{4,1}|\le 8\EC^3
\]
and
\[
(1-\EC \cos E)\EXP ^{i\EC \sin E}
=1-\frac{\EC^2}{4}-\EC \EXP^{-i E}-\frac{\EC^2}{8}\EXP^{2iE}+\frac{3\EC^2}{8}\EXP^{-2iE} +\tilde E_{4,2}, \quad |\tilde E_{4,2}|\le \frac{3}{2}\EC^3
\]
which give
$$
\CT _1^{3,3}=\frac{15}{4}a^2\EC^2(1-\frac{\EC^2}{4})+3a^4\EC^2+\frac{3}{8}a^6\EC^2+ \tilde E_{4,3} ,\quad |\tilde E_{4,3}|\le 56 \EC^3
$$
Now, using \eqref{desis} we obtain the value for $\CT_1^{3,3}$.}

We compute $\CT _0^{2,0}$ using equation \eqref{c020calc}, as well as equation~\eqref{E_123b} to get
$$
\CT _0^{2,0}=\frac{1}{2\pi}\int _{0}^{2\pi} (1-3\EC  \cos E + \bar E_2) dE=1+Q_3
$$
with
$$
Q_3=\frac{1}{2\pi}\int _{0}^{2\pi} \bar E_2 dE
$$
and we have immediately, using the bound for $\bar E_2$, that $|Q_3|\leq 4\EC ^2$, the desired result for $\CT _0^{2,0}$.

We finally compute $\CT _0^{3,1}$ using equation~\eqref{c031calc}, as well as equations~\eqref{E_123b}
$$
\CT _0^{3,1}=\frac{1}{2\pi}\int _{0}^{2\pi} (\EXP ^{iE}-\EC +\bar E_1)(1-3\EC \cos E+\bar E_2)dE .
$$
We now want to find, up to order $\EC $, the Fourier coefficient of order zero of the function
\begin{align*}
 (\EXP ^{iE}-\EC +\bar E_1)(1-3\EC \cos E+\bar E_2)= \EXP ^{iE}-3\EC \EXP ^{iE}\cos E -\EC +\bar E_5 ,
\end{align*}
where
$$
\bar E_5=\bar E_2\EXP ^{iE}+3\EC ^2\cos E -\EC \bar E_2+\bar E_1-3\EC \bar E_1\cos E+\bar E_2\bar E_1 ,
$$
from where we find
$$
\CT _0^{3,1}=-\frac{5}{2}\EC +Q_4
$$
with
$$
Q_4=\frac{1}{2\pi}\int _{0}^{2\pi} \bar E_5 dE,
$$
and using the bounds for $\bar E_2$ and $\bar E_1$, we find $|Q_4|\leq 19 \EC ^2$.
\end{proof}

The next step provides an asymptotic formula for $\mathcal{L}_1=2\Re\left\{\EXP ^{i\SO } L_{1}\right\}$.

\begin{lemma}\label{2ndapr}
For $\GO \geq 32$ and $ \EC  \GO \leq 1/8$ we have the following formula for
$L_1$ given in~\eqref{LFourierLq}
\begin{equation}\label{eqn:L1Asymp}
%\begin{array}{rcl}
\textr{\Re\left\{\EXP ^{i\SO }L_1 \right\}=
\Re \left\{\EXP ^{i\SO }\left(L_{1,-1}\,\EXP ^{-i\AO }
+L_{1,-2}\,\EXP ^{-2i\AO }
+L_{1,-3}\,\EXP ^{-3i\AO }+ E_1 \right) \right\}}
%\Re \left\{\EXP ^{i\SO }\left(\bigl(\CT _1^{3,1}N(1,2,1)+ E_{3}\bigl)\,\EXP ^{-i\AO }
%+\bigl(\CT _1^{2,2}N(1,2,0)+ E_{4}\bigl)\,\EXP ^{-2i\AO } \right.\right.\\
%&+&\left.\left.\bigl(\CT _1^{3,3}N(1,3,0)+ \tilde E_{4}\bigl)\,\EXP ^{-3i\AO }+ E_1 \right) \right\}
%\end{array}
\end{equation}
where
\begin{equation}\label{nousL123}
\begin{array}{rcl}
L_{1,-1}&=& \CT _1^{3,1}N(1,2,1)+ E_{3}\\
L_{1,-2}&=& \CT _1^{2,2}N(1,2,0)+ E_{4}\\
L_{1,-3}&=& \CT _1^{3,3}N(1,3,0)+ \tilde E_{4}\, ,
\end{array}
\end{equation}
$N(q,m,n)$ are defined by formula \eqref{Nqmn} and
\begin{align*}
\left| E_1(\AO ,\GO ;\EC)\right|&\leq 2^{18} \EXP ^{-\GO ^3\!/3}
 \EC \biggl[  \GO ^{-3/2} +\EC ^{2}\GO ^{7/2}\biggr]\\
%2^{10} \EXP ^2 \EXP ^{-\GO ^3\!/3}\biggl[\GO ^{-7/2} + \EC ^{2}\GO ^{5/2}+\EC  \GO ^{-3/2} \biggl]\\
\left| E_{3}(\AO ,\GO ;\EC)\right|&\leq  2^{20} \EXP ^{-\GO ^3\!/3}\GO ^{-3/2}\\
\left| E_{4}(\AO ,\GO ;\EC)\right|&\leq  2^{18} \EXP ^{-\GO ^3\!/3}\EC\,  \GO ^{1/2} \\
\left| \tilde E_{4}(\AO ,\GO ;\EC)\right|&\leq  2^{20} \EXP ^{-\GO ^3\!/3}\EC^2\,  \GO ^{3/2}
\end{align*}
\end{lemma}
\begin{proof}
From equation \eqref{LFourierLq}, we have that
\begin{align*}
L_1&= L_{1,0}+\sum_{k\geq 1}\bigl(L_{1,k}\EXP ^{ik\AO }+L_{1,-k}\EXP ^{-ik\AO }\bigl)\\
&=L_{1,-1}\EXP ^{-i\AO }+L_{1,-2}\EXP ^{-2i\AO } +L_{1,-3}\EXP ^{-3i\AO } +\sum_{k\geq 0}L_{1,k}\EXP ^{ik\AO }+
\sum_{k\geq 4}L_{1,-k}\EXP ^{-ik\AO }
\end{align*}
Now, setting
\begin{equation}\label{L61}
 E_{1}=\sum_{k\geq 0}L_{1,k}\EXP ^{ik\AO }+\sum_{k\geq 4}L_{1,-k}\EXP ^{-ik\AO }
\end{equation}
we can write
\begin{equation}\label{L62}
\Re\Bigl\{L_1\EXP ^{i \SO }\Bigl\}=
\Re\Bigl\{ \bigl(L_{1,-1}\EXP ^{-i\AO }+L_{1,-2}\EXP ^{-2i\AO }
+L_{1,-3}\EXP ^{-3i\AO }
+ E_{1}\bigl)\EXP ^{i\SO }\Bigl\}.
\end{equation}
By definitions \eqref{FourierLqkN} we have
\begin{equation}\label{FourierL1m}
\begin{aligned}
 L_{1,-1}&=\CT _{1}^{3,1}N(1,2,1)+\sum_{l\geq 3}\CT _{1}^{2l-1,1}N(1,l,l-1)\\%\label{FourierL1m:m1}\\
 L_{1,-2}&=\CT _{1}^{2,2}N(1,2,0)+\sum_{l\geq 3}\CT _{1}^{2l-2,2}N(1,l,l-2)\\%\label{FourierL1m:m2}.
L_{1,-3}&=\CT _{1}^{3,3}N(1,3,0)+\sum_{l\geq 4}\CT _{1}^{2l-3,3}N(1,l,l-3)%\label{FourierL1m:m2}.
 \end{aligned}
\end{equation}
If we now set
\begin{equation}\label{Es}
\begin{split}
 E_{3}&= \sum_{l\geq 3}\CT _{1}^{2l-1,1}N(1,l,l-1)\\
 E_{4}&= \sum_{l\geq 3}\CT _{1}^{2l-2,2}N(1,l,l-2)\\
\tilde E_{4}&= \sum_{l\geq 4}\CT _{1}^{2l-3,3}N(1,l,l-3)
 \end{split}
\end{equation}
we obtain just~\eqref{eqn:L1Asymp} from equations \eqref{FourierL1m} and \eqref{L62}.
%\begin{equation}\label{L63}
%\Re\{\EXP ^{i\SO }L_1 \}=\Re
%\bigl\{ [(\CT _1^{3,1}N(1,2,1)+ E_{3})\EXP ^{-i\AO }+(\CT _1^{2,2}N(1,2,0)+ E_{4})\EXP ^{-2i\AO }+ E_1]\EXP ^{i\SO } \bigl\}
%\end{equation}
Once we have obtained the formula~\eqref{eqn:L1Asymp}, it only remains to bound properly the errors
$E_1$, $E_{3}$, $E_{4}$ and $\tilde E_{4}$.
From equation \eqref{L61}, the triangle inequality and Proposition~\ref{boundLqk} we have,
using also that $\dfrac{\EC }{\sqrt{1-\EC ^2}}\leq 1$:
%fites noves
% |L_{1,\phantom{\pm}0}|&\leq 2^{10}  \EXP ^{2} \EC   \GO ^{-3/2}\EXP ^{-\GO ^3\!/3} \\
% |L_{1,\phantom{\pm}1}|&\leq 2^7  \EXP  (1+\EC )^4 \GO ^{-7/2} \EXP ^{-\GO ^3\!/3}\\
% |L_{1,- 1}|&\leq  2^{10}  \EXP ^{2}   \GO ^{-1/2} \EXP ^{-\GO ^3\!/3}\\
% |L_{1,\phantom{\pm} k}|&\leq  2^5 2^{k} \EXP  (1+\EC )^{k} \GO ^{-2k-1/2}  \EXP ^{-\GO ^3\!/3}\\
% |L_{1,-k}|&\leq 2^6 2^{2k} \EXP ^{2}  \EC ^{k-1} \GO ^{k-1/2} \EXP ^{-\GO ^3\!/3}

\begin{align}
 | E_1|&\leq |L_{1,0}|+|L_{1,1}|+\sum_{k\geq 2}|L_{1,k}|+\sum_{k\geq 4}|L_{1,-k}| \notag\\
&\leq \EXP ^{2} \EXP ^{-\GO ^3\!/3}
\biggl[2^{10} \EC  \GO ^{-3/2}+2^{11} \EC \GO ^{-7/2}+
2^5 \EC  \sum_{k\geq 2} 2^{2k} \GO ^{-2k-1/2}\notag\\
&\hphantom{\leq} +2^6  \sum_{k\geq 4} 2^{2k}\EC ^{k-1}\GO ^{k-1/2}\biggl]\notag\\
&\leq \EXP ^{2} \EXP ^{-\GO ^3\!/3}
\biggl[2^{10} \EC   \GO ^{-3/2}+
2^{11} \EC \GO ^{-7/2}+
2^{10} \EC \GO ^{-9/2}
%\notag\\&
+2^{14} \EC ^{3}\GO ^{7/2}\biggl]\notag\\
%&2^{10} \EXP ^{2} \leq \EXP ^{-\GO ^3\!/3}
%\biggl[\EC  \GO ^{-3/2} +2^{2} \EXP ^{-2}\GO ^{-7/2}+ 2 \GO ^{-9/2}\sum_{k=2}^{\infty}(2(1+\EC )\GO ^{-2})^k\notag\\
%& + 2^2\EXP ^{\sqrt{1-\EC ^2}}\EC ^{-1}\GO ^{-1/2} \sum_{k\geq 3} (2^{2}\EC \GO )^k\biggl]\notag\\
%&\leq K\EXP ^{-\GO ^3\!/3}\EXP ^{c^2}\biggl[2^6 \EC  \EXP ^{\sqrt{1-\EC ^2}}\GO ^{-3/2} +2^3(1+\EC )^4 \GO ^{-7/2}+ 2^3(1+\EC )^3\GO ^%{-9/2}\notag\\
%& + 2^9\EXP ^{\sqrt{1-\EC ^2}}\EC ^{2}\GO ^{5/2} \biggl]\notag\\
&\leq  \EXP ^{-\GO ^3\!/3}
\biggl[2^{18} \EC  \GO ^{-3/2} +2^{18} \EC ^{3}\GO ^{7/2}\biggr]\notag\\
%&\leq 2^{14}\EXP ^{-\GO ^3\!/3}\biggl[\GO ^{-7/2} + 2^{10} \EXP ^{\sqrt{1-\EC ^2}}(\EC ^{2}\GO ^{5/2}+\EC \GO ^{-3/2}) \biggl]\notag\\
%&\leq 2^{17} \leq \EXP ^{-\GO ^3\!/3}\biggl[(1+\EC )^4 \GO ^{-7/2} + 2^{6} \EXP ^{\sqrt{1-\EC ^2}}(\EC ^{2}\GO ^{5/2}+\EC \GO ^{-3/2}) \biggl]
&\leq 2^{18} \EC \EXP ^{-\GO ^3\!/3}
\biggl[  \GO ^{-3/2}+\EC^2 \GO ^{7/2}\biggr] . \label{bE2}
\end{align}

We now proceed with $E_{3}$, $E_{4}$ and $\tilde E_{4}$.
By Propositions \ref{boundcqmn} and \ref{boundN}, from equations \eqref{Es}

\begin{align*}
 | E_{3}|&\leq\sum_{l\geq 3}|\CT _{1}^{2l-1,1}N(1,l,l-1)|
 \leq 2^3 \EXP ^{\sqrt{1-\EC ^2}} \, \EXP \, \EXP ^{-\GO ^3\!/3}\GO ^{3/2}\sum_{l\geq 3}(2^4 \GO ^{-1})^l
 \leq  2^{16}\EXP ^{2}\EXP ^{-\GO ^3\!/3}\GO ^{-3/2} ,\\
 | E_{4}|&\leq\sum_{l\geq 3}|\CT _{1}^{2l-2,2}N(1,l,l-2)|
 \leq  2 \EC  \EXP ^{\sqrt{1-\EC ^2}}\,  \EXP \, \EXP ^{-\GO ^3\!/3}\GO ^{7/2}\sum_{l\geq 3}(2^4 \GO ^{-1})^l
  \leq  2^{14}\EXP ^{2} \EXP ^{-\GO ^3\!/3}\EC  \GO ^{1/2} \\
| \tilde E_{4}|&\leq\sum_{l\geq 4}|\CT _{1}^{2l-3,3}N(1,l,l-3)|
 \leq  2^{-1} \EC ^2 \EXP ^{\sqrt{1-\EC ^2}}\,  \EXP \, \EXP ^{-\GO ^3\!/3}\GO ^{11/2}\sum_{l\geq 4}(2^4 \GO ^{-1})^l
  \leq  2^{16}\EXP ^{2} \EXP ^{-\GO ^3\!/3}\EC ^2  \GO ^{3/2} .
\end{align*}
The two estimates above, together with estimate~\eqref{bE2} provide the desired bounds for the errors
of equation~\eqref{eqn:L1Asymp}.
\end{proof}

Putting together Lemmas~\ref{boundSumLqk} and~\ref{2ndapr} we already have
\begin{equation}\label{oldlemma8}
\begin{array}{rcl}
\mathcal{L}&=&L_0+2\Re \left\{
\left[ L_{1,-1} \EXP ^{-i\AO }
+L_{1,-2}\EXP ^{-2i\AO }
+ L_{1,-3}\EXP ^{-3i\AO }+ E_1\right]\EXP ^{i\SO } \right\}  +\mathcal{L}_{\geq 2}
\end{array}
\end{equation}
with $L_{1,-1}$, $L_{1,-2}$ and $L_{1,-3}$ as given in \eqref{nousL123}
and
\begin{equation}\label{E1234}
\begin{split}
\left| E_1(\AO ,\GO ;\EC)\right|&\leq 2^{18} \EXP ^{-\GO ^3\!/3}
 \EC \biggl[  \GO ^{-3/2} +\EC ^{2}\GO ^{7/2}\biggr]\\
%2^{10} \EXP ^2 \EXP ^{-\GO ^3\!/3}\biggl[\GO ^{-7/2} + \EC ^{2}\GO ^{5/2}+\EC  \GO ^{-3/2} \biggl]\\
\left| E_{3}(\AO ,\GO ;\EC)\right|&\leq  2^{20} \EXP ^{-\GO ^3\!/3}\GO ^{-3/2}\\
\left| E_{4}(\AO ,\GO ;\EC)\right|&\leq  2^{18} \EXP ^{-\GO ^3\!/3}\EC\,  \GO ^{1/2} \\
\left| \tilde E_{4}(\AO ,\GO ;\EC)\right|&\leq  2^{20} \EXP ^{-\GO ^3\!/3}\EC^2\,  \GO ^{3/2}\\
|\mathcal{L}_{\geq 2}(\AO ,\GO ,\SO;\EC)|&\leq  2^{28}\GO ^{3/2}\,\EXP ^{-\GO ^3\frac{4}{9}} .
\end{split}
\end{equation}
We now compute $ N(1,2,1)$,  $ N(1,2,0)$  and \textr{$ N(1,3,0)$}.
\begin{lemma}\label{N1}
 Let $N(q,m,n)$ be defined by equations \eqref{Nqmn}. Then
\begin{align*}
 N(1,2,1)&=\frac{1}{4}\sqrt{\frac{\pi}{2}}\GO ^{-1/2}\EXP ^{-\GO ^3\!/3}+{^1}E_{TT}\\
 N(1,2,0)&=\sqrt{\frac{\pi}{2}}\GO ^{3/2}\EXP ^{-\GO ^3\!/3}+{^2}E_{TT}\\
\textr{N(1,3,0)}&=\frac{1}{3}\sqrt{\frac{\pi}{2}}\GO ^{5/2}\EXP ^{-\GO ^3\!/3}+{^3}E_{TT}
 \end{align*}
where
$$
|{^1}E_{TT}|\leq 2^{6}\, 9\, \GO ^{-2}\EXP ^{-\GO ^3\!/3},\qquad |{^2}E_{TT}|\leq 2^{5}\, 9\, \EXP ^{-\GO ^3\!/3},
\qquad |{^3}E_{TT}|\leq 2^{6}\, 9\,G\, \EXP ^{-\GO ^3\!/3}.
$$
\end{lemma}

\begin{proof}
 From Proposition \ref{lema1refinado2N} we have
\begin{multline}\label{L101}
 N(1,2,1)=\frac{d_{2,1}}{\GO ^5}\EXP ^{-\GO ^3\!/3}\Biggl[d_4^{2,1}\sqrt{\pi}
 \Bigl(\frac{2}{\GO }\Bigl)^{3/2}\\
 -2^2d_2^{2,1}\sqrt{\pi}\sqrt{\frac{\GO ^3}{2}}+\frac{2^3}{3}d_0^{2,1}\sqrt{\pi}\left(\sqrt{\frac{\GO ^3}{2}}\right)^3+T^1_{2,1}+R^1_{2,1}\Biggl]
\end{multline}
where
\begin{equation*}
 |T^1_{2,1}|\leq 45\, 2^{6}\GO ^{-3},\qquad |R^1_{2,1}|\leq 18 \, \GO ^3,
\end{equation*}
\begin{multline}\label{L102}
 N(1,2,0)=\frac{d_{2,0}}{\GO ^3}\EXP ^{-\GO ^3\!/3}\Biggl[2d_4^{2,0}\sqrt{\pi}\left(\sqrt{\frac{\GO ^3}{2}}\right)^{-1}\\
 -2^2d_2^{2,0}\sqrt{\pi}\sqrt{\frac{\GO ^3}{2}}+\frac{2^3}{3}d_0^{2,0}\sqrt{\pi}\left(\sqrt{\frac{\GO ^3}{2}}\right)^3+T^1_{2,0}+R^1_{2,0}\Biggl]
\end{multline}
where
\begin{equation*}
 |T^1_{2,0}|\leq 45 \, 2^{6} \GO ^{-3}\qquad |R^1_{2,0}|\leq 18  \, \GO ^3,
\end{equation*}
and
\begin{multline}\label{L102nou}
\textr{N(1,3,0)}=\frac{d_{3,0}}{\GO ^5}\EXP ^{-\GO ^3\!/3}\Biggl[2d_6^{3,0}\sqrt{2\pi}\GO ^{-3/2}\\
 -2d_4^{3,0}\sqrt{2\pi}\GO ^{3/2}+ \frac{2}{3}\sqrt{2\pi}d_2^{3,0}\GO^{9/2}-
 \frac{2}{15}d_0^{3,0}\sqrt{2\pi}\GO ^{15/2}+T^1_{3,0}+R^1_{3,0}\Biggl]
\end{multline}
where
\begin{equation*}
 |T^1_{3,0}|\leq 45 \, 2^{8} \GO ^{-3}\qquad |R^1_{3,0}|\leq 18  \, \GO ^6.
\end{equation*}

Taking the dominant terms in  \eqref{L101}, \eqref{L102} and \eqref{L102nou}we get:
\begin{equation}\label{L103}
 N(1,2,1)=d_{2,1}d_0^{2,1}\frac{2\sqrt{2}}{3}\sqrt{\pi}\GO ^{-1/2}\EXP ^{-\GO ^3\!/3}+{^1}E+{^1}E_{TR}
\end{equation}
where
$$
{^1}E=2^{\frac{3}{2}}d_{2,1}\sqrt{\pi}\bigl(d_4^{2,1}\GO ^{-13/2}-d_2^{2,1}\GO ^{-7/2}\bigl)\EXP ^{-\GO ^3\!/3},\qquad
{^1}E_{TR}=(T^1_{2,1}+R^1_{2,1})d_{2,1}\GO ^{-5}\EXP ^{-\GO ^3\!/3},
$$
\begin{equation}\label{L104}
 N(1,2,0)=d_{2,0}d_0^{2,0}\frac{2\sqrt{2}}{3}\sqrt{\pi}\GO ^{3/2}\EXP ^{-\GO ^3\!/3}+{^2}E+{^2}E_{TR}
\end{equation}
where
$$
{^2}E=2^{\frac{3}{2}}d_{2,0}\sqrt{\pi}\bigl(d_4^{2,0}\GO ^{-9/2}-d_2^{2,0}\GO ^{-3/2}\bigl)\EXP ^{-\GO ^3\!/3}\qquad
{^2}E_{TR}=(T^1_{2,0}+R^1_{2,0})d_{2,0}\GO ^{-3}\EXP ^{-\GO ^3\!/3},
$$
and
\begin{equation}\label{L104nou}
 N(1,3,0)=-d_{3,0}d_0^{3,0}\frac{2}{15}\sqrt{2\pi}\GO ^{5/2}\EXP ^{-\GO ^3\!/3}+{^3}E+{^3}E_{TR}
\end{equation}
where
$$
{^3}E=2d_{3,0}\sqrt{2\pi}\bigl(d_6^{3,0}\GO ^{-13/2}-d_4^{3,0}\GO ^{-7/2}+
\frac{d_2^{3,0}}{3}\GO ^{-1/2}
\bigl)\EXP ^{-\GO ^3\!/3},\quad
{^3}E_{TR}=(T^1_{3,0}+R^1_{3,0})d_{3,0}\GO ^{-5}\EXP ^{-\GO ^3\!/3}.
$$

From the bounds given in Lemma \ref{constd} for $d_j^{m,n}$ and the bounds in Lemma \ref{binombound} for $d_{m,n}$ we get:
\begin{equation*}
\begin{array}{rcl}
 |{^1}E| &\leq &2^{\frac{3}{2}}|d_{2,1}|\sqrt{\pi}(|d_4^{1,2}|+|d_2^{2,1}|)\GO ^{-7/2}\EXP ^{-\GO ^3\!/3} \leq
 2^7 \, 9\,  \GO ^{-7/2}\EXP ^{-\GO ^3\!/3}\\
|{^1}E_{TR}| &\leq & |d_{2,1}| \, 36 \, \GO ^{-2}\EXP ^{-\GO ^3\!/3} \leq 2^5 9\, \GO ^{-2}\EXP ^{-\GO ^3\!/3},
\end{array}
\end{equation*}
\begin{equation*}
\begin{array}{rcl}
|{^2}E|&\leq & 2^{\frac{3}{2}}|d_{2,0}|\sqrt{\pi}(|d_4^{2,0}|+|d_2^{2,0}|)\GO ^{-3/2}\EXP ^{-\GO ^3\!/3}
\leq   2^6 9\,  \GO ^{-3/2}\EXP ^{-\GO ^3\!/3} \\
|{^2}E_{TR}| &\leq & |d_{2,0}| \, 36\EXP ^{-\GO ^3\!/3} \leq 2^4 9\, \EXP ^{-\GO ^3\!/3},
\end{array}
\end{equation*}
and
\begin{equation*}
\begin{array}{rcl}
|{^3}E|&\leq & 2|d_{3,0}|\sqrt{2\pi}(|d_6^{3,0}|+|d_4^{3,0}|+\frac{|d_2^{3,0}|}{3})\GO ^{-1/2}\EXP ^{-\GO ^3\!/3}
\leq   2^8 9\,  \GO ^{-1/2}\EXP ^{-\GO ^3\!/3} \\
|{^3}E_{TR}| &\leq & |d_{3,0}| \, 36\,\GO \EXP ^{-\GO ^3\!/3} \leq 2^5 9\,\GO \EXP ^{-\GO ^3\!/3}.
\end{array}
\end{equation*}

Using Lemma \ref{constd}, $d_0^{m,n}=1/(2i)^{2n+1}$ and the definition \eqref{dmn} for $d_{m,n}$ we have that
\begin{align*}
 d_{2,1}d_0^{2,1}&=-i2^3\binom{-1/2}{2}\binom{-1/2}{1}\Bigl(\frac{i}{2^3}\Bigl)=-\frac{3}{2^4}\\
 d_{2,0}d_0^{2,0}&=i2^2\binom{-1/2}{2}\Bigl(-\frac{i}{2}\Bigl)=\frac{3}{2^2}\\
d_{3,0}d_0^{3,0}&=i2^3\binom{-1/2}{3}\Bigl(-\frac{i}{2}\Bigl)=-\frac{5}{2^2}.
\end{align*}
We can then write equation \eqref{L103} as
\begin{equation*}%\label{N121}
 N(1,2,1)=\frac{1}{4}\sqrt{\frac{\pi}{2}}\GO ^{-1/2}\EXP ^{-\GO ^3\!/3}+{^1}E_{TT}
\end{equation*}
where
$$
{^1}E_{TT}={^1}E+{^1}E_{TR}
$$
satisfies
$$
|{^1}E_{TT}|\leq 2^7 \, 9\,  \GO ^{-\frac{7}{2}}\EXP ^{-\GO ^3\!/3}+ 2^5 9\, \GO ^{-2}\EXP ^{-\GO ^3\!/3} \leq 2^6 9\, \GO ^{-2}\EXP ^{-\GO ^3\!/3}.
$$

In an analogous way, equation \eqref{L104} can be written as
\begin{equation*}%\label{N120}
 N(1,2,0)=\sqrt{\frac{\pi}{2}}\GO ^{3/2}\EXP ^{-\GO ^3\!/3}+{^2}E_{TT}
\end{equation*}
where
$$
{^2}E_{TT}={^2}E+{^2}E_{TR}
$$
satisfies
$$
 |{^2}E_{TT}|\leq 2^6 9\,  \GO ^{-\frac{3}{2}}\EXP ^{-\GO ^3\!/3}+2^4 9\, \EXP ^{-\GO ^3\!/3}\leq 2^5 9\, \EXP ^{-\GO ^3\!/3}.
 $$
Finally,  equation \eqref{L104nou} can be written as
\begin{equation*}%\label{N120}
 N(1,3,0)=\frac{1}{3}\sqrt{\frac{\pi}{2}}\GO ^{5/2}\EXP ^{-\GO ^3\!/3}+{^3}E_{TT}
\end{equation*}
where
$$
{^3}E_{TT}={^3}E+{^3}E_{TR}
$$
satisfies
$$
 |{^3}E_{TT}|\leq 2^8 9\,  \GO ^{-\frac{1}{2}}\EXP ^{-\GO ^3\!/3}+2^5 9\,\GO \EXP ^{-\GO ^3\!/3}\leq 2^6 9\,\GO \EXP ^{-\GO ^3\!/3},
 $$
and this proves the Lemma.
\end{proof}
Using the approximations given in Lemma \ref{N1} we have from Lemmas \ref{boundSumLqk} and \ref{2ndapr}:

\begin{lemma}\label{4thapr}
For $\GO \geq 32$ and $ \EC  \GO \leq 1/8$, the Melnikov potential $\mathcal{L}$ given in~\eqref{LFourierLq} satisfies
\begin{equation}\label{potecialfinal}
\textr{\mathcal{L}=L_0+
2L_{1,-1}\cos (\SO -\AO )+2L_{1,-2}\cos (\SO -2\AO )+2L_{1,-3}\cos (\SO -3\AO )
+2\Re\{E_1\EXP ^{i\SO }\}+\mathcal{L}_{\geq 2}}
\end{equation}
with
\[
 \begin{array}{rcl}
\textr{2\,L_{1,-1}}&=&  \CT _1^{3,1}\sqrt{\frac{\pi}{8}}\GO ^{-1/2}\EXP ^{-\GO ^3\!/3}+E_3+E_5\\
\textr{2\,L_{1,-2}}&=& \CT _1^{2,2}\sqrt{2\pi}\GO ^{3/2}\EXP ^{-\GO ^3\!/3}+E_4+E_6 \\
\textr{2\,L_{1,-3}}&=&\CT _1^{3,3}\frac{\sqrt{2\pi}}{3}\GO ^{5/2}\EXP ^{-\GO ^3\!/3}+\tilde E_4+\tilde E_6
\end{array}
\]
and where $\mathcal{L}_{\geq 2}$ and $E_k$ with $k=1,3,4$ are given in equations \eqref{E1234} and
\[
 |E_5|\leq 2^{13}  \, 9 \,   \GO ^{-2}\EXP ^{-\GO ^3\!/3},\qquad
 |E_6|\leq 2^{11}  \, 9 \, \EC   \EXP ^{-\GO ^3\!/3} ,\qquad
 |\tilde E_6|\leq 2^{13}  \, 9 \, \GO  \EC ^2  \EXP ^{-\GO ^3\!/3}.
\]
\end{lemma}
\begin{proof}
 By Lemma \ref{N1} we have that $N(1,2,1)$, $N(1,2,0)$ and $N(1,3,0)$ are real and then coincide with their real part.
Equation \eqref{oldlemma8} gives the correct estimation of $\mathcal{L}$. To complete the proof is enough to take
$$
E_5 =\CT _1^{3,1}\cdot {^1}E_{TT}\, ,\quad E_6 =\CT _1^{2,2}\cdot {^2}E_{TT}\quad \text{and}\quad \tilde E_6 =\CT _1^{3,3}\cdot {^3}E_{TT}
$$
where ${^1}E_{TT}$, ${^2}E_{TT}$ and ${^3}E_{TT}$ are given in Lemma \ref{N1}.
Therefore by Proposition \ref{boundcqmn} we find directly the bounds of $E_5$, $E_6$ and $\tilde E_6$.

\end{proof}
The next Proposition contains the final asymptotic estimate for $\mathcal{L}_1$:

\begin{proposition}\label{5thapr}
\textr{For $\GO \geq 32$ and $ \EC  \GO \leq 1/8$, the Melnikov potential $\mathcal{L}$~\eqref{LFourierLq} is given by}
\eqref{potecialfinal} with:
%\begin{multline*}
%\mathcal{L}=L_0+\cos (\SO -\AO )\left(\sqrt{\frac{\pi}{8}}\GO ^{-1/2}\EXP ^{-\GO ^3\!/3}+E_3+E_5+E_7\right)
%-\cos (\SO -2\AO )\left(3\sqrt{2\pi}\EC \GO ^{3/2}\EXP ^{-\GO ^3\!/3}+E_4+E_6+E_8 \right)\\
%+\cos (\SO -3\AO )\left(\frac{19}{8}\sqrt{2\pi}\EC ^2\GO ^{5/2}\EXP ^{-\GO ^3\!/3}+\tilde E_4+\tilde E_6+\tilde E_8 \right)
%+2\Re\{E_1\EXP ^{i\SO }\}+\mathcal{L}_{\geq 2}
%\end{multline*}
\[
 \begin{array}{rcl}
\textr{2\,L_{1,-1}}&=&  \sqrt{\frac{\pi}{8}}\GO ^{-1/2}\EXP ^{-\GO ^3\!/3}+E_3+E_5+E_7\\
\textr{2\,L_{1,-2}}&=& -3\sqrt{2\pi}\EC \GO ^{3/2}\EXP ^{-\GO ^3\!/3}+E_4+E_6+E_8 \\
\textr{2\,L_{1,-3}}&=&\frac{19}{8}\sqrt{2\pi}\EC ^2\GO ^{5/2}\EXP ^{-\GO ^3\!/3}+\tilde E_4+\tilde E_6+\tilde E_8
\end{array}
\]
and where $\mathcal{L}_{\geq 2}$ and $E_k$ with $k=1,3,\dots 6$ and $\tilde E_k$ with $k=4,5,6$ are given in equations \eqref{E1234} and
\[
|E_7|\leq 98 \EC ^2\GO ^{-1/2}\EXP ^{-\GO ^3\!/3}\, ,\quad
 |E_8|\leq 98 2^2 \EC ^2\GO ^{3/2}\EXP ^{-\GO ^3\!/3}\, , \quad
 |\tilde E_8|\leq 98 2^2 \EC ^3\GO ^{5/2}\EXP ^{-\GO ^3\!/3}.
\]
\end{proposition}
\begin{proof}
 From Lemma \ref{ccalc} we have
\begin{align*}
\CT _1^{3,1}\sqrt{\frac{\pi}{8}}\GO ^{-1/2}\EXP ^{-\GO ^3\!/3}&=\sqrt{\frac{\pi}{8}}\GO ^{-1/2}\EXP ^{-\GO ^3\!/3}+E_7\\
\CT _1^{2,2}\sqrt{2\pi}\GO ^{3/2}\EXP ^{-\GO ^3\!/3}&=-3\sqrt{2\pi}\EC \GO ^{3/2}\EXP ^{-\GO ^3\!/3}+E_8 \\
\CT _1^{3,3}\frac{\sqrt{2\pi}}{3}\GO ^{5/2}\EXP ^{-\GO ^3\!/3}&=\frac{19}{8}\sqrt{2\pi}\EC ^2\GO ^{5/2}\EXP ^{-\GO ^3\!/3}+\tilde E_8
\end{align*}
with
\begin{align*}
E_7&=Q_1\sqrt{\frac{\pi}{8}}\GO ^{-1/2}\EXP ^{-\GO ^3\!/3}\\
E_8&=Q_2\sqrt{2\pi}\GO ^{3/2}\EXP ^{-\GO ^3\!/3}\\
\tilde E_8&=Q_5\sqrt{2\pi}\GO ^{5/2}\EXP ^{-\GO ^3\!/3}
\end{align*}
Therefore by Lemma \ref{4thapr} and the bounds of $Q_1$ and $Q_2$ given in Lemma \ref{ccalc} we conclude the proof.
\end{proof}
\subsection{Asymptotic estimate of $\mathcal{L}_0$}
\label{MP0}

It only remains to estimate the Fourier coefficient  $L_0=\mathcal{L}_0$ defined in~\eqref{eqn:LFourier3} or~\eqref{LFourierLq}.
\begin{lemma}\label{N0}
Let $N(q,m,n)$ be defined by equations \eqref{Nqmn}. Then for $m,n\in\mathbb{N}$, $m+n>0$,
$$
|N(0,m,n)|\leq   2^{m+n+2}\GO ^{-2m-2n+1}.
$$
\end{lemma}
\begin{proof}
 Since $\tau\in\mathbb{R}$ in the integral \eqref{Nqmn}, it is clear that
$$
\frac{1}{|\tau+i|},\frac{1}{|\tau-i|}\leq 1
$$
and then
$$
\frac{1}{|\tau+i|^{2n}}\frac{1}{|\tau-i|^{2m}}\leq\frac{1}{1+\tau ^2}\ .
$$
For $n,m> 0$, using equation \eqref{Nqmn} and Lemma \ref{binombound} to bound $d_{m,n}$, the Lemma follows:
\begin{align*}
|N(0,m,n)|
&\leq2^{m+n}\GO ^{-2m-2n+1}\EXP ^{-1/2}\int_{-\infty}^{\infty}\frac{d\tau}{1+ \tau^2}\\
&=2^{m+n}\GO ^{-2m-2n+1}\EXP ^{-1/2}\pi\leq 2^{m+n+2} \GO ^{-2m-2n+1}.
\end{align*}
\end{proof}
\begin{lemma}\label{L0k}
 Let $k\in\mathbb{N}$ and $L_{0, k}$ defined by equation~\eqref{eqn:LFourier3}. Then
\[
 L_{0,k}=\sum_{l\geq k+1}\CT _{0}^{2l-k,-k}N(0,l-k,l).
\]
\end{lemma}
\begin{proof}
 From equations \eqref{FourierLqkN}, we have just to prove $N(0,0,k)=N(0,k,0)=0$ for $k\geq 2$.
By equations \eqref{Nqmn} this reduces to show that
$$
\int_{-\infty}^{\infty}\frac{d\tau}{(\tau\pm i)^{2k}}=0
$$
where the positive sign in the denominator correspond to $I(0,0,k)$ and the negative to $I(0,k,0)$.
Since the variable $\tau\in\mathbb{R}$ this integral is trivially zero
$$
\int_{-\infty}^{\infty}\frac{d\tau}{(\tau\pm i)^{2k}}=-\frac{1}{2k-1}\frac{1}{(\tau\pm i)^{2k-1}}\biggl|_{-\infty}^{\infty}=0.
$$
\end{proof}
\begin{lemma}\label{boundL0k}
Let $L_{0,k}$ be defined by equations \eqref{FourierLqkN}
for $k\geq 0$. If $\GO \geq 32$,
$$
|L_{0,k}| \leq  2^{2k+8}\EC ^k  \GO ^{-2k-3}.
$$
\end{lemma}
\begin{proof}
From Lemma \ref{L0k} we have
\[
|L_{0,k}| \leq \sum_{l\geq k+1}|\CT _{0}^{2l-k,- k}||N(0,l-k,l)| ,
\]
and by Propositions \ref{boundcqmn} and \ref{boundN},
\[
|L_{0, \pm k}| \leq  2^{-2k+3}\EC ^k\GO ^{2k+1}\sum_{l\geq k+1}(2^4\GO ^{-4})^l
\leq  \EC ^k 2^{2k+8} \GO ^{-2k-3}.
\]
\end{proof}

\begin{lemma}\label{L0_1apr}
Let $L_0=\mathcal{L}_0$ be defined by equations~\eqref{eqn:LFourier3} or~\eqref{LFourierLq}. Then
for $\GO \geq 32$
\begin{align*}
L_0&=L_{0,0}+(\CT _0^{3,1}\frac{3}{4}\pi \GO ^{-5}+F_2)\cos(\AO )+F_3\\
L_{0,0}&=\CT _0^{2,0}\frac{\pi}{2}\GO ^{-3}+F_1
\end{align*}
where
\[
 |F_1|\leq 2^{12} \GO ^{-7},\qquad
 |F_2|\leq  2^{15}\EC  \GO ^{-9},\qquad
 |F_3|\leq 2^{14} \EC ^2 \GO ^{-7}.
\]
\end{lemma}

\begin{proof}
From Proposition \ref{RLFourierLq} we know that
$$
L_0=L_{0,0}+2\sum_{k\geq 1} L_{0,k}\cos{k\AO },
$$
and from Lemma \ref{L0k} we have that
\begin{equation}\label{L0s}
\begin{split}
 L_{0,0}&=\CT _{0}^{ 2,0}N(0,1,1)+\sum_{l\geq 2}\CT _{0}^{2l,0}N(0,l,l) \\
 L_{0,1}&=\CT _{0}^{3,-1}N(0,1,2)+\sum_{l\geq 3}\CT _{0}^{2l-1,-1}N(0,l-1,l)\\
 L_{0,k}&=\sum_{l\geq k+1}\CT _{0}^{2l-k,-k}N(0,l-k,l)\qquad\text{for $k\geq2$}.
\end{split}
\end{equation}
Introducing
\[
 F_1=\sum_{l\geq 2}\CT _{0}^{2l,0}N(0,l,l),\quad
 F_2=2 \sum_{l\geq 3}\CT _{0}^{2l-1,-1}N(0,l-1,l),\quad
 F_3=2\sum_{k\geq 2}\cos {k\AO }L_{0,k},% F_3&=2\Re\Bigl\{\sum_{k\geq2}\EXP ^{ik\AO }\sum_{l\geq k+1}\CT _{0}^{2l-k,-k}N(0,l-k,l)\Bigl\}
\]
and using $\GO \geq 32$ in Lemmas \ref{N0}, \ref{boundL0k} and Proposition \ref{boundcqmn}, we have
\begin{subequations}\label{Fs}
\begin{align*}
 |F_1|&\leq  2^3 \GO   \sum_{l\geq 2} (2^4\GO ^{-4})^l \leq  2^{12}  \GO ^{-7}\\
 |F_2|&\leq 2^2 \EC  \GO ^3 \sum_{l\geq 3}(2^4\GO ^{-4})^l\leq  2^{15}\EC  \GO ^{-9}\\
|F_3|&\leq2\sum_{k\geq2}|L_{0,k}|%\notag\\% |F_3|&\leq2\sum_{k\geq2}\sum_{l\geq k+1}\EC ^{k}K  2^{2l-k}\GO ^{-4l+2k+1}
\leq  2^{14}\EC ^2 \GO ^{-7} .
\end{align*}
\end{subequations}
 From definition \eqref{Nqmn} we have now that
\begin{align*}
 N(0,1,1)&=\phantom{-}\frac{2^2}{\GO ^3}\binom{-1/2}{1}\binom{-1/2}{1}\int_{-\infty}^{\infty}\frac{d\tau}{(\tau^2+1)^2}=2^2\Bigl(-\frac{1}{2}\Bigl)\Bigl(-\frac{1}{2}\Bigl)\GO ^{-3}=\frac{\pi}{2}\GO ^{-3} ,\\
 N(0,1,2)&=\frac{2^3}{\GO ^5}\binom{-1/2}{1}\binom{-1/2}{2}\int_{-\infty}^{\infty}\frac{d\tau}{(\tau-i)(\tau+i)^2}=2^3\Bigl(-\frac{1}{2}\Bigl)\Bigl(\frac{3}{2^3}\Bigl)\Bigl(-\frac{\pi}{4}\Bigl)\GO ^{-5}=\frac{3}{8}\pi \GO ^{-5} .
\end{align*}
From these equations, substituting equations \eqref{L0s} in the definition of $L_0$ and the bounds given in equations \eqref{Fs}
we have proven this Lemma.
\end{proof}
A refinement of this Lemma is
\begin{lemma}%\label{L0_2apr}
Let $L_0=\mathcal{L}_0$ be defined by equations~\eqref{eqn:LFourier3} or~\eqref{LFourierLq}.
Then for $\GO \geq 2^{3\!/2}$
\begin{align*}
L_0&=L_{0,0} +(-\frac{15}{8}\pi \EC  \GO ^{-5}+F_2+F_5)\cos(\AO )+F_3\\
L_{0,0}&=\frac{\pi}{2}\GO ^{-3}+F_1+F_4
\end{align*}
where $F_1$, $F_2$ and $F_3$ are given in Lemma \ref{L0_1apr} and
\[
|F_4|\leq 2 \,  98  \, \GO ^{-3}\EC ^2, \qquad
|F_5|\leq 2^2 \,  98 \,   \GO ^{-5}\EC ^2 .
\]
\end{lemma}
\begin{proof}
 In Lemma \ref{ccalc} we have computed the constants $\CT _0^{2,0}$ and $\CT _0^{3,1}$, then by setting
\[
 F_4= \frac{\pi}{2}Q_3 \GO ^{-3},\qquad
 F_5= \frac{3}{4}\pi Q_4 \GO ^{-5}\cos \AO ,
\]
and using the bounds for $Q_3$ and $Q_4$ we find the desired bound for $F_4$ and $F_5$.
\end{proof}

With this Lemma we can rewrite Proposition~\ref{5thapr} exactly as Theorem~\ref{thepropositionmain}, and so it is proven.

%\section{Proofs of Propositions \ref{RLFourierLq}, \ref{boundN}, \ref{lema1refinado2N}}\label{proofs}

\section*{Acknowledgments}
The authors are indebted to Marcel Gu\`ardia, Pau Mart\'{\i}n, Regina Mart\'{\i}nez, Eva Miranda and Carles Sim\'o
for helpful discussions.

\bibliographystyle{alphabbrv}
\bibliography{ERTBP}
\end{document}